\documentclass[12pt]{article}

\usepackage[margin=1in]{geometry}
\usepackage{amsmath,amssymb,amsthm}
\usepackage{mathtools}
\usepackage[authoryear,round]{natbib}
\usepackage{booktabs, float}
\usepackage{tikz}
\usepackage{graphicx}
\usetikzlibrary{shapes.geometric,arrows.meta,positioning,calc,fit,backgrounds}
\usepackage{enumitem}
\usepackage{hyperref}
\usepackage{xcolor}
\usepackage{authblk}
\definecolor{cSurvive}{RGB}{29,158,117}
\definecolor{cDrop}{RGB}{136,135,128}
\definecolor{cDyn}{RGB}{186,117,23}
\definecolor{cPrs}{RGB}{29,158,117}
\definecolor{cHost}{RGB}{136,135,128}
\usepackage{authblk}

\usepackage{placeins}

\usepackage{setspace}
\setlist{topsep=3pt,itemsep=1pt,parsep=0pt}

\newtheorem{proposition}{Proposition}[section]
\newtheorem{remark}{Remark}[section]

\newtheorem{corollary}{Corollary}[section]
\makeatletter
\let\c@remark\c@proposition
\let\c@principle\c@proposition
\let\c@corollary\c@proposition
\makeatother

\newcommand{\bs}[1]{\boldsymbol{#1}}
\newcommand{\bv}{\mathbf{v}}

\newcommand{\bom}{\bs{\omega}}
\newcommand{\bOm}{\bs{\Omega}}
\newcommand{\bk}{\hat{\mathbf{k}}}
\newcommand{\Div}{\nabla\!\cdot\,}
\newcommand{\gradh}{\nabla_{\!h}}
\newcommand{\gradz}{\nabla_{\!z}}
\newcommand{\divh}{\nabla_{\!h}\!\cdot\,}
\newcommand{\divz}{\nabla_{\!z}\!\cdot\,}
\newcommand{\Dt}{\partial_t}
\newcommand{\OO}{\mathcal{O}}

\newcommand{\dd}{\,\mathrm{d}}

\newcommand{\hodge}{\star}

\newcommand{\Fr}{\mathrm{Fr}}
\newcommand{\rhoa}{\rho_{\mathrm{AC}}^{\alpha}}
\hypersetup{colorlinks=true, linkcolor=blue!60!black,
  citecolor=blue!60!black, urlcolor=blue!60!black}

\title{\bfseries
Foundations of Global Ocean Climate Modelling at all Scales
  } 

\author[1,2]{Peter Korn}
\affil[1]{Max Planck Institute for Meteorology, Hamburg, Germany}
\affil[2]{Department of Mathematics, Imperial College London,
  United Kingdom}
\begin{document}
\maketitle

\begin{abstract}
We introduce  a computational method that extends the range of  current global ocean climate models to non-hydrostatic scales, and establishes a convection resolving ocean model.  The main obstruction towards this goal at global scale -the calculation of the non-hydrostatic pressure-  is resolved by computing the pressure locally in a manner that capitalizes on particular conditions of ocean dynamics, and the structure of ocean models. 
 
Our claims  are substantiated 1) by a theoretical analysis, which shows that the proposed method is compatible with ocean physics, 2) by numerical experiments, which evidence by comparison that it represents faithfully and to high accuracy non-hydrostatic dynamics, and 3) by a performance analysis that establishes that our non-hydrostatic 
method exceeds the cost of a hydrostatic model by a fixed factor of about $1.2$ in operation count and $1.3$ in runtime at any resolution.

We show how to use the additional information contained in non-hydrostatic scales can be used to calculate the  diffusivity free from numerical contributions and the mixing efficiency, transforming both from prescribed values by the modeller to calculated values from the dynamics.
\end{abstract}

\tableofcontents

\section{Introduction}\label{sec:intro}
The dynamics of the world's ocean spans a breathtaking range of scales.
The circulation couples processes from thousands of kilometres down to
the millimetre, and from seconds to millennia. Our understanding of
global ocean dynamics is to a large extent based on computational
models, which provide a spatial and temporal coverage of the ocean's
state that observation alone does not.  However, the range of scales a
computational ocean model can represent is limited by its computational
capacity.

Because of this computational limitation the Assessment Reports (AR)  of the 
The Intergovernmental Panel on Climate Change (IPCC) has since its first report AR1 (1990),
to the sixth report AR6 (2020-2023) relied on ocean climate models whose dynamical equation were hydrostatic Boussinesq equations (the ``Primitive Equations'' of large scale ocean dynamics). The dynamical cores of 
all of these models differ in implementation and calibration. What changed most distinctively over the years is 
the spatial resolution, i.e. the balance between resolved and parametrized processes is shifting towards resolved, but the fundamental limitation inherited by the models dynamical foundation remained in place. 
The impact of changing exactly this foundation can be observed for atmospheric models when they crossed the barrier
to convection resolving models.


The purpose of this paper is to enlarge -within today's computational budget- the spectrum of motions that
ocean climate model represent. More precisely, we present, analyse and evaluate a computational method,
named AC/DC, that transforms global non-hydrostatic ocean modelling
from an impossibility into an affordable option.  
Such a claim has to meet two requirements: first, that the model is able to represent non-hydrostatic dynamics, 
second, that it is able to do this by affordable costs.

We show by theoretical analysis and numerical evidence that he new method meets the first requirement, and by
a detailed performance analysis that it also satisfies the second one with a
small overhead over the hydrostatic cost.

\paragraph{Information Gain.}
A hydrostatic ocean model cannot grow
resolved plumes from an unstable stratification, whatever its resolution is, so convection must be
parametrised, by convective adjustment or by a strongly
enhanced vertical diffusivity in the boundary-layer scheme
\citep{MarshallSchott1999}.
 Non-hydrostatic models remove this
structural limitation.  At metre to decametre spacing the convective
plumes are simulated explicitly \citep{JonesMarshall1993,
MarshallSchott1999}.  
The atmospheric community crossed the analogous
threshold when models reached convection-permitting resolution, widely
regarded as a step change in simulated vertical transport
\citep{Prein2015}.  For the compressible atmosphere that step required
increasing resolution alone, because the acoustic mode is already
prognostic and split-explicitly sub-stepped.  For the incompressible
ocean, increasing resolution alone does not take the step, because it
places the three-dimensional elliptic pressure problem at the centre of
every time step; AC/DC removes this obstruction.

The information gain of resolved convection is that dissipation and mixing
change status from prescribed inputs to computable outputs.  In a
convection-resolving simulation the kinetic-energy dissipation rate
$\varepsilon$ and the tracer-variance dissipation rate $\chi$ are
diagnosed from the simulated fields, not implied by a parametrization.
The tracer diffusivity then follows from the Osborn--Cox balance
$K_T = \chi/(2(\partial_z\overline{T})^2)$ \citep{OsbornCox1972} and the
diapycnal diffusivity from $K_\rho = \Gamma\varepsilon/N^2$
\citep{Osborn1980}, so that the mixing efficiency $\Gamma$, canonically
prescribed at $0.2$, and reported in the range $0.15$--$0.20$ for
stratified shear turbulence \citep{CaulfieldPeltier2000,
SalehipourPeltier2015}, becomes a diagnosed function of the simulated
turbulence instead of a constant supplied to the model.  For this
programme to survive the approximation made in this paper, the discrete
tracer-variance budget must close despite the residual divergence the
approximation; that it does, exactly and with the numerical part
of the sink separately computable, is the content of
Proposition~\ref{prop:variance} and Corollary~\ref{cor:osborn_cox}.

The fundamental challenge in moving from hydrostatic to
non-hydrostatic dynamics is not the additional vertical velocity
equation; it is the change in the nature of incompressibility.
Hydrostatic models set the vertical velocity such that the
three-dimensional velocity field is solenoidal, $\Div\bv = 0$.
Non-hydrostatic equations cannot do that, because they already carry an
evolution equation for the vertical velocity; for them $\Div\bv = 0$
becomes a constraint: the velocity field satisfies its evolution
equation under the condition that it stays solenoidal.  This is
commonly enforced by adjusting the pressure, and that creates the
wall.

\paragraph{The Wall, and the way around it.}
To illustrate the origin of the ``computational wall'' we decompose canonically the pressure into 
hydrostatic pressure $p_{\mathrm{hyd}}$,
surface pressure $p_{\mathrm{sfc}}$ and non-hydrostatic
pressure $p_{\mathrm{NH}}$,
\begin{equation*}
  p = p_{\mathrm{hyd}} + p_{\mathrm{sfc}} + p_{\mathrm{NH}},
\end{equation*}
where the surface pressure is calculated from the kinematic surface
boundary condition and the hydrostatic pressure from the fluid's
weight over a certain level. 
The non-hydrostatic pressure is calculated via the pressure
equation, which arises when the divergence is applied to the velocity
equation,
\begin{equation}\label{press_intro}
\Delta  p_{\mathrm{NH}} = \Div\big( \text{advection + dissipation}\ \dots\big),
\end{equation}
supplemented with Dirichlet or Neumann boundary conditions.

The non-hydrostatic pressure at a single point depends on the flow in
the entire ocean and on its whole boundary; this renders the pressure
equation \emph{non-local}, and the three-dimensional elliptic problem
has to be solved at every time step.  Only this guarantees that the
flow remains incompressible, i.e.\ that volume is conserved.  Discretised,
\eqref{press_intro} is an enormous linear system on a domain as complex
as the ocean's, of the size of the mesh---$\mathcal{O}(10^{10})$ cells
and beyond---to be solved on massively parallel architectures and
accelerator technology.  This computational cost is the wall that no
approach has removed so far; it renders global non-hydrostatic ocean
modelling infeasible and limits it to regional or idealised
configurations.  Section~\ref{sec:pressure_calc} provides more
details.


\begin{figure}[tbp]
  \centering
  \begin{tikzpicture}[
    font=\footnotesize,
    >={Stealth[length=1.8mm]},
    ocean/.style={draw,thick},
    pcap/.style={align=center,text width=50mm,font=\scriptsize},
    ptitle/.style={align=center,text width=52mm,font=\footnotesize\bfseries},
  ]
  \begin{scope}[xshift=0mm]
    \fill[black!30] (0,0) rectangle (40mm,-17mm);
    \draw[ocean] (0,0) rectangle (40mm,-17mm);
    \fill[black] (5mm,-8.5mm) circle (1mm);
    \draw[->,black!75] (5mm,-8.5mm) -- (36mm,-8.5mm);
    \draw[->,black!75] (5mm,-8.5mm) -- (5mm,-2mm);
    \draw[->,black!75] (5mm,-8.5mm) -- (5mm,-15mm);
    \node[anchor=north east,font=\scriptsize] at (39mm,-1mm) {within one step};
    \node[ptitle,anchor=north] at (20mm,-19mm) {(a) elliptic projection};
    \node[pcap,anchor=north] at (20mm,-25.5mm)
      {$(L_H{+}L_z)\,p_{\mathrm{NH}}/\rho_0=S$\\[3pt]
       coupling completed at once;\\
       error set by $\Delta t$};
  \end{scope}
  \begin{scope}[xshift=54mm]
    \foreach \i in {0,...,9}{\fill[black!8] ({\i*4mm+0.4mm},0)
        rectangle ({\i*4mm+3.6mm},-17mm);}
    \fill[black!30] (4.4mm,0) rectangle (7.6mm,-17mm);
    \draw[ocean] (0,0) rectangle (40mm,-17mm);
    \foreach \i in {0,...,9}
      \draw[<->,black!70] ({\i*4mm+2mm},-2mm) -- ({\i*4mm+2mm},-15mm);
    \fill[black] (6mm,-8.5mm) circle (1mm);
    \node[ptitle,anchor=north] at (20mm,-19mm) {(b) AC/DC: column solve};
    \node[pcap,anchor=north] at (20mm,-25.5mm)
      {$L_z\,p_V/\rho_0=S$\\[3pt]
       exact per column;\\ no error, no communication};
  \end{scope}
  \begin{scope}[xshift=108mm]
    \begin{scope}
      \clip (0,0) rectangle (40mm,-17mm);
      \fill[black!10] (10mm,-8.5mm) circle (34mm);
      \foreach \r/\op in {14/16,10/23,6/30}{
        \fill[black!\op] (10mm,-8.5mm) circle (\r mm);}
      \foreach \r in {6,10,14}{\draw[black!50] (10mm,-8.5mm) circle (\r mm);}
      \draw[densely dashed,black!65] (10mm,-8.5mm) circle (34mm);
    \end{scope}
    \draw[ocean] (0,0) rectangle (40mm,-17mm);
    \fill[black] (10mm,-8.5mm) circle (1mm);
    \draw[->,thick] (11mm,-8.5mm) -- (23mm,-8.5mm)
      node[midway,above,font=\scriptsize]{$c_{\mathrm{ac}}$};
    \node[anchor=south west,font=\scriptsize] at (0.8mm,-16.2mm)
      {one cell per sub-step};
    \node[anchor=north east,font=\scriptsize] at (39.2mm,-0.8mm)
      {after $L/c_{\mathrm{ac}}$};
    \node[ptitle,anchor=north] at (20mm,-19mm) {(c) AC/DC: horizontal step};
    \node[pcap,anchor=north] at (20mm,-25.5mm)
      {residual $S^{*}=-L_H\,p_V/\rho_0$\\[3pt]
       same coupling, over time;\\ error set by $\alpha$};
  \end{scope}
  \end{tikzpicture}
  \caption{\it \small 
    Projection~(a) couples instantaneously within a time step by a
    global solve, its error set by the pressure--velocity splitting.
    AC/DC propagates the coupling at finite speed $c_{\mathrm{ac}}$:
    the vertical block~(b) is solved exactly and independently
    in each column, while the horizontal residual~(c) advances one
    cell per sub-step, which is the acoustic CFL, and so covers a
    horizontal distance $L$ in time $L/c_{\mathrm{ac}}$, shorter than
    the time $L/U$ the flow needs to cover it by a factor
    $\Fr_{\mathrm{baro}} = U/c_{\mathrm{ac}}$, which is why the error
    it leaves is second order in that factor.  Neither method is
    exact, projection has a splitting error, AC/DC a divergence error
    of order $\delta^2\Fr_{\mathrm{baro}}^2$, but AC/DC needs no
    global communication to get there.
    On a variable-resolution mesh the contrast sharpens: the solve in (a) stays global, its iteration count is set by the finest spacing  in the domain, whereas the disc in (c) need only fill the patch where the non-hydrostatic correction is active (see Section~\ref{sec:telescoping}).}
  \label{fig:acdc_idea}
\end{figure}
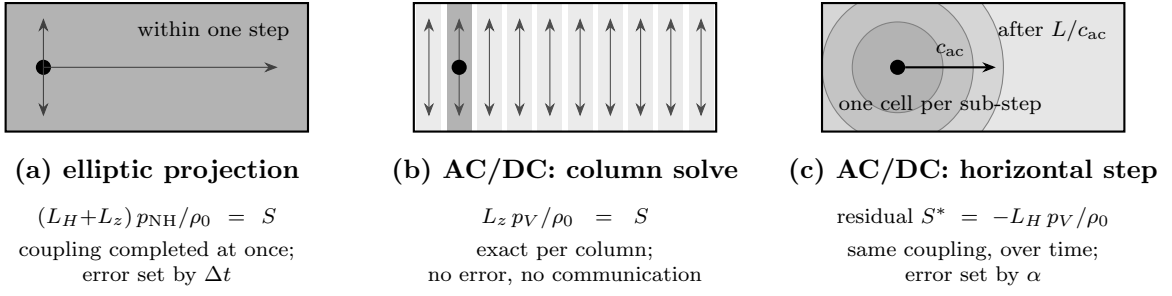

A legitimate way around the computational wall has to a) approximate
non-hydrostatic dynamics, b) make global simulations computationally
feasible, and c) not sacrifice one goal for the other.

The idea is to make the pressure equation local, i.e.\
to replace the non-local constraint $\Div\bv = 0$ by a (local) wave
equation,
\begin{equation}\label{wave_intro}
\Dt   \psi + \Div\bv = 0.
\end{equation}

 The full equation, restored in
Section~\ref{sec:ac_system}, carries a constant $\alpha$ with the units of a
pressure, which~\eqref{wave_intro} suppresses.  With it, $\psi$ is the
pressure of a slightly compressible fluid whose density is
\begin{equation}
  \rhoa = \rho_0\Big(1 + \frac{\psi}{\alpha}\Big),
  \qquad\text{equivalently}\qquad
  \psi = c_{\mathrm{ac}}^{2}\,(\rhoa - \rho_0),
  \qquad
  c_{\mathrm{ac}}^{2} := \frac{\alpha}{\rho_0},
\end{equation}
which is a linear barotropic equation of state.  The pseudo-pressure and
the pseudo-density are therefore not two quantities but one field written
twice, and $\alpha$ is its stiffness.

What this changes is how the pressure enforces the constraint, not the
constraint it enforces.  In an incompressible fluid the pressure is
whatever field, everywhere at once, makes the velocity divergence-free, and
that simultaneity \emph{is} the global solve.  Here it is replaced by a
spring: a converging velocity compresses the fluid, the compression raises
$\psi$ through the equation of state, and the resulting $\nabla\psi$ pushes
back until the convergence is gone.  The adjustment is no longer
instantaneous but travels at the finite speed $c_{\mathrm{ac}}$.  The
elastic energy mentioned below is then not a bookkeeping term but the
ordinary acoustic potential energy of that compression: the contribution
$\tfrac{1}{2\alpha\rho_0}\int_\Omega\psi^{2}\dd V$
of~\eqref{eq:EKL_def} is $p'^2/(2\rho_0c_{\mathrm{ac}}^{2})$ written in the
Boussinesq convention.

Two consequences shape everything that follows.  First, $\rhoa$ is not the
density that makes water sink.  It responds to the pressure relaxation
alone, while buoyancy is carried by the thermodynamic $\rho(\Theta,S,z)$ as
in any Boussinesq model; the two densities do different jobs and are kept
apart throughout.  Second, the sound speed is chosen.
The calibration used here makes $c_{\mathrm{ac}} = \sqrt{gH_{\mathrm{full}}}$,
so the artificial acoustic mode is exactly as fast as the barotropic gravity
wave the model already resolves -about $200\,\mathrm{m\,s^{-1}}$ for a
four-kilometre ocean, against some $1500\,\mathrm{m\,s^{-1}}$ for sound in
real seawater- and the compression it produces is of order $10^{-5}$ in
relative density.  The fluid is not being made compressible to go faster.  It
is given one slow, weak acoustic mode, deliberately so whose only
task is to relax divergence, and whose error is governed by the single
dimensionless number $\Fr_{\mathrm{baro}} = U/c_{\mathrm{ac}}$
of~\eqref{eq:Froude}.

This simple idea raises the immediate concern that the solenoidal constraint is abandoned, and with it
volume conservation. In a rigorous interpretation this is true. However, we show that with a careful
refinement of the localisation idea and respecting ocean conditions,
the concern is void: the volume error stays bounded and small and does
not accumulate over time, energy (including an elastic component) is
conserved, tracer content is conserved, and the non-hydrostatic
dispersion of internal gravity waves is maintained.  The details are
delivered in Section~\ref{sec:properties}.

\paragraph{Cost Analysis.}
The cost of the method is the second pillar, and the argument, made in Section~\ref{sec:pressure_calc}l.  On the accelerator systems on which such models run, the
time of a routine is set by datamovent, not by arithmetic intensity; the ICON time step is memory-bound by an order of
magnitude \citep{Klocke2025}.  The runtime of the AC/DC step relative
to the hydrostatic one is therefore a ratio of data volumes, from
which the machine cancels: the factor holds on any machine of this
class and needs no timing on a particular one.  We count data volume
in \emph{mesh passes}, the traffic of carrying one tracer once
through the mesh, and convert the operation counts read off the
source code into passes through the arithmetic intensity of each
routine.

In this unit the reason for the small overhead is visible 
(Figure~\ref{fig:why_cheap}): AC/DC adds no machinery of its
own.  The momentum tendencies gain terms 
the hydrostatic approximation omits, a cost any non-hydrostatic method
pays whatever computes its pressure; the depth-averaged
pseudo-pressure obeys an equation of the same stiffness as the free
surface and joins the host's free-surface solve as one more
two-dimensional field; the remaining pressure work is column-local
sweeps of the kind the implicit vertical mixing already performs.
Nothing iterates to convergence and nothing requires global
communication, so the count is fixed before the model runs and does
not grow with resolution or with node count.  The numbers follow: a
hydrostatic step costs about eight passes of operations and AC/DC
adds about two, a factor of $1.22$ by operation count; weighted by
the data each routine moves, the factor is $R\approx1.36$ with a
bracket of $1.24$ to $1.38$, that is, AC/DC lengthens the hydrostatic
time step by about a third, at any resolution and machine size, and
by less as tracers are added.  The projection method, priced in the
same unit, spends $10^{2}$--$10^{3}$ passes and as many rounds of
global communication on the pressure alone
(Appendix~\ref{app:cost_detail}).  The estimate is a scale analysis
under a stated configuration and the measured memory-bound regime,
not a measurement of AC/DC.

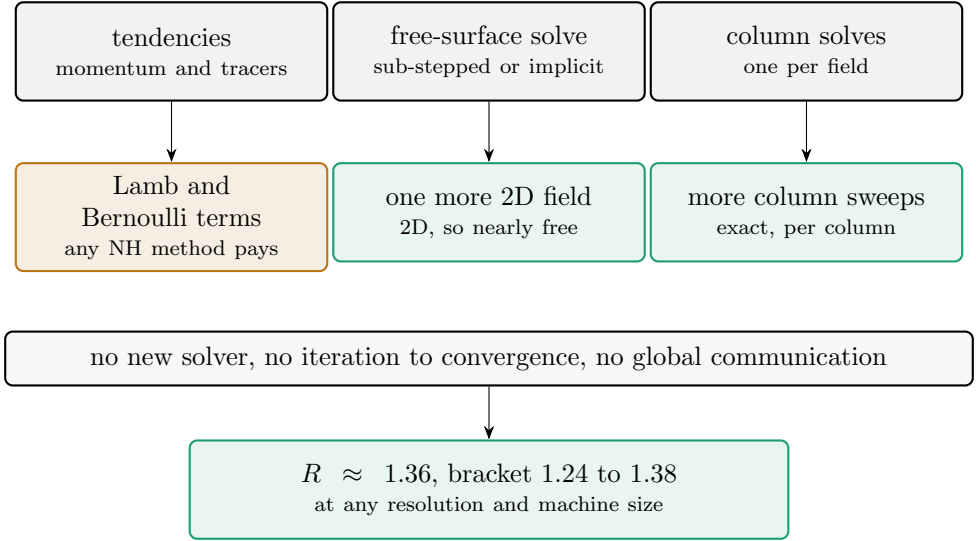
\begin{figure}[tbp]
  \centering
  \begin{tikzpicture}[
    font=\footnotesize,
    >={Stealth[length=1.8mm]},
    blk/.style={draw,thick,rounded corners=1mm,align=center,
                text width=38mm,inner sep=1.6mm,minimum height=13mm},
    host/.style={blk,fill=black!5},
    dyn/.style={blk,draw=cDyn,fill=cDyn!12},
    prs/.style={blk,draw=cPrs,fill=cPrs!10},
  ]
  \node[anchor=west,font=\scriptsize] at (-64mm,10mm) {the hydrostatic step already runs};
  \node[host] (H1) at (-42mm,0) {tendencies\\[-1pt]\scriptsize momentum and tracers};
  \node[host] (H2) at (0,0)     {free-surface solve\\[-1pt]\scriptsize sub-stepped or implicit};
  \node[host] (H3) at (42mm,0)  {column solves\\[-1pt]\scriptsize one per field};

  \node[dyn,below=8mm of H1] (A1) {Lamb and Bernoulli terms\\[-1pt]\scriptsize any NH method pays};
  \node[prs,below=8mm of H2] (A2) {one more 2D field\\[-1pt]\scriptsize 2D, so nearly free};
  \node[prs,below=8mm of H3] (A3) {more column sweeps\\[-1pt]\scriptsize exact, per column};
  \foreach \i in {1,2,3} \draw[->] (H\i) -- (A\i);

  \node[draw,thick,rounded corners=1mm,fill=black!3,inner sep=2mm,
        minimum width=128mm,below=9mm of A2] (N)
    {no new solver, no iteration to convergence, no global communication};
  \node[prs,below=7mm of N,text width=76mm] (R)
    {$R \approx 1.36$, bracket $1.24$ to $1.38$\\[-1pt]
     \scriptsize at any resolution and machine size};
  \draw[->] (N) -- (R);
  \end{tikzpicture}
  \caption{\it\small Why the overhead is small.  AC/DC adds no
    machinery of its own: each increment attaches to a block the
    hydrostatic model already runs.  The momentum tendencies gain the
    Lamb and Bernoulli terms that the hydrostatic approximation omits,
    which any non-hydrostatic method pays, whatever computes its pressure (amber).  The depth-averaged pseudo-pressure obeys an
    equation of the same stiffness as the free surface, so it joins
    whichever machinery the host uses for it, one more two-dimensional
    field in a sub-step loop or in an implicit solve.  The remaining
    pressure work is column-local sweeps of the same kind as the
    implicit vertical mixing, exact and independent per column
    (teal).  Nothing iterates to convergence and nothing stops the
    whole machine, so the count is fixed before the model is run and
    does not grow with resolution or with node count.  The resulting
    factor is derived in Section~\ref{sec:pressure_calc} and shown in
    Figure~\ref{fig:pass_budget}.}
  \label{fig:why_cheap}
\end{figure}

\paragraph{Integration into a Hydrostatic Ocean Model.}
AC/DC is designed as an increment to an existing hydrostatic core, not
as a new model.  Its components map onto machinery a hydrostatic model already has: the non-hydrostatic momentum terms use the model's
discrete operators; the vertical pressure equation is a tridiagonal
column solve of the same form as the implicit vertical mixing; the
pseudo-density is advanced by the model's transport scheme; and the
depth-averaged pseudo-pressure satisfies an equation of the same
structure and the same stiffness as the free surface, so it is
advanced inside the barotropic solve the model already performs, at
looser tolerance and with no synchronisation of its own.  Switching the
increments off recovers the hydrostatic model identically, and on a
variable-resolution mesh they act only where the mesh is fine
(Section~\ref{sec:telescoping}).  Sections~\ref{sec:disc_time_impl}
and~\ref{sec:ogcm} do this for ICON-O \citep{Korn2017}.

{\bf Structure. } The paper is organised as follows.  Section~\ref{sec:ac} formulates
artificial compressibility for thin fluids and refines it into AC/DC.
Section~\ref{sec:pressure_calc} establishes computational feasibility:
every available route to a non-hydrostatic pressure is priced in
passes, and the count is converted into runtime.
Section~\ref{sec:properties} analyses the method: energy, volume and
tracer errors, the free surface, the dispersion of internal gravity
waves, and Section~\ref{sec:diss-mixing-efficiency} develops the
dissipation and mixing diagnostics that resolved convection makes
computable.  Sections~\ref{sec:disc_time_impl} and~\ref{sec:ogcm} give
the discrete formulation and its integration into a hydrostatic ocean
model.  Section~\ref{sec:experiments} tests the method against
projection references, Section~\ref{sec:telescoping} takes it onto telescoping meshes, and Section~\ref{sec:conclusions} states what is
established and what is not.

\section{Artificial Compressibility for Thin Fluids }\label{sec:ac}
Consider the Boussinesq equations

\begin{equation}\begin{split}\label{Boussinesq_eq0}
  \Dt\bv + \mathbf{N}_{\mathrm{3D}} 
  + \nabla B_{\mathrm{3D}}
    &= \mathcal{D}(\bv) + b\,\bk,\\[2ex]
 \Div\bv &= 0, \\
 \Dt b +\nabla\cdot(b\bv)=\Delta b,
\end{split}\end{equation}
where the rotational part of the nonlinearity is given by $\mathbf{N}_{\mathrm{3D}}:=\bom_a\times \bv$, with absolute vorticity $\bom_a:=\nabla\times\bv+\mathbf{f}$ and $\mathbf{f}$ the Coriolis vector,
the Bernoulli function is  $\nabla B_{\mathrm{3D}}:=\nabla (E_{\mathrm{kin}}+p)$, with $E_{\mathrm{kin}}$
the kinetic energy, $p$ the pressure, and $b$ denotes the buoyancy.
\subsection{General Artificial Compressibility}\label{sec:ac_system}

The AC method introduces a pseudo-pressure $\psi$ as a new
prognostic variable.  The system reads:
\begin{align}
  \frac{1}{\alpha}\Dt\psi 
    + \frac{1}{\alpha}\Div(\psi\bv) &= - \Div\bv,
    \label{eq:ac_psi} \\
  \Dt\bv + \mathbf{N}_{\mathrm{3D}}
    + \nabla\!\left(B_{\mathrm{hyd}} + \frac{\psi}{\rho_0}\right)
    &= \mathcal{D}(\bv) + b\,\bk,
    \label{eq:ac_mom}
\end{align}
where the Bernoulli function is defined as
\begin{equation}\label{eq:B_hyd}
  B_{\mathrm{hyd}} := \frac{p_{\mathrm{hyd}}+p_{\mathrm{sfc}}}{\rho_0}
    + \tfrac12|\bv|^2 ,
\end{equation}
with the full three-dimensional kinetic energy and with
$p_{\mathrm{sfc}} := \rho_0 g\eta$ the surface pressure of the free-surface
elevation $\eta$ (Section~\ref{sec:free_surface}; on a fixed domain
$\eta\equiv0$ and $p_{\mathrm{sfc}}$ vanishes), and where we have split
the pressure
\begin{equation} \label{eq:pressure_split}
  p = p_{\mathrm{hyd}} + p_{\mathrm{sfc}} + p_{\mathrm{NH}}.
\end{equation}
The non-hydrostatic pressure
$p_{\mathrm{NH}}$ is deliberately absent: it is replaced by the pseudo-pressure $\psi$.  
The pseudo-pressure equation~\eqref{eq:ac_psi} replaces the
incompressibility constraint: divergence drives a pseudo-pressure
response, and $\nabla\psi$ in the momentum equation acts to
suppress divergence.  As $\alpha\to\infty$,
$\psi\to p_{\mathrm{NH}}$ and $\Div\bv\to 0$, recovering formally the full
incompressible NH equations.  The last term in~\eqref{eq:ac_psi} is the
flux-form transport of $\psi$; the reasons for carrying it are
collected at the end of Section~\ref{sec:acdc_closure}.

What matters is that~\eqref{eq:ac_psi} is a
\emph{local, explicit, hyperbolic} equation, no global solve, no
iteration, no preconditioner.  It replaces the Poisson problem with
a wave equation, and the ``cost'' of non-hydrostatic dynamics
becomes the cost of carrying one additional prognostic variable per
cell.  The convergence of AC to the incompressible limit as
$\alpha\to\infty$ is well established in the PDE literature
\citep{Temam1968,Guermond2015}; recent extensions to polytopal
meshes and variable-density flows are given by
\citet{Milani2022} and \citet{Cappanera2025}.

\paragraph*{Choice of the parameter $\alpha$.}
Here $\alpha$ has units of pressure and fixes the artificial
sound speed
$c_{\mathrm{ac}} := \sqrt{\alpha/\rho_0} = \sqrt{\tilde\alpha}$, with
$\tilde\alpha := \alpha/\rho_0$.  Throughout, the calibration is
\begin{equation}\label{eq:alpha_choice}
  \alpha = \rho_0\,g H_{\mathrm{full}} ,
  \qquad
  c_{\mathrm{ac}} = \sqrt{g H_{\mathrm{full}}} ,
\end{equation}
with $H_{\mathrm{full}}$ the full ocean depth: the artificial acoustic
speed then equals the barotropic gravity-wave speed, so the artificial
mode is carried by whatever machinery already resolves the barotropic
mode.  The dimensionless measure of the compressibility error is the
barotropic Froude number
\begin{equation}\label{eq:Froude}
\Fr_{\mathrm{baro}} := U/c_{\mathrm{ac}},
\end{equation} 
with $U$ the flow velocity
scale; at $U\sim1\,\mathrm{m\,s^{-1}}$ and $H_{\mathrm{full}} = 4000$\,m
this gives $\Fr_{\mathrm{baro}}\approx5\times10^{-3}$ and
$\Fr_{\mathrm{baro}}^{2}\approx2.5\times10^{-5}$.  Every AC error
established below is proportional to $\Fr_{\mathrm{baro}}^{2}$ except
where stated otherwise, and the experiments calibrate $\alpha$ per
configuration from the same rule $c_{\mathrm{ac}}\gg U$.


\subsection{Artificial Compressibility for Thin Fluids}\label{sec:acdc_closure}

We define the pseudo-density
\begin{equation}\label{eq:rho_ac}
  {\rhoa} \;:=\; \rho_0\!\left(1 + \frac{\psi}{\alpha}\right).
\end{equation}
As a consequence of~\eqref{eq:ac_psi}, the pseudo-density
satisfies the flux-form continuity equation
\begin{equation}\label{eq:ac_flux}
  \Dt{\rhoa} + \Div({\rhoa}\bv) = 0 ,
\end{equation}
which is~\eqref{eq:ac_psi} rescaled by $\rho_0/\alpha$.  The relaxation
form
\begin{equation}\label{eq:ac_cont}
  \Dt{\rhoa} + \rho_0\,\Div\bv
    = -\frac{\rho_0}{\alpha}\,\Div(\psi\bv) ,
\end{equation}
by contrast, holds only up to the  transport term on the
right-hand side, which is of order $\OO(1/\alpha)$. 
Artificial compressibility is usually written in the relaxation form. 

The AC system is structurally
weakly compressible, its pseudo-density transported conservatively
by~\eqref{eq:ac_flux}, the same flux form in which ocean models
transport mass and tracers.  The energy and tracer statements below
depend on this choice: the flux form is what makes them exact.


The refinement of AC through AC/DC rests on the fact that the three-dimensional
pressure problem for a thin fluid decomposes on prismatic meshes into a
vertical component (cheap, local) and a horizontal component (the
part that makes the Poisson solve expensive).  On prismatic meshes
the vertical direction is orthogonal to the horizontal, and the
three-dimensional Laplacian becomes 
\begin{equation}\label{eq:laplace_split}
\Delta = L_H + L_z,
\end{equation}
where $L_z = \partial_{zz}$ acts in $z$ at fixed $(x,y)$ and
$L_H = \nabla_H^2$ acts in $(x,y)$ at fixed $z$.

The non-hydrostatic pressure satisfies the Poisson equation
\begin{equation}\label{eq:laplace_split1}
(L_H + L_z)(p_{\mathrm{NH}}/\rho_0) = S,
\end{equation}
 where $S$ is the
divergence source from the velocity tendency.  AC/DC
decomposes this into three steps.
\begin{enumerate}[nosep]
\item \emph{Step 1: Column solve.}  For fixed horizontal position $(x,y)$
and time $t$, solve the vertical problem
\begin{equation}\label{eq:laplace_split2}
  L_z(p_V/\rho_0) = S(x,y,z,t),
\end{equation}
subject to homogeneous Dirichlet condition at the free surface,
  $p_V = 0$ and a homogeneous Neumann condition $\partial_z p_V = 0$ at
  the bottom.  
\item \emph{Step 2: Horizontal AC.}  The residual in the full Poisson
equation after applying $p_V$ is
\begin{equation}
  S^* = S - (L_H + L_z)(p_V/\rho_0)
  = S - L_z(p_V/\rho_0) - L_H(p_V/\rho_0)
  = -L_H(p_V/\rho_0),
  \label{eq:residual}
\end{equation}
since $L_z(p_V/\rho_0) = S$ by construction.  
The flux-form continuity equation for the pseudo-density drives the residual divergence to zero:
\begin{equation}
  \Dt  {\rhoa}  + \Div(  {\rhoa}\bv) = 0,
  \label{eq:hybrid_psi}
\end{equation}
where $\bv$ is the velocity corrected by both $p_V$ and $\psi$.  
\item \emph{Step 3: Aggregation.} Define the non-hydrostatic pressure as 
$p_{\mathrm{NH}} := p_V + \psi$.
\end{enumerate}
\vskip0.5cm
The AC/DC system is therefore~\eqref{eq:ac_psi} together with the
momentum equation carrying both parts of the aggregated pressure,
\begin{equation}\label{eq:acdc_mom}
  \Dt\bv + \mathbf{N}_{\mathrm{3D}}
    + \nabla\!\left(B_{\mathrm{hyd}}
        + \frac{p_V+\psi}{\rho_0}\right)
  = \mathcal{D}(\bv) + b\,\bk ,
\end{equation}
with $B_{\mathrm{hyd}}$ of~\eqref{eq:B_hyd}.  It differs from the pure-AC
system~\eqref{eq:ac_psi}--\eqref{eq:ac_mom} of
Section~\ref{sec:ac_system} by the single additional gradient
$-\nabla(p_V/\rho_0)$, supplied by the column solve of Step~1; setting
$p_V\equiv0$ recovers~\eqref{eq:ac_mom} and every statement made for
AC/DC reduces to its pure-AC counterpart. 
\vskip0.3cm
At convergence of the artificial-compressibility iteration ($\Div\bv\to 0$, equivalently
$\alpha\to\infty$), the momentum correction $-\nabla(\psi/\rho_0)$ cancels the residual
divergence left by $p_V$, so that $\psi$ solves the residual Poisson problem
\begin{equation}\label{eq:psi_converged}
  (L_H + L_z)(\psi/\rho_0) = S^* = -L_H(p_V/\rho_0).
\end{equation}
The aggregated pressure $p_{\mathrm{NH}} = p_V + \psi$ then recovers the full
three-dimensional Poisson equation~\eqref{eq:laplace_split1}:
\begin{equation}\label{eq:Laplace_solve}
  (L_H + L_z)(p_{\mathrm{NH}}/\rho_0)
  = \underbrace{(L_H + L_z)(p_V/\rho_0)}_{=\,S - S^*}
  + \underbrace{(L_H + L_z)(\psi/\rho_0)}_{=\,S^*}
  = S ,
\end{equation}
where the first brace uses the column solve $L_z(p_V/\rho_0)=S$ of~\eqref{eq:laplace_split2}
with $S^*=-L_H(p_V/\rho_0)$ from~\eqref{eq:residual}, and the second uses the converged AC
balance~\eqref{eq:psi_converged}.  At finite $\alpha$ this identity holds up to the
$\OO(1/\alpha)$ relaxation residual.

Equation~\eqref{eq:psi_converged} shows that the converged $\psi$ has vertical structure: 
the horizontal AC step therefore relaxes the stiff vertical
operator acting on $\psi$ as well, so the column solve removes vertical stiffness from $p_V$ but
not from $\psi$.



\subsection{Consequences of Artificial Compressibility}\label{sec:acdc_compress}
AC/DC approximates the Boussinesq equations by a specific compressible system. This is a
fundamental transition that imposes two modifications in the dynamical formulation:
the introduction of the pseudo-density in the tracer equations, and the formulation of kinetic energy. Both together restore the conservation structure of the original system: a flux-form continuity equation for the pseudo-density,
and an energy that is conserved identically and not only to order  $\OO(1/\alpha)$.

\subsubsection*{Tracer equation.}
Consistency between the continuity and the tracer equations, i.e. tracer fluxes are identical to volume fluxes 
for a constant tracer, imply the following form of the tracer equation
\begin{equation}\label{eq:onescale_tracer}
  \boxed{\;
  \Dt\bigl({\rhoa}C\bigr)
    + \Div\bigl({\rhoa}\,C\,\bv\bigr)
  \;=\;
  \Div\bigl(\rho_0\,\mathbb{K}\nabla C\bigr) ,
  \;}
\end{equation}
with $\mathbb{K}$ the symmetric positive semi-definite diffusivity. 
The reasons for this form are explained in Section~\ref{sec:tracer_transport}.

\subsubsection*{Kinetic Energy.}
Consider the classical definition of the kinetic energy for incompressible fluids 
in terms of the velocities $L^2$-norm
\begin{equation}
K_{\mathrm{df}} := \tfrac{1}{2}\int_\Omega |\bv|^2\dd V.
\end{equation}
We refer to this energy as the \emph{density-free kinetic energy}.

Take the inner product of~\eqref{eq:ac_mom} with $\bv$ and integrate over
the domain, using~\eqref{eq:ac_psi} to substitute
$\Div\bv = -\frac{1}{\alpha}\bigl(\Dt\psi + \Div(\psi\bv)\bigr)$.  The nonlinear term vanishes,
$\left\langle \bom_a\times\bv,\bv\right\rangle=0$, and the pressure terms give
\begin{equation}\begin{split}\label{eq:energy1}
  \left\langle\nabla\!\left(B_{\mathrm{hyd}}
    + \frac{\psi}{\rho_0}\right),\bv\right\rangle
  &= -\left\langle B_{\mathrm{hyd}}
    + \frac{\psi}{\rho_0},\;\Div\bv\right\rangle \\
  &= -\left\langle B_{\mathrm{hyd}},\;\Div\bv\right\rangle
    +\frac{d}{dt}\frac{1}{2\alpha\rho_0}\|\psi\|^2
    +\frac{1}{2\alpha\rho_0}\int_\Omega\psi^2\,\Div\bv\dd V,
\end{split}\end{equation}
where the last two terms come from
$$-\langle\psi/\rho_0,\Div\bv\rangle
= \alpha^{-1}\rho_0^{-1}\langle\psi,\Dt\psi + \Div(\psi\bv)\rangle
= \alpha^{-1}\rho_0^{-1}\langle\psi,\Dt\psi\rangle 
$$
together with $\langle\psi,\Div(\psi\bv)\rangle
= \tfrac12\int_\Omega\psi^2\Div\bv\dd V$.
The middle term in \eqref{eq:energy1} is  absorbed into the definition of  
the total energy, defined as kinetic plus elastic energy,
\begin{equation}
    e: = K_{\mathrm{df}} + \frac{1}{2\alpha\rho_0}\int_\Omega\psi^2\dd V,
  \label{eq:total_energy_df}
\end{equation}
leaving the inviscid budget
\begin{equation}\label{eq:df_budget_33}
  \frac{de}{dt}
  = \underbrace{\tfrac12\!\int_\Omega|\bv|^2\,\Div\bv\dd V}_{
      =:\;\mathcal{S}_{\mathrm{df}}}
  + \Bigl\langle\frac{p_{\mathrm{hyd}}+p_{\mathrm{sfc}}}{\rho_0},\;\Div\bv\Bigr\rangle
  - \frac{1}{2\alpha\rho_0}\int_\Omega\psi^2\,\Div\bv\dd V
  + \langle b,\,w\rangle .
\end{equation}
This budget is not closed, $\mathcal{S}_{\mathrm{df}}$ is the remainder of  the kinetic-energy part of the
Bernoulli gradient when $\Div\bv\neq0$; it is cubic in
the velocity, it is not sign-definite. Section \ref{sec:df_source} establishes
the incompatibility of $K_{\mathrm{df}}$ with energy conservation.
This is a manifestation of the energy--circulation dichotomy, investigated in 
\citep{Korn2026b} for the barotropic Navier--Stokes equations.  

The resolution is as follows:  
in the weakly compressible AC system with pseudo-density ${\rhoa}$, defined in~\eqref{eq:rho_ac}
and transported conservatively by~\eqref{eq:ac_flux}, the obstruction
$\mathcal{S}_{\mathrm{df}}$ is completely removed by
replacing $K_{\mathrm{df}}$ with the
density-weighted kinetic energy
\begin{equation}
  \boxed{
K_{\mathrm{dw}} := \tfrac{1}{2}\int_\Omega {\rhoa}|\bv|^2\dd V.
  }
  \label{eq:dw_energy}
\end{equation}
The mechanism, described in Proposition~\ref{prop:dw_removes}, is
that the time derivative of the weight supplies
$-\mathcal{S}_{\mathrm{df}}$, while the correction's own cubic
contribution telescopes into a divergence and integrates to zero.  Both
steps use the flux form~\eqref{eq:ac_flux}; with the relaxation
form~\eqref{eq:ac_cont} the second step leaves a remainder of order
$1/\alpha$, and the conservation statements below would hold only
asymptotically.
Its specific form
$$k_{\mathrm{dw}} := K_{\mathrm{dw}}/\rho_0
= \tfrac12\int_\Omega({\rhoa}/\rho_0)|\bv|^2\dd V$$
is the quantity all total-energy budgets below are written in;
throughout we assume
$\psi/\alpha > -1$, i.e.\ ${\rhoa}>0$,
since $\psi/\alpha = \OO(\Fr_{\mathrm{baro}}^2)$, so that
$k_{\mathrm{dw}}\ge0$.
For the kinetic plus elastic energy
\begin{equation}
    E: = k_{\mathrm{dw}} + \frac{1}{2\alpha\rho_0}\int_\Omega\psi^2\dd V
  \label{eq:total_energy_dw}
\end{equation}
the inviscid budget is, on a domain closed to mass flux,
\begin{equation}\label{eq:KL_energy_inviscid}
  \boxed{\;
    \frac{dE}{dt}
    = \Bigl\langle\frac{p_{\mathrm{hyd}}+p_{\mathrm{sfc}}}{\rho_0},\;
        \Div\!\bigl(\tfrac{{\rhoa}}{\rho_0}\bv\bigr)\Bigr\rangle
      + \Bigl\langle\frac{{\rhoa}}{\rho_0}\,b,\;w\Bigr\rangle ,
  \;}
\end{equation}
here every cubic
contribution has cancelled. 
 If the pressure is $\psi$
alone and buoyancy is absent, both pairings on the right vanish and
\begin{equation}\label{eq:KL_energy_exact}
  \frac{dE}{dt} = 0
\end{equation}
The first term in~\eqref{eq:KL_energy_inviscid} is the work of the hydrostatic and surface
pressures against the artificial compressibility, the second the
exchange with the potential-energy reservoir.  Both belong to the total
energy budget, inclusing potential energy, this is the topic of Section~\ref{sec:energy}.
The contrast with~\eqref{eq:df_budget_33} is: the same system, the same pressures, and the
cubic obstruction $\mathcal{S}_{\mathrm{df}}$ present in one budget and
completely absent from the other.

With the viscous term included as the density-weighted operator
in curl--divergence (Hodge) form,
\begin{equation}\label{eq:visc_hodge}
  \mathcal{D}(\bv)
  = \nu\,{\rhoa}^{-1}
    \bigl[\nabla({\rhoa}\Div\bv)
      - \nabla\times({\rhoa}\nabla\times\bv)\bigr],
\end{equation}
the viscous pairing closes, and when the two pairings
of~\eqref{eq:KL_energy_inviscid} vanish what remains is
\begin{equation}\label{eq:KL_energy_identity}
  \boxed{\;
    \frac{dE}{dt}
    = -\nu\int_\Omega\frac{{\rhoa}}{\rho_0}\bigl(|\nabla\times\bv|^2 + |\Div\bv|^2\bigr)\dd V \leq 0,
  \;}
\end{equation}
each term of~\eqref{eq:visc_hodge} integrating by parts on its own;
the density-weighting absorbs the Bernoulli
kinetic-energy pressure work that incompressibility would otherwise
eliminate.  The full-gradient weighted operator
$\nu{\rhoa}^{-1}\Div({\rhoa}\nabla\bv)$ is
not equivalent: pointwise
$$|\nabla\bv|^2 - |\nabla\times\bv|^2 - (\Div\bv)^2
= \Div\bigl[(\bv\cdot\nabla)\bv - \bv\,\Div\bv\bigr],$$
so its weighted dissipation differs
from~\eqref{eq:KL_energy_identity} by the commutator
$$\nu\rho_0^{-1}\int_\Omega\nabla{\rhoa}\cdot
\bigl[(\bv\cdot\nabla)\bv - \bv\,\Div\bv\bigr]\dd V,$$
of relative order $\OO(\Fr_{\mathrm{baro}}^2)$ since
$|\nabla{\rhoa}|/\rho_0 = \OO(\Fr_{\mathrm{baro}}^2/L)$.
If  we would transport ${\rhoa}$ by the relaxation
form~\eqref{eq:ac_cont} instead of the flux form~\eqref{eq:ac_flux},
the cubic terms would no longer telescope and~\eqref{eq:KL_energy_identity}
would acquire the $\OO(1/\alpha)$ remainder
$\tfrac{1}{2\alpha}\int_\Omega|\bv|^2\Div(\psi\bv)\dd V$; that is the
price of the more familiar form, and the reason it is not used here.

\paragraph*{Anisotropic Viscosity.}
Anisotropic viscosity and diffusivity can be handled analogously. Denote the anisotropic 
diffusivity by
$\kappa_h\gg\kappa_v$; and the viscosity, by $\nu_h\gg\nu_v$.
Write $\mathbf{P}_H := \mathsf{I}-\bk\otimes\bk$ for the horizontal
projector and
\begin{equation}\label{eq:aniso_tensors}
  \mathsf{K} := \kappa_h\,\mathbf{P}_H + \kappa_v\,\bk\otimes\bk ,
  \qquad
  \mathsf{M} := \nu_h\,\bk\otimes\bk + \nu_v\,\mathbf{P}_H .
\end{equation}
The transposed structure of $\mathsf{M}$ relative to
$\mathsf{K}$ is deliberate and is explained below.

On a prismatic Delaunay--Voronoi mesh every face is either horizontal, with normal
$\bk$, or lateral, with a horizontal normal, and in either case the
line joining the two adjacent circumcentres is parallel to that normal.
A grid-aligned tensor therefore has no off-diagonal action across any
face, and the two-point flux of~\eqref{eq:tracer_flux} remains exact
with
\begin{equation}\label{eq:kappa_face}
  \kappa_f =
  \begin{cases}
    \kappa_h, & f \text{ lateral},\\[2pt]
    \kappa_v, & f \text{ horizontal},
  \end{cases}
  \qquad \gamma_f = \frac{|f|}{d_f} .
\end{equation}


Anisotropy is therefore carried inside the curl--divergence
form~\eqref{eq:visc_hodge} itself, keeping the operator in the native
pair,
\begin{equation}\label{eq:visc_hodge_aniso}
  \mathcal{D}(\bv)
  = \frac{1}{{\rhoa}}\Bigl[
      \nabla\bigl({\rhoa}\,\nu_d\,\Div\bv\bigr)
      - \nabla\times\bigl({\rhoa}\,\mathsf{M}\,\nabla\times\bv\bigr)
    \Bigr],
\end{equation}
with $\mathsf{M}$ of~\eqref{eq:aniso_tensors} and $\nu_d\ge0$ the viscosity acting on the divergence, for
which~\eqref{eq:KL_energy_identity} becomes
\begin{equation}\label{eq:KL_energy_identity_aniso}
  \boxed{\;
  \frac{dE}{dt}
  = -\int_\Omega \frac{{\rhoa}}{\rho_0}
      \Bigl[\nu_d(\Div\bv)^2
        + \nu_h\,\omega_z^2
        + \nu_v\,|\bom_H|^2\Bigr]\dd V \;\le\; 0 ,
  \;}
\end{equation}
with $\bom = \nabla\times\bv = \omega_z\bk + \bom_H$.  The dissipation is resolved
into three separately diagnosable reservoirs instead of one.  The
assignment in~\eqref{eq:aniso_tensors} is the physical one: $\nu_h$
multiplies the vertical-axis vorticity generated by horizontal shear
and $\nu_v$ the horizontal-axis vorticity generated by vertical shear.
\section{Information Gain: Dissipation, Mixing and its Efficiency}\label{sec:diss-mixing-efficiency}
This section describes how under non-hydrostatic dynamics the mixing efficiency
becomes a computable quantity, while under hdrostatic dynamics it is prescribed by the modeller.

The mixing efficiency $\Gamma = \varepsilon_b/\varepsilon$ is the irreversible conversion of kinetic into potential
energy, $\varepsilon_b$, per unit of kinetic energy dissipated,
$\varepsilon$.  It is the factor in the Osborn relation
$K_\rho = \Gamma\varepsilon/N^2$ that turns a measured dissipation
into a diapycnal diffusivity \citep{Osborn1980}.  Osborn introduced
$\Gamma\approx0.2$ as an upper bound, taking a flux Richardson number
$R_f\lesssim0.15$; in later usage the inequality was dropped and the
value treated as a constant, and the review of \citet{Gregg2018}
finds that four decades of estimates and parametrizations of
$\Gamma$ have not converged.  In a hydrostatic model with
parametrized convection $\Gamma$ is an input in exactly this sense:
the closure prescribes it and the model returns the prescribed
value.  In a model that
resolves the convection $\Gamma$ can be an output, provided three
things hold: the dissipation $\varepsilon$ that enters it is the
physical one, not the sum of everything the discretisation removes;
the buoyancy-variance sink $\varepsilon_b$ is likewise physical; and
the numerical parts of both are computed separately, so that their
size is known and not inferred.  The section establishes the three
in turn, for the tracer variance first and for the kinetic energy
second, because the variance identity is the simpler of the two and
its proof is the pattern for the other.

The frame is the pair of budgets
\begin{equation}\label{eq:target_budgets}
  \frac{dE}{dt} = W - \int_\Omega \varepsilon\,\dd V
    - \varepsilon_{\mathrm{num}},
  \qquad
  \frac{dV}{dt} = P - \frac{\rho_0}{2}\int_\Omega \chi\,\dd V
    - \chi_{\mathrm{num}},
\end{equation}
with $W$ the energy input, $P$ the variance production, $\varepsilon$
and $\chi$ the dissipation rates of kinetic energy and of tracer
variance by the viscous and diffusive operators, and
$\varepsilon_{\mathrm{num}}$, $\chi_{\mathrm{num}}$ the removal by the
discretisation.  What makes the separation possible is that the
spatial discretisation of Section~\ref{sec:properties} conserves
energy and variance exactly, so that every term with a formula can be
evaluated from that formula, independently of the budget, and the
closure of the budget becomes a test of the implementation.  In a standard model only the sum is accessible,
as a residual.  The budgets themselves are a property of the
operators and hold for a hydrostatic model built on them as well;
resolved convection changes what the terms contain, not whether they
close.

\subsection{Dissipation Rate of Tracer Variance}\label{sec:variance}
The AC velocity has a residual divergence, so it is not obvious
that a variance budget closes at all; if it did not, $\chi$ would be
contaminated by the compressibility error.  It closes as an algebraic
identity, because the pseudo-density weighting absorbs the
divergence, and the physical and numerical parts of the sink separate
face by face.

We work with the semi-discrete transport of
Proposition~\ref{prop:tracer_consistency}.  Let $\Phi_f$ be the
pseudo-mass flux~\eqref{eq:Ff_def} through face $f$, the flux that
advances ${\rhoa}$ through the exact flux-form
balance~\eqref{eq:ac_flux} as hypothesis~(H1) requires, so that
$\Phi_f = F_f$ in the notation of that proposition.  Let $D$ be the
discrete divergence with the incidence convention that a face $f$
separates cells $c^-$ and $c^+$ with $D_{c^-\!,f}=+1$ and
$D_{c^+\!,f}=-1$, so that $\Phi_f>0$ is a flux from $c^-$ to $c^+$.
Write $[C]_f := C_{c^+}-C_{c^-}$ for the jump across $f$,
$\langle C\rangle_f := \tfrac12(C_{c^-}+C_{c^+})$ for the arithmetic
face mean, and $(RC)_f$ for the face reconstruction the scheme
actually uses.  The total tracer flux, written $J_f$ to keep it
distinct from the pseudo-mass flux, is
\begin{equation}\label{eq:tracer_flux}
  J_f = (RC)_f\,\Phi_f + J_f^{\mathrm{diff}},
  \qquad
  J_f^{\mathrm{diff}} = -\kappa_f\,\gamma_f\,[C]_f ,
\end{equation}
with $\kappa_f\ge0$ the face diffusivity and $\gamma_f>0$ the
geometric factor of the two-point diffusive flux; on the orthogonal
Delaunay--Voronoi meshes used throughout, $\gamma_f = |f|/d_f$ with
$d_f$ the distance between the adjacent circumcentres.

Two face sums recur throughout this section and are named now:
\begin{equation}\label{eq:chi_disc}
  \int_\Omega\chi\,\dd V := 2\sum_f\kappa_f\,\gamma_f\,[C]_f^2,
  \qquad
  \mathcal{D}_{\mathrm{num}}
    := -\sum_f \Phi_f\,[C]_f\bigl((RC)_f-\langle C\rangle_f\bigr) .
\end{equation}
The first is the dissipation rate of tracer variance by the explicit
diffusive flux, the discrete form of $\chi = 2\kappa|\nabla C|^2$
integrated over the domain; it is nonnegative and it is the
$\chi$ of the frame~\eqref{eq:target_budgets}.  The second is the
variance sink of the reconstruction: it measures, face by face, how
far the reconstructed face value departs from the arithmetic mean,
weighted by the jump and the flux.  It vanishes identically for the
centred reconstruction and is nonnegative for the upwind one; for a
general reconstruction it has no sign.  Both are explicit sums over
faces of fields the scheme holds.  Over a set of cells $\omega$ the
same sums restricted to the faces with both cells in $\omega$ are
written $\chi_\omega$ and $\mathcal{D}_{\mathrm{num},\omega}$.

\begin{proposition}[Tracer-variance budget under AC/DC]
  \label{prop:variance}
  Let the transport satisfy hypotheses~(H1)--(H3) of
  Proposition~\ref{prop:tracer_consistency} with the
  flux~\eqref{eq:tracer_flux}, on a domain closed to tracer flux.  Then
  the pseudo-density-weighted tracer variance
  \begin{equation}\label{eq:variance_def}
    V_h := \tfrac12\sum_c |c|\,\rho_{\mathrm{AC},c}\,C_c^2
  \end{equation}
  obeys the exact identity
  \begin{equation}\label{eq:variance_budget}
    \boxed{\;
    \frac{dV_h}{dt}
      = -\,\rho_0\,\mathcal{D}_{\mathrm{num}}
        \;-\;\frac{\rho_0}{2}\int_\Omega\chi\,\dd V \; }
  \end{equation}
  for any face reconstruction, with the two terms
  of~\eqref{eq:chi_disc}.  In particular:
  \begin{enumerate}[nosep]
  \item[(i)] for the centred reconstruction
    $(RC)_f = \langle C\rangle_f$, $\mathcal{D}_{\mathrm{num}} = 0$ and
    the only sink is the explicit one, $dV_h/dt = -\tfrac{\rho_0}{2}
    \int_\Omega\chi\,\dd V \le 0$;
  \item[(ii)] for the upwind flux
    $\mathcal{D}_{\mathrm{num}} = \tfrac12\sum_f|\Phi_f|\,[C]_f^2
    \ge 0$, a numerical sink of the same algebraic form as the
    explicit one.
  \end{enumerate}
\end{proposition}
\begin{proof}
  Appendix~\ref{app:proofs}.
\end{proof}

The residual divergence does not appear in~\eqref{eq:variance_budget}:
it is absorbed by the weighting ${\rhoa}$ in~\eqref{eq:variance_def},
as in the content budget of
Proposition~\ref{prop:tracer_consistency}.  The numerical sink is a
computable face sum, not a residual.

The two cases of the proposition are the ends of a family.  A
flux-corrected transport blends the two face values with a limiter
coefficient chosen from the local field \citep{Zalesak1979}, and the
sink of the blend is inherited from its ends.

\begin{corollary}[Variance sink under flux-corrected transport]
  \label{cor:fct_sink}
  Let the face value be the limited blend
  \begin{equation}\label{eq:fct_face}
    (RC)_f = C^{\mathrm{up}}_f + \lambda_f\bigl(\langle C\rangle_f
             - C^{\mathrm{up}}_f\bigr),
    \qquad \lambda_f\in[0,1],
    \qquad
    C^{\mathrm{up}}_f = \langle C\rangle_f
             - \tfrac12\operatorname{sgn}(\Phi_f)\,[C]_f ,
  \end{equation}
  of the upwind value and the centred mean, with $\lambda_f$ any
  limiter that respects hypotheses~(H1)--(H3).  Then
  \begin{equation}\label{eq:fct_sink}
    \mathcal{D}_{\mathrm{num}}
    = \tfrac12\sum_f (1-\lambda_f)\,|\Phi_f|\,[C]_f^2 ,
    \qquad
    0 \;\le\; \mathcal{D}_{\mathrm{num}}
    \;\le\; \tfrac12\sum_f |\Phi_f|\,[C]_f^2 ,
  \end{equation}
  the upper bound being the upwind sink of
  Proposition~\ref{prop:variance}(ii).
\end{corollary}
\begin{proof}
  From~\eqref{eq:fct_face},
  $(RC)_f - \langle C\rangle_f = -(1-\lambda_f)\tfrac12
  \operatorname{sgn}(\Phi_f)[C]_f$, so
  $\Phi_f[C]_f\bigl((RC)_f-\langle C\rangle_f\bigr)
  = -\tfrac12(1-\lambda_f)|\Phi_f|[C]_f^2$; insert
  in~\eqref{eq:chi_disc}.  The bounds follow from
  $0\le\lambda_f\le1$.
\end{proof}

Three things follow.  The sink is signed because the high-order
target is the centred mean; with any other target the identity of
Proposition~\ref{prop:variance} still holds, but the sign is lost.
Attribution is face by face: $1-\lambda_f$ is the fraction of
the upwind sink the limiter lets through at face $f$, so recording
the limiter coefficient, one face field, makes the numerical mixing
local as well as global.  It is also a property of the spatial operator, since $\lambda_f$ depends on the field, not on the time
step; shortening $\Delta t$ does not remove it, refining the mesh
does, as the gradients the limiter acts on become resolved.  This is
the scheme used in the experiments of Section~\ref{sec:experiments},
where $\mathcal{D}_{\mathrm{num}}$ is nonnegative at every step and
converges under time-step refinement to a nonzero limit.

\begin{corollary}[Variance budget over a region]\label{cor:variance_patch}
  Let $\omega$ be any set of cells, $\partial\omega$ the set of faces
  with exactly one adjacent cell in $\omega$, and for
  $f\in\partial\omega$ let $c(f)$ be that cell and
  $\sigma_f := D_{c(f),f}$ the outward orientation.  Under the
  hypotheses of Proposition~\ref{prop:variance}, but without closing
  the domain, the variance
  $V_\omega := \tfrac12\sum_{c\in\omega}|c|\,\rho_{\mathrm{AC},c}C_c^2$
  satisfies
  \begin{equation}\label{eq:variance_patch}
    \frac{dV_\omega}{dt}
    = -\,\rho_0\,\mathcal{D}_{\mathrm{num},\omega}
      \;-\;\frac{\rho_0}{2}\,\chi_\omega
      \;-\;\rho_0\!\!\sum_{f\in\partial\omega}\!\sigma_f
        \Bigl(C_{c(f)}\,J_f-\tfrac12\,C_{c(f)}^2\,\Phi_f\Bigr),
  \end{equation}
  with the interior sums of~\eqref{eq:chi_disc} and $J_f$ the total
  flux~\eqref{eq:tracer_flux}.
\end{corollary}
\begin{proof}
  The telescoping in the proof of Proposition~\ref{prop:variance},
  carried out over the cells of $\omega$: interior faces pair both
  sides and give the two interior sums; a boundary face contributes
  only its interior side, which is the boundary sum.
\end{proof}

The boundary sum is the discrete transport of variance through
$\partial\omega$, computed from interior-side values alone; its
continuous counterpart is
$\oint_{\partial\omega}\bigl(\tfrac12 C^2\,\bv + C\,\mathbf{q}\bigr)
\cdot\mathbf{n}\,\dd A$ with $\mathbf{q}$ the diffusive flux.  It is
transport, not dissipation: the sink of the region is the two interior
sums, and the budget closes over any focal region as it does globally.
For $\partial\omega$ empty, \eqref{eq:variance_patch}
is~\eqref{eq:variance_budget}.

The step from the variance sink to a diffusivity is the Osborn--Cox
argument \citep{OsbornCox1972}, and its assumptions are worth
stating, because they, not the identity, are what limit its use.
Split the tracer into a mean profile and a fluctuation,
$C = \overline{C}(z) + C'$.  The production term $P$ of the
frame~\eqref{eq:target_budgets} is the work of the turbulent flux
against the mean gradient, $P = -\rho_0\int_\Omega
\overline{w'C'}\,\partial_z\overline{C}\,\dd V$, and the eddy
diffusivity is defined by $\overline{w'C'} = -K_C\,\partial_z
\overline{C}$.  If the variance is statistically steady and its
transport across the region is negligible, production balances
dissipation, $P = \tfrac{\rho_0}{2}\int_\Omega\chi\,\dd V$, and for a
uniform mean gradient the two definitions combine to
$K_C = \chi/(2(\partial_z\overline{C})^2)$: a diffusivity follows
from the variance dissipation alone.  What the discrete budget adds
is that $\chi$ on the right is computed, and that the
non-physical part of the sink is known separately.

\begin{corollary}[Osborn--Cox diffusivity and irreversible mixing rate]
  \label{cor:osborn_cox}
  Under the hypotheses of Proposition~\ref{prop:variance}, in a
  statistically steady state with negligible variance transport and
  a uniform mean gradient $\partial_z\overline{C}$, the eddy
  diffusivity is
  \begin{equation}\label{eq:K_osborn_cox}
    K_C = \frac{\displaystyle\int_\Omega\chi\,\dd V}
               {2\,|\Omega|\,(\partial_z\overline{C})^2},
    \qquad
    \int_\Omega\chi\,\dd V = 2\sum_f\kappa_f\,\gamma_f\,[C]_f^2 ,
  \end{equation}
  with $\chi$ the explicit sum of~\eqref{eq:chi_disc}, and the
  numerical sink $\mathcal{D}_{\mathrm{num}}$ excluded from it.
  For the buoyancy, $C=b$, with the weighted potential energy
  $\mathrm{PE}_{\mathrm{AC}} := (2N^2)^{-1}\sum_c|c|
  (\rho_{\mathrm{AC},c}/\rho_0)\,b_c^2$, the discrete form of the
  small-amplitude, constant-$N^2$ reduction of~\eqref{eq:pe_def}, the
  same identity is
  $d\,\mathrm{PE}_{\mathrm{AC}}/dt
  = -\int_\Omega(\varepsilon_b+\varepsilon_b^{\mathrm{num}})\dd V$
  with the irreversible buoyancy-mixing rate
  \begin{equation*}
    \int_\Omega\varepsilon_b\,\dd V = N^{-2}\sum_f\kappa_f\gamma_f[b]_f^2 .
  \end{equation*}
\end{corollary}

The numerical sink $\chi_{\mathrm{num}}$ of the
frame~\eqref{eq:target_budgets} has two sources that are measured
differently: the reconstruction, whose contribution is the face sum
$\mathcal{D}_{\mathrm{num}}$ and falls with resolution, and the time
integration, which has no independent formula and appears as the
misfit of the semi-discrete identity under time-step refinement.

The relation to the discrete-variance-decay analysis of
\citet{BurchardRennau2008} and \citet{Klingbeil2014} is one of
complementary cuts through the same quantity.  Discrete variance
decay evaluates the fully discrete update: the variance a cell loses
in one time step beyond what the explicit diffusion accounts for,
exact for that update and for any scheme, linear or limited, and
therefore inclusive of the time-integration part, which the
semi-discrete identity here sees only as the misfit.  It is the
natural instrument for localising $\chi_{\mathrm{num}}$ in a
simulation, and the lock-exchange configuration in which it was
calibrated \citep{Ilicak2012} is the setting in which the sizes
reported in Section~\ref{sec:experiments} can be judged.  What the
face sum adds is the other cut: the spatial operators conserve
exactly, so $\mathcal{D}_{\mathrm{num}}$ is the whole spatial sink
and not a part of it; the identity holds under the non-solenoidal AC
velocity through the ${\rhoa}$ weighting; and, for a limited
scheme, Corollary~\ref{cor:fct_sink} attributes the sink to the
limiter coefficient face by face.  Applied to the same run, the two
diagnostics should agree on the spatial part and differ by the
time-integration part, which is a test either analysis can run.

\subsection{Dissipation of Energy}\label{sec:diss}
An ocean model removes energy from the resolved flow in four ways:
through the viscous closure, through the transport scheme, through
the time integration and, under AC/DC, through the compressibility
relaxation.  Only the first is physics.  The spatial discretisation
conserves energy exactly, \eqref{eq:EKL_budget_h}, so the four can be
computed one by one as sums over the mesh; in a model whose spatial
operators dissipate, only their total is accessible.  The four terms
are the following.

\begin{enumerate}[nosep]
\item[(a)] \emph{Physical dissipation.}  The viscous
  identity~\eqref{eq:KL_energy_identity_aniso} resolves the
  dissipation of the closure into three reservoirs.  The two
  rotational ones are physical friction,
  \begin{equation}\label{eq:eps_phys}
    \int_\Omega \varepsilon\,\dd V
    \;:=\; \int_\Omega \frac{{\rhoa}}{\rho_0}
       \bigl(\nu_h\,\omega_z^2 + \nu_v\,|\bom_H|^2\bigr)\dd V ,
  \end{equation}
  and this is the $\varepsilon$ of the mixing diagnostics: it acts on
  the vorticity of the resolved flow, it is nonnegative, and its
  integrand is local, so $\varepsilon$ is available as a spatial
  field.
\item[(b)] \emph{Compressibility dissipation.}  The divergence
  reservoir is not physics, since the ocean's velocity is solenoidal,
  and is counted separately,
  \begin{equation}\label{eq:eps_comp}
    \int_\Omega \varepsilon_{\mathrm{c}}\,\dd V
    \;:=\; \int_\Omega \frac{{\rhoa}}{\rho_0}\,\nu_d\,(\Div\bv)^2\,\dd V
    \;=\;\OO(\Fr_{\mathrm{baro}}^4) ,
  \end{equation}
  the order following from the divergence error of
  Section~\ref{sec:hybrid_div}.
\item[(c)] \emph{Advective mismatch.}  When buoyancy is advected
  with a non-centred reconstruction while the kinetic--potential
  conversion uses the centred interpolation, the two pairings do not
  cancel.  The mismatch is an explicit sum over faces, written
  $\mathcal{E}_{\mathrm{adv}}$, and it vanishes identically for the
  centred reconstruction (Section~\ref{sec:energy}).
\item[(d)] \emph{Time integration.}  It has no independent formula
  and appears as the misfit $r$ of the balance.
\end{enumerate}

\begin{proposition}[Separation of physical and numerical dissipation]
  \label{prop:diss_separation}
  Let $E_h^{\,n}$ be the discrete energy of~\eqref{eq:EKL_budget_h} at
  time level $n$, on a domain closed to mass and tracer flux and
  without forcing, and let $\varepsilon_h^{\,n}$,
  $\varepsilon_{\mathrm{c},h}^{\,n}$, $\mathcal{E}_{\mathrm{adv}}^{\,n}$
  be the discrete forms of~(a)--(c) and $W_{\alpha,h}^{\,n}$ the two
  $\OO(1/\alpha)$ pressure-work terms of~\eqref{eq:EKL_budget_h},
  evaluated at the time levels the scheme uses.  Then
  \begin{equation}\label{eq:closure_test}
    -\frac{E_h^{\,n+1}-E_h^{\,n}}{\Delta t}
    \;=\; W_{\alpha,h}^{\,n}
       + \varepsilon_h^{\,n} + \varepsilon_{\mathrm{c},h}^{\,n}
       + \mathcal{E}_{\mathrm{adv}}^{\,n}
       + r^{\,n},
  \end{equation}
  where every term except $r^{\,n}$ is a sum over cells, faces or
  edges of fields the scheme holds, and
  \begin{enumerate}[nosep]
  \item[(i)] $\varepsilon_h^{\,n}\ge0$ and
    $\varepsilon_{\mathrm{c},h}^{\,n}\ge0$;
  \item[(ii)] $\mathcal{E}_{\mathrm{adv}}^{\,n} = 0$ for the centred
    reconstruction;
  \item[(iii)] $r^{\,n}\to0$ as $\Delta t\to0$ at fixed mesh, at the
    order of the time scheme.
  \end{enumerate}
\end{proposition}
\begin{proof}
  The left side is computed from two states, since $E_h$ is a
  functional of the state.  (i)~In the native mimetic pair the
  viscous power is
  $-\langle\nu_d D\bv,D\bv\rangle_h - \langle\mathsf{M}C\bv,
  C\bv\rangle_h$ by the adjointness~(E2) of
  Appendix~\ref{app:boundedness}, with the weight ${\rhoa}$ carried
  inside the operator as in~\eqref{eq:visc_hodge_aniso}: a sum of
  weighted squares.  (ii)~is the adjoint pairing of transport and
  conversion of Section~\ref{sec:energy}.  (iii)~The semi-discrete
  balance~\eqref{eq:EKL_budget_h} is exact, so the spatial
  discretisation leaves no remainder; the misfit
  of~\eqref{eq:closure_test} is the defect of the time discretisation
  of an exactly conservative semi-discrete system, and vanishes with
  $\Delta t$ at the order of the scheme.
\end{proof}

The closure of~\eqref{eq:closure_test} under time-step refinement is
thus a computable test of the implementation; in the numerical
experiments every term is evaluated independently and the closure is
checked.  Like the variance budget, the energy balance localises.

\begin{corollary}[Energy budget over a region]\label{cor:energy_patch}
  Let $\omega$, $\partial\omega$, $c(f)$ and $\sigma_f$ be as in
  Corollary~\ref{cor:variance_patch}.  Under the hypotheses of
  Proposition~\ref{prop:diss_separation}, but without closing the
  domain, \eqref{eq:closure_test} holds over $\omega$ with the
  boundary term $\mathcal{F}_{\partial\omega}$ added to its right
  side, where
  \begin{equation*}
    \mathcal{F}_{\partial\omega}
    = \mathcal{F}^{\nabla}_{\partial\omega}
      + \Lambda_{\partial\omega}
      + \text{viscous band},
  \end{equation*}
  a two-point face sum for the gradient pairings~\eqref{eq:grad_flux},
  a one-stencil band of edge pairs for the Lamb
  term~\eqref{eq:lamb_band}, and a band of cut dual cells for the
  viscous term~\eqref{eq:visc_band}, all computed from cells, faces
  and boundary-adjacent edges of $\omega$ alone.  The continuous
  counterpart of $\mathcal{F}_{\partial\omega}$ is the energy flux
  \begin{equation}\label{eq:energy_flux}
    \mathcal{F}_{\partial\omega} \;\widehat{=}\;
    \oint_{\partial\omega}
      \Bigl(\tfrac12|\bv|^2
        + \frac{p_{\mathrm{hyd}}+p_{\mathrm{sfc}}+p_V+\psi}{\rho_0}\Bigr)
      \frac{{\rhoa}}{\rho_0}\,\bv\cdot\mathbf{n}\,\dd A
    \;+\;\text{viscous flux}.
  \end{equation}
\end{corollary}
\begin{proof}
  The mechanism is the patch form of the adjointness~(E2): for a cell
  field $\phi$ and a face field $u$,
  \begin{equation}\label{eq:patch_sbp}
    \langle\phi, Du\rangle_{h,\omega}
    \;=\; -\,(G\phi,\,u)_{h,\mathrm{int}(\omega)}
    \;+\; \sum_{f\in\partial\omega}\sigma_f\,\phi_{c(f)}\,u_f\,|f| ,
  \end{equation}
  with the interior product over the faces having both cells in
  $\omega$ and $|f|$ the face area of the divergence stencil; the
  global statement of Appendix~\ref{app:boundedness} is the case of
  empty $\partial\omega$, and~\eqref{eq:variance_patch} is the same
  identity applied twice.  The kinetic-energy gradient is a two-point
  difference between cell centres, so its energy pairing, and likewise
  the pairings of the four pressures and of the elastic term,
  telescope on the interior faces and stop at $\partial\omega$.
  Applied with the mass flux in the place of $u$, the gradient
  pairings collect the boundary work
  \begin{equation}\label{eq:grad_flux}
    \mathcal{F}^{\nabla}_{\partial\omega}
    \;=\; \sum_{f\in\partial\omega}\sigma_f
      \Bigl(K + \frac{p_{\mathrm{hyd}}+p_{\mathrm{sfc}}+p_V+\psi}
        {\rho_0}\Bigr)_{\!c(f)}\,\Phi_f ,
  \end{equation}
  the discrete counterpart of~\eqref{eq:energy_flux}, with
  interior-side evaluation and the pseudo-mass flux $\Phi_f$ of
  Corollary~\ref{cor:variance_patch}.  The Lamb term contributes no
  face flux.  Writing it as a matrix on edge velocities,
  $(\bom\times\bv)_e = \sum_{e'}w_{ee'}v_{e'}$, with weights built
  from the Stokes vorticity at the shared vertex, global energy
  neutrality is the antisymmetry $w_{ee'} = -w_{e'e}$; the
  restriction of an antisymmetric matrix to the edges of $\omega$ is
  antisymmetric, so the interior part of the quadratic form vanishes
  and what survives is
  \begin{equation}\label{eq:lamb_band}
    \Lambda_{\partial\omega}
    := \sum_{e\in\omega}\sum_{e'\notin\omega} w_{ee'}\,v_e\,v_{e'} ,
  \end{equation}
  a sum over edge pairs straddling the boundary, one stencil wide,
  computed from the weights the operator already holds.  Its
  continuum counterpart is zero, since $\bom\times\bv$ is orthogonal
  to $\bv$ pointwise, so~\eqref{eq:lamb_band} is a discrete boundary
  effect.  The viscous curl--curl term leaves the analogous band
  across the dual cells the boundary cuts: with $c_{d,e}$ the
  circulation weights of the discrete curl,
  \begin{equation}\label{eq:visc_band}
    \langle C^{\mathsf T}(\mathsf{M}C\bv), \bv\rangle_{h,\omega}
    \;=\; (\mathsf{M}C\bv,\,C\bv)_{h,\,d\subset\omega}
    \;+\; \sum_{d\ \mathrm{cut}} (\mathsf{M}C\bv)_d
          \sum_{e\in d\cap\omega} c_{d,e}\,v_e ,
  \end{equation}
  the first sum over the dual cells all of whose edges lie in
  $\omega$ and the second, the band, over the dual cells with edges
  on both sides, each contributing its partial circulation.
\end{proof}

As for the variance, the boundary terms are transport, not
dissipation, and do not enter the dissipative side of the balance.
The physical term~\eqref{eq:eps_phys} is a spatial field; the
numerical terms $\varepsilon_{\mathrm{c}}$, $\mathcal{E}_{\mathrm{adv}}$
and $r$ are global numbers, or per column for the compressibility
term, and serve as controls.  Their size relative to the physical
dissipation,
\begin{equation}\label{eq:quality}
  q := \frac{\varepsilon_{\mathrm{c},h}
        + |\mathcal{E}_{\mathrm{adv}}|
        + |r|}{\varepsilon_h} ,
\end{equation}
can be computed at every step and over any region.  It is also a
convergence measure in the method's own parameters:
$\varepsilon_{\mathrm{c}} = \OO(\Fr_{\mathrm{baro}}^4)$ falls with
$\tilde\alpha$, and $r$ falls with $\Delta t$.

\subsection{Mixing Efficiency}\label{sec:efficiency}
Two distinct quantities are called the mixing efficiency.  The
\emph{flux coefficient} $\Gamma := \varepsilon_b/\varepsilon$ is the
quantity in the Osborn relation $K_\rho = \Gamma\varepsilon/N^2$; the
\emph{flux Richardson number}
$R_f := \varepsilon_b/(\varepsilon_b+\varepsilon)$ is the fraction of
the turbulent kinetic energy converted irreversibly into potential
energy.  They are related by
\begin{equation}\label{eq:gamma_Rf}
  \Gamma = \frac{R_f}{1-R_f},
  \qquad
  R_f = \frac{\Gamma}{1+\Gamma},
\end{equation}
so that the canonical value $\Gamma = 0.2$ corresponds to
$R_f = 1/6 \approx 0.167$.  Throughout this paper $\Gamma$ denotes
the flux coefficient.  The constant is not settled.  The review of
\citet{Gregg2018} finds the estimates of $\Gamma$ spread by up to a
factor of five, attributes part of the spread to the multiplicity of
definitions, and recommends that $\Gamma$ be sought as a function of
two parameters, one for the stratification and one for the strength
of the turbulence, $\Gamma = \Gamma(\mathrm{Ri},\mathrm{Re}_b)$.  Such a dependence can only be
diagnosed in a computation in which both $\varepsilon_b$ and
$\varepsilon$ are outputs.

\begin{corollary}[Computed mixing efficiency]\label{cor:gamma_computed}
  Under the hypotheses of Propositions~\ref{prop:variance}
  and~\ref{prop:diss_separation}, with the buoyancy transported by
  the centred reconstruction,
  \begin{equation}\label{eq:gamma_h}
    \Gamma_h
    := \frac{\displaystyle N^{-2}\sum_f\kappa_f\gamma_f[b]_f^2}
            {\displaystyle \varepsilon_h}
  \end{equation}
  is an explicit function of the discrete state, with numerator the
  irreversible mixing rate of Corollary~\ref{cor:osborn_cox}, that is
  $N^{-2}$ times the buoyancy form of $\tfrac12\int_\Omega\chi\,\dd V$
  in~\eqref{eq:chi_disc}, and denominator the physical
  dissipation~(a).  The numerical energy
  removal relative to that denominator is the ratio $q$
  of~\eqref{eq:quality}, computed alongside it.  By
  Corollaries~\ref{cor:variance_patch} and~\ref{cor:energy_patch} the
  same holds over any region $\omega$, with the exchange across
  $\partial\omega$ accounted for.
\end{corollary}

The centred reconstruction is the one configuration in which both
budgets are purely physical: it removes $\mathcal{E}_{\mathrm{adv}}$
in the energy balance, Proposition~\ref{prop:diss_separation}(ii), and
the numerical sink in the variance balance,
Proposition~\ref{prop:variance}(i), because both are the same
non-adjointness of a non-centred tracer flux against the centred
buoyancy interpolation.  When a monotone transport is required and
the reconstruction is not centred, $\Gamma_h$ is still computable;
its numerator then contains the face sum $\mathcal{D}_{\mathrm{num}}$
and its denominator the mismatch $\mathcal{E}_{\mathrm{adv}}$, both
explicit and both reported with it.  For the flux-corrected
transport of Corollary~\ref{cor:fct_sink} the numerical sink is
state-dependent and is not removed by shortening the time step; it
falls with resolution, which the sweep of Remark~\ref{rem:eps_sgs}
measures, and it is reported, never driven to negligibility.  The
mixing diagnostics are reliable where $q\ll1$.

\begin{remark}[What the computed $\varepsilon$ contains]
  \label{rem:eps_sgs}
  The coefficients $\nu_h$, $\nu_v$ in~\eqref{eq:eps_phys} are eddy
  viscosities in the sense of Reynolds averaging: each is the sum of a
  physical part and a part standing in for the motions the mesh does
  not resolve.  The operator is linear in the coefficient, so the
  computed dissipation splits in the same proportion,
  $\varepsilon = \varepsilon_{\mathrm{res}} + \varepsilon_{\mathrm{sgs}}$,
  but the proportion is not known, because the coefficient is tuned.  It can be estimated by a resolution sweep.
  Over a fixed region and flow, refine the mesh: motions move from
  the parametrized to the resolved side, the eddy part of the
  coefficient falls, and the computed dissipation re-partitions.
  Writing $\varepsilon(\Delta x)$ for its value over the region,
  \begin{equation}\label{eq:eps_sgs}
    \varepsilon_{\mathrm{sgs}}(\Delta x)
     := \varepsilon(\Delta x) - \varepsilon_{\mathrm{res}},
    \qquad
    \varepsilon_{\mathrm{res}}
     := \lim_{\Delta x\to 0}\varepsilon(\Delta x).
  \end{equation}
  If the sum over the region stays unchanged across the sweep while
  its parts shift, the sum estimates the physical dissipation of the
  flow and the limit exists in practice; a sum that keeps changing is
  itself a diagnosis, of a coefficient removing energy erroneously or
  of a genuine flux of energy toward larger scales.  The same
  decomposition applies to the diffusivities $\kappa_h$, $\kappa_v$
  and hence to $\chi$ and to $\Gamma_h$.
\end{remark}

\section{Performance Analysis}\label{sec:pressure_calc}
With current technology a global non-hydrostatic ocean simulation in which 
 the non-hydrostatic elliptic pressure equation is solved by the projection
method is impossible (Appendix~\ref{app:cost_detail}), so at global a cost
comparison between AC/DC and the projection method would compare a
model that can be run with one that cannot.  We compare instead  the hydrostatic model 
and the same model with AC/DC added, and we ask by what factor the time step lengthens.

The argument proceeds in four steps.  We first state what sets the
runtime on present hardware.  From this we form the
ratio of the two runtimes, hydro and non-hydrostatic,  
and observe which of its factors cancel.
What survives cancellation is a ratio of data volumes, and we
measure the data volume  in units of ``mesh pass''.
The structure of both time steps is then described in this unit, and
at the end are numbers attached, for a representative
configuration using ICON-O.  Origins of the counted entries and the detail of the
$\bar\psi$ solve are in Appendix~\ref{app:cost_detail}.

\subsection{Runtime Factors}\label{sec:pass_runtime}
The runtime of a model is set by four
quantities:
\begin{enumerate}[nosep]
  \item[(i)] $F$, the \emph{operation count}: how much arithmetic the
    routine does;
  \item[(ii)] $Q$, the \emph{memory traffic}: how much data that work
    moves between memory and processor;
  \item[(iii)] $B$, the \emph{memory bandwidth}: how fast the machine
    moves data;
  \item[(iv)] $\eta$, the \emph{achieved bandwidth fraction}: how much
    of that speed the routine actually reaches, which depends on how
    it is written.
\end{enumerate}
The first two are linked by the \emph{arithmetic intensity}
$$I := F/Q,$$
 the operations performed on each byte fetched.  Every
machine has its own ratio of arithmetic speed to bandwidth.  
When the intensity $I$ lies below that ratio the routine is \emph{memory-bound}, i.e. the
processor waits for data, and the run time is set by the data volume alone,
\begin{equation}\label{eq:routine_time}
  t \;=\; \frac{Q}{\eta\,B} \;=\; \frac{F}{I\,\eta\,B} .
\end{equation}
This is the memory-bound branch of the roofline model
\citep{Williams2009}.  Whether an ocean time step falls on this branch
is a question of measurement, and the measurement
exists: on the accelerator system of the Gordon Bell run of
\citet{Klocke2025}, the machine's ratio is about $15$ operations per
byte, while the routines of the ICON time step lie between $0.4$ and
$1.6$; a tracer sweep performs some $50$ operations per cell while
moving four fields, a tridiagonal column solve some $10$, while moving
three.  The step is limited by data movement throughout, an order of
magnitude below the crossover.  We take this machine as
representative of the class on which such models will run, and use
\eqref{eq:routine_time} for every routine of both models.  The
consequence for what follows is that on such machines the
algorithm costs are costs of reading.

\subsection{Hydrostatic-Non-Hydrostatic Runtime Ratio}\label{sec:pass_ratio}
Adding~\eqref{eq:routine_time} over the all subroutines of a time step, the
elapsed time of the AC/DC step relative to the hydrostatic one is
\begin{equation}\label{eq:runtime_ratio}
  R \;=\;
  \frac{\displaystyle\sum_{j\,\in\,\mathrm{AC/DC}} \frac{F_j}{I_j\,\eta_j\,B}}
       {\displaystyle\sum_{k\,\in\,\mathrm{hydrostatic}} \frac{F_k}{I_k\,\eta_k\,B}} .
\end{equation}
Consider the different quantities individually.
\
\begin{enumerate}[nosep]
  \item[(i)]  The bandwidth $B$ a machine property, since it is the
same machine in numerator and denominator, it cancels.  The
factor $R$ therefore holds for any hydrostatic wall-clock time on any
machine of the memory-bound class, no timing on a particular
machine is needed.  In this sense the
comparison is structural: the machine drops out.

  \item[(ii)] The efficiencies $\eta_j$ do not cancel term by term, because each
routine reaches its own fraction of the total bandwidth.  They enter,
however, only through the question whether the routines AC/DC adds are
written as well as the routines of the host.  We collect this in a
single number $m$, the ratio of the mean efficiency of the AC/DC
routines to that of the hydrostatic ones; $m=1$ when they are written
equally well, and $m$ is what an implementation, not the method,
controls.
\end{enumerate}
What survives is a ratio of two sums of $F_j/I_j$, that is, of memory
traffic $Q_j$.  The runtime ratio is a ratio of data volumes.  It is
therefore natural to count cost in a unit of data volume, and the
next subsection introduces one.  Separating the hydrostatic step from
the increments AC/DC adds to it, the ratio takes the form
\begin{equation}\label{eq:R_form}
  R \;=\; 1 + m\,\frac{\mathrm{OV}}{\mathrm{BASE}} + p_{\bar\psi},
  \qquad
  \mathrm{OV} := \!\!\sum_{j\,\in\,\mathrm{increments}}\!\! T_j ,
  \qquad
  \mathrm{BASE} := \!\!\sum_{k\,\in\,\mathrm{hydrostatic}}\!\! T_k ,
\end{equation}
where $T_j$ is the traffic of routine $j$ in the unit defined next,
$\mathrm{OV}$ is the traffic of everything AC/DC adds, $\mathrm{BASE}$
the traffic of the hydrostatic step, and $p_{\bar\psi}$ the
depth-averaged solve, which uses machinery the model already runs and
is carried separately.
\begin{figure}[tbp]
  \centering
  \begin{tikzpicture}[
    font=\footnotesize,
    >={Stealth[length=1.8mm]},
    ing/.style={draw,thick,rounded corners=1mm,align=center,
                text width=27mm,inner sep=1.6mm,minimum height=9mm},
    keep/.style={ing,draw=cSurvive,fill=cSurvive!10},
    drop/.style={ing,draw=cDrop,fill=black!4},
  ]
  \node[draw,thick,rounded corners=1mm,inner sep=2mm,minimum width=126mm] (h) at (0,0)
    {time of one routine\quad $t=Q/(\eta B)=F/(I\,\eta\,B)$};

  \node[ing,below left=7mm and 24mm of h.south] (F)
    {$F$ operations\\[-1pt]\scriptsize counted from source};
  \node[ing,right=3mm of F] (I)
    {$I$ intensity\\[-1pt]\scriptsize operations per byte};
  \node[ing,right=3mm of I] (B)
    {$B$ bandwidth\\[-1pt]\scriptsize the machine};
  \node[ing,right=3mm of B] (E)
    {$\eta$ efficiency\\[-1pt]\scriptsize how it is written};

  \foreach \n in {F,I,B,E} \draw[->] (h.south-|\n.north) -- (\n.north);

  \node[keep,below=8mm of F.south east,anchor=north,text width=57mm] (T)
    {$F/I$ is traffic, counted in mesh passes\\[-1pt]\scriptsize what remains};
  \node[drop,below=8mm of B] (Bc)
    {cancels\\[-1pt]\scriptsize same machine};
  \node[drop,below=8mm of E] (Ec)
    {collapses to $m$\\[-1pt]\scriptsize one number};

  \draw[->] (F.south) -- (F.south|-T.north);
  \draw[->] (I.south) -- (I.south|-T.north);
  \draw[->] (B.south) -- (Bc.north);
  \draw[->] (E.south) -- (Ec.north);

  \node[keep,below=8mm of T,text width=90mm] (R)
    {$R = 1 + m\,(\Delta T_{\mathrm{NH}}+\Delta T_{\mathrm{acdc}})/T_{\mathrm{hyd}} + p_{\bar\psi}$\\[-1pt]
     \scriptsize a ratio of mesh passes, on any machine};
  \draw[->] (T.south) -- (T.south|-R.north);
  \draw[->] (Ec.south) |- ($(R.north east)+(0,4mm)$) -- ($(R.north east)-(6mm,0)$);
  \end{tikzpicture}
  \caption{\it\small What survives the ratio of two runtimes.  Four
    quantities set the time of a routine on a memory-bound machine.
    Under the ratio~\eqref{eq:runtime_ratio} the bandwidth $B$ cancels
    exactly, since numerator and denominator run on the same machine,
    and the per-routine efficiencies collapse into the single number
    $m$, which an implementation controls.  What remains is $F/I$,
    the memory traffic, counted in mesh passes.  The factor therefore
    holds for any hydrostatic wall-clock time on any machine of this
    class, and no timing on a particular machine is needed to state
    it.}
  \label{fig:cost_cascade}
\end{figure}
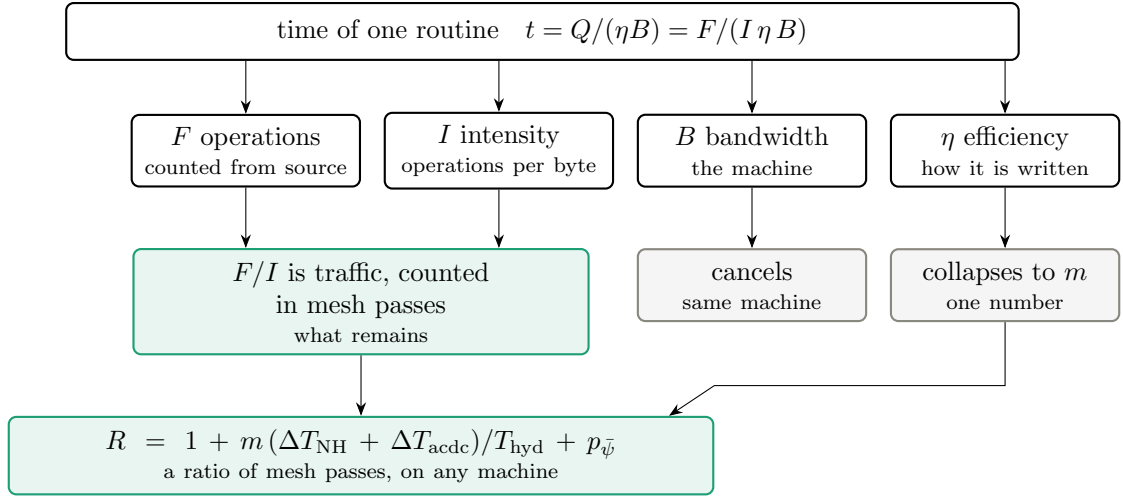
\subsection{Cost Counting Cost}\label{sec:pass_counting}
The counting unit is  the \emph{mess pass}: the amount of data moved by
moving a single tracer once through the mesh with the transport routine,
one cell at a time.  The pass is the smallest unit of work an ocean
model performs that touches every cell, and every routine of a time
step is some number of such passes.  

A routine that touches every cell once with the same data per cell as a tracer sweep costs one pass; a
routine that touches every cell with more or fewer data per cell costs
correspondingly more or less.  Two-dimensional routines, which touch
only the surface cells, cost a fraction of a pass equal to the ratio of
surface cells to volume cells, which on a mesh of $n_z$ layers is
$1/n_z$ per touch.  This is the reason why the two-dimensional pieces
of both models are nearly free, and it applies to the barotropic
solve of the host and to the barotropic pseudo-pressure of AC/DC.

The reference pass performs $C_{\mathrm{stencil}}$
operations per cell,
\begin{equation*}
  F_1 = N_{\mathrm{col}}\,n_z\,C_{\mathrm{stencil}},
\end{equation*}
with $N_{\mathrm{col}}$ the number of columns and $n_z$ the number of
layers, at intensity $I_1$.  A routine of $F_j$ operations at
intensity $I_j$ moves
\begin{equation}\label{eq:Tj}
  T_j \;=\; \frac{F_j}{F_1}\,\frac{I_1}{I_j}
\end{equation}
mesh passes of data.  The first factor is the operation count in units of 
mesh passes, or tracer sweeps; the second
converts operations to traffic and is what distinguishes run time
from an operation count.  A routine with less arithmetic per byte than
the reference pass weighs more in traffic than in operations.  
The increments AC/DC adds are the parts of the step with the least arithmetic per
datum touched, and so are the host's own column solves, which is why
the correction moves numerator and denominator of \eqref{eq:R_form} in
the same direction.  Intensities are obtained from the arithmetic of each stencil and the fields it touches, at eight bytes per value
and without cache reuse between cells.  

\subsection{Time Step Structure}\label{sec:pass_acdc}
Three blocks make up the hydrostatic time step.  The tendencies of
momentum and tracers dominate: each tracer is one pass by definition,
and the momentum tendencies, vector-invariant advection with
reconstruction, pressure gradients, the equation of state and the
mixing coefficients, are a fixed number of passes independent of the
tracer count.  In either of its two forms the surface solve is two-dimensional and
nearly free, sub-stepped or solved implicitlys; the two forms differ 
not in traffic but in communication,
the sub-step loop exchanging data with its neighbours, the implicit solve needing two global reductions per iteration.  Implicit vertical mixing costs one tridiagonal solve per column and per
prognostic field.  The tridiagonal solve is the routine with the least arithmetic
per datum in the step, and it is the routine AC/DC will add more of.
Together the three blocks form the denominator of every factor quoted
below.

AC/DC inherits this structure and adds four terms, each with a fixed amount
of work per time step and each counted before the model is run.
\begin{enumerate}[nosep]
  \item[(i)]  The first is the \emph{contraction increment}.  The vector-invariant
form is used in either model; the non-hydrostatic one adds to the Lamb
contraction the terms that the hydrostatic approximation omits, one
extra pass over the edge triples.  This is the cost of non-hydrostatic
dynamics, not of AC/DC: any non-hydrostatic method on this dynamical core must pay it.

 \item[(ii)]  The second is the \emph{column solve}, one tridiagonal solve and one
vertical gradient per column.  The tridiagonal is the same operation
as the host's implicit vertical mixing for one field, and it is local:
each column is solved independently of its neighbours, with no
communication and no global stop.

 \item[(iii)]  The third is the \emph{barotropic pseudo-pressure} $\bar\psi$.  It is
two-dimensional, and the same conversion that made the host's surface
solve nearly free applies.  This row covers the entire acoustic
sub-loop, for three reasons (Section~\ref{sec:variants2}).  The
pressure recursion acts on the depth average and is two-dimensional.
The sub-step count follows from the horizontal acoustic condition
alone, since the vertical is implicit, and with
$c_{\mathrm{ac}}=\sqrt{gH}$ the stepping is the model's barotropic
stepping itself: AC/DC subdivides nothing of its own but adds one
field to sub-steps that already exist, and on a semi-implicit model,
which has no sub-step loop, $\bar\psi$ enters the free-surface solve at
a looser tolerance and hence at a fraction of its iterations
(Appendix~\ref{app:cost_detail}).  Finally, the velocity update per
sub-step is the gradient of a field constant in the vertical, and
reduces to a two-dimensional accumulator with a single
three-dimensional correction at the end of the loop; without this
reduction the sub-loop would be three-dimensional at every sub-step
and would dominate the step.

 \item[(iv)]  The fourth is the \emph{baroclinic $\psi'$ update}, one column-implicit
sweep per step.
\end{enumerate}
AC/DC performs no global communication of its
own: the column solve is local, the horizontal sub-steps are
nearest-neighbour, and $\bar\psi$ satisfies an equation of the same
stiffness as the free surface, so it is advanced inside the barotropic
solve the host already performs.  
The overhead of AC/DC is small, and
its decisive property is that it is bounded, known in advance, and flat
under refinement.

\subsection{Specific Numbers}\label{sec:pass_numbers}
The configuration used for the numbers is submesoscale-resolving:
$\Delta x = 50$\,m and $\Delta z = 1$\,m in the upper ocean stretching
to $50$\,m at depth over $H = 4000$\,m, with
$N_{\mathrm{col}} = 4\times10^{6}$ columns, $n_z = 500$ levels,
$n_{\mathrm{sub}} = 200$ barotropic sub-steps and
$n_{\mathrm{tracer}} = 2$ tracers, temperature and salinity.  The
reference sweep uses $C_{\mathrm{stencil}}\sim50$ FLOP per cell for
advection with reconstruction, so that $F_1\approx10^{11}$ FLOP, at
intensity $I_1 = 1.56$.  With $n_z = 500$ it takes $2500$
two-dimensional touches to equal one three-dimensional pass.

Table~\ref{tab:cost_accounting} gives the operation counts.  The
hydrostatic tendencies bring the counted total to $7\,F_1$, some $350$
FLOP per cell; the sub-stepped surface solve, $n_{\mathrm{sub}} = 200$
steps at $20$ FLOP per surface cell, costs $0.16\,F_1$ (an implicit
solve at $200$ iterations would cost $0.24\,F_1$); the implicit
vertical mixing, $(n_{\mathrm{tracer}}{+}2)=4$ tridiagonal solves,
$0.80\,F_1$.  The hydrostatic step is about $8$ passes of operations.
The non-hydrostatic Lamb contraction costs $3.86\,F_1$ against the
hydrostatic core's $2.96\,F_1$, both counted from source
(Appendix~\ref{app:cost_detail}), an increment of $0.90\,F_1$.  For each of its two parts the column solve is $10$ FLOP per cell for each of its two parts,
$0.20 + 0.20\,F_1$, the tridiagonal a quarter of the host's implicit
mixing.  The $\psi'$ update is $20$ FLOP per cell, $0.40\,F_1$.  The
barotropic $\bar\psi$, $200$ sub-steps at $10$ FLOP per surface cell,
costs $0.08\,F_1$; without the two-dimensional reduction of the
sub-loop it would cost $15$--$20$ passes.  The overhead totals $1.78$
passes, $0.90$ for the dynamics and $0.88$ for the pressure, and the
AC/DC step is about $10$ passes of operations.  Spanning the counting
conventions for the contraction widens the overhead to
$1.68$--$2.18$.

\begin{table}[ht]
  \centering
  \caption{Operation counts for the hydrostatic baseline and the
    AC/DC overhead, in tracer-equivalents
    $F_1 = N_{\mathrm{col}}\,n_z\,C_{\mathrm{stencil}}
    \approx 10^{11}$ FLOP, for the configuration of the text
    ($N_{\mathrm{col}} = 4\times 10^6$, $n_z = 500$,
    $n_{\mathrm{sub}} = 200$, $n_{\mathrm{tracer}} = 2$).  In the
    operation counts, ``c'' denotes FLOP per cell over
    $N_{\mathrm{col}}\,n_z$ cells and ``s'' FLOP per cell per sub-step
    over $n_{\mathrm{sub}}\,N_{\mathrm{col}}$.  The overhead is
    separated into the increment of non-hydrostatic dynamics, which any non-hydrostatic method pays, and the cost of the AC/DC pressure.
    Provenance of the two contraction entries is given in
    Appendix~\ref{app:cost_detail}.}
  \label{tab:cost_accounting}
  \renewcommand{\arraystretch}{1.05}
  \begin{tabular}{llc}
    \toprule
    Term & Operation count & $F_1$ \\
    \midrule
    \emph{Hydrostatic baseline}
      & tendency $7.00$, barotropic $0.16$, vert.\ impl.\ $0.80$
      & $\mathbf{\approx 8}$ \\
    \midrule
    \multicolumn{3}{l}{\emph{Non-hydrostatic dynamics, independent of the pressure method}} \\
    3D contraction increment
      & NH $3.86$ less hydrostatic $2.96$; one $\widetilde U$ pass
      & $0.90$ \\
    \midrule
    \multicolumn{3}{l}{\emph{AC/DC pressure}} \\
    Column solve: tridiagonal
      & $N_{\mathrm{col}}\,n_z \times 10$\,c $= C_{\mathrm{vimp}}/4$
      & $0.20$ \\
    Column solve: gradient
      & $N_{\mathrm{col}}\,n_z \times 10$\,c
      & $0.20$ \\
    Barotropic $\bar\psi$
      & $n_{\mathrm{sub}}\,N_{\mathrm{col}} \times 10$\,s
      & $0.08$ \\
    Baroclinic $\psi'$ update
      & $N_{\mathrm{col}}\,n_z \times 20$\,c
      & $0.40$ \\
    \quad\emph{pressure subtotal} & & $\mathbf{0.88}$ \\
    \quad\emph{overhead total} & & $\mathbf{1.78}$ \\
    \midrule
    \quad\emph{AC/DC total} & & $\mathbf{\approx 10}$ \\
    \bottomrule
  \end{tabular}
\end{table}

Table~\ref{tab:intensity} converts operations to traffic
by~\eqref{eq:Tj}.  The baseline grows from $8$ passes of operations to
$T_{\mathrm{hyd}} = 10.4$ passes of traffic, because the implicit
vertical mixing has little arithmetic per byte; the increments grow
from $1.78$ to $3.4$, because the column solve and the $\psi'$ update
have the same property.  The two movements partly cancel in the ratio.
Of the $3.4$, one pass is the non-hydrostatic dynamics,
$\Delta T_{\mathrm{NH}} = 0.94$, and the remaining
$\Delta T_{\mathrm{acdc}} = 2.50$ is the AC/DC pressure.

\begin{table}[ht]
  \centering
  \caption{Operation count $F_j$ in tracer-equivalents, arithmetic
    intensity $I_j$ in operations per byte, and memory traffic
    $T_j = (F_j/F_1)(I_1/I_j)$ in passes, for the
    configuration of Section~\ref{sec:pass_numbers}.  The reference
    sweep has $I_1 = 1.56$.  Routines with less arithmetic per byte
    than the reference weigh more in traffic than in operations.  The
    increments are separated into the non-hydrostatic dynamics
    $\Delta T_{\mathrm{NH}}$, which any non-hydrostatic method pays, and the AC/DC pressure $\Delta T_{\mathrm{acdc}}$.}
  \label{tab:intensity}
  \renewcommand{\arraystretch}{1.05}
  \begin{tabular}{lccc}
    \toprule
    Routine & $F_j$ & $I_j$ & $T_j$ \\
    \midrule
    \multicolumn{4}{l}{\emph{Hydrostatic step}} \\
    Tracer transport ($n_{\mathrm{tracer}}=2$) & $2.00$ & $1.56$ & $2.00$ \\
    Momentum tendency                          & $5.00$ & $1.50$ & $5.21$ \\
    Implicit vertical mixing                   & $0.80$ & $0.42$ & $3.00$ \\
    Surface solve                              & $0.16$ & $1.56$ & $0.16$ \\
    \quad$T_{\mathrm{hyd}}$                    &        &        & $\mathbf{10.4}$ \\
    \midrule
    \multicolumn{4}{l}{\emph{Increment of non-hydrostatic dynamics, carried by any method}} \\
    3D contraction increment                   & $0.90$ & $1.50$ & $0.94$ \\
    \quad$\Delta T_{\mathrm{NH}}$              &        &        & $\mathbf{0.94}$ \\
    \midrule
    \multicolumn{4}{l}{\emph{Increment of the AC/DC pressure}} \\
    Column solve: tridiagonal                  & $0.20$ & $0.42$ & $0.75$ \\
    Column solve: vertical gradient            & $0.20$ & $0.42$ & $0.75$ \\
    $\psi'$ update                             & $0.40$ & $0.62$ & $1.00$ \\
    \quad$\Delta T_{\mathrm{acdc}}$            &        &        & $\mathbf{2.50}$ \\
    \bottomrule
  \end{tabular}
\end{table}

The runtime factor is now computed from~\eqref{eq:R_form} in three
steps: a central value, a lower bound, and an upper bound.  Each step
states exactly which ingredients it uses, so that every number can be
recomputed from the two tables.

\emph{Central value.}  Insert the traffic sums of
Table~\ref{tab:intensity}, $\Delta T_{\mathrm{NH}} = 0.94$ and
$\Delta T_{\mathrm{acdc}} = 2.50$ against $T_{\mathrm{hyd}} = 10.4$,
take the AC/DC routines to be written as well as the host's, $m = 1$,
and take the depth-averaged solve at the middle of its range,
$p_{\bar\psi} = 0.03$.  Then
\begin{equation*}
  R \;=\; 1 + \frac{0.94 + 2.50}{10.4} + 0.03 \;\approx\; 1.36 .
\end{equation*}

\emph{Lower bound.}  The assigned intensities are the one ingredient
that is not counted from source.  If every routine were as
arithmetic-rich as the tracer sweep, $I_j = I_1$ throughout, traffic
and operations would coincide and~\eqref{eq:R_form} would reduce to
the ratio of the operation counts of Table~\ref{tab:cost_accounting},
$1 + (0.90 + 0.88)/8 \approx 1.22$; across the counting conventions for the
contraction, $1.21$ to $1.27$.  This cannot be the true value, because
the column solves of both models do move more data per operation than
a tracer sweep, but it is the value below which no assignment of
intensities can push the factor, since every correction in
Table~\ref{tab:intensity} moves a routine to an intensity at or below
$I_1$ and the increments contain proportionally more such routines
than the baseline.  Adding the smallest depth-averaged contribution,
$p_{\bar\psi} = 0.02$, the lower bound is $1.24$.

\emph{Upper bound.}  Keep the assigned intensities, which already
weigh the increments at the least favourable intensities in the step,
and take the depth-averaged solve at the top of its range,
$p_{\bar\psi} = 0.05$, the value for a semi-implicit host on which
$\bar\psi$ costs a fifth of the free-surface iterations
(Appendix~\ref{app:cost_detail}); on a split-explicit host it is one
extra field in the existing sub-step loop and lies at the bottom of
the range.  Then $R = 1 + (0.94 + 2.50)/10.4 + 0.05 \approx 1.38$.

The estimate is therefore
\begin{equation}\label{eq:R_result}
  \boxed{\;R \;\approx\; 1.36, \qquad 1.24 \;\le\; R \;\le\; 1.38
  \quad\text{at } m = 1,\;}
\end{equation}
that is, AC/DC lengthens the hydrostatic time step by about a third.
Figure~\ref{fig:pass_budget} shows the budget behind this number and
where the bracket falls.
The bounds are not a confidence interval; they are the two weightings
the section possesses, operations and assigned traffic, with the
depth-averaged solve at either end of its range.  What is outside the
bracket is $m$: the bracket assumes the added routines reach the same
fraction of the bandwidth as the host's.  Every $10\%$ shortfall in
$m$ adds about $0.03$ to $R$, and $m$ is measurable routine by routine
with standard bandwidth profiling before a model is committed.  Two
implementation defects, and only these, carry the factor past $1.5$:
the depth-averaged pseudo-pressure communicating outside the
free-surface solve it shares, and $m$ well below one.  Both are
transient.

\begin{figure}[tbp]
  \centering
  \begin{tikzpicture}[
    font=\footnotesize,
    >={Stealth[length=1.8mm]},
    bar/.style={draw,thick},
    x=9mm,
  ]
  \def\ybone{0}
  \def\ybtwo{-13mm}
  \node[anchor=west,font=\scriptsize] at (0,\ybone+9mm) {(a) where the passes go};
  \node[anchor=east,font=\scriptsize] at (0,\ybone) {hydrostatic};
  \fill[cHost!18] (0,\ybone-3.5mm) rectangle (10.4,\ybone+3.5mm);
  \draw[bar]     (0,\ybone-3.5mm) rectangle (10.4,\ybone+3.5mm);
  \node[font=\scriptsize] at (5.2,\ybone) {$T_{\mathrm{hyd}} = 10.4$ passes};

  \node[anchor=east,font=\scriptsize] at (0,\ybtwo) {AC/DC};
  \fill[cHost!18] (0,\ybtwo-3.5mm) rectangle (10.4,\ybtwo+3.5mm);
  \draw[bar]     (0,\ybtwo-3.5mm) rectangle (10.4,\ybtwo+3.5mm);
  \node[font=\scriptsize] at (5.2,\ybtwo) {$T_{\mathrm{hyd}} = 10.4$ passes};
  \fill[cDyn!30]  (10.4,\ybtwo-3.5mm) rectangle (11.34,\ybtwo+3.5mm);
  \draw[bar,draw=cDyn] (10.4,\ybtwo-3.5mm) rectangle (11.34,\ybtwo+3.5mm);
  \fill[cPrs!25]  (11.34,\ybtwo-3.5mm) rectangle (13.84,\ybtwo+3.5mm);
  \draw[bar,draw=cPrs] (11.34,\ybtwo-3.5mm) rectangle (13.84,\ybtwo+3.5mm);

  \fill[cDyn!30] (0,\ybtwo-9mm) rectangle ++(2.6mm,2.6mm);
  \draw[cDyn]    (0,\ybtwo-9mm) rectangle ++(2.6mm,2.6mm);
  \node[anchor=west,font=\scriptsize] at (3.6mm,\ybtwo-7.7mm)
    {$\Delta T_{\mathrm{NH}} = 0.94$: non-hydrostatic dynamics, carried by any method};
  \fill[cPrs!25] (0,\ybtwo-14mm) rectangle ++(2.6mm,2.6mm);
  \draw[cPrs]    (0,\ybtwo-14mm) rectangle ++(2.6mm,2.6mm);
  \node[anchor=west,font=\scriptsize] at (3.6mm,\ybtwo-12.7mm)
    {$\Delta T_{\mathrm{acdc}} = 2.50$: the AC/DC pressure};

  \def\ax{-46mm}
  \draw[->] (0,\ax) -- (14.4,\ax);
  \foreach \r/\lab in {0/1.0, 2.768/1.1, 5.536/1.2, 8.304/1.3, 11.072/1.4, 13.84/1.5}{
    \draw (\r,\ax-1mm) -- (\r,\ax+1mm);
    \node[below,font=\scriptsize] at (\r,\ax-1mm) {\lab};}
  \fill[cPrs!25] (6.643,\ax+1.5mm) rectangle (10.518,\ax+4.5mm);
  \draw[cPrs]    (6.643,\ax+1.5mm) rectangle (10.518,\ax+4.5mm);
  \draw[very thick] (9.965,\ax+0.5mm) -- (9.965,\ax+6mm);
  \node[above,font=\scriptsize] at (9.965,\ax+6mm) {$1.36$};
  \node[anchor=west,font=\scriptsize] at (0,\ax+11mm)
    {(b) the resulting factor, bracket $1.24$ to $1.38$ at $m=1$};
  \node[below,font=\scriptsize] at (6.92,\ax-5.5mm) {runtime factor $R$};
  \end{tikzpicture}
  \caption{\it\small The pass budget and the factor it implies, for
    the configuration of Section~\ref{sec:pass_numbers} and the
    traffic column of Table~\ref{tab:intensity}.  Bar lengths in~(a)
    are memory traffic in mesh passes.  The two increments are of
    almost the same size, and only the second belongs to AC/DC: the
    non-hydrostatic dynamics are carried by any non-hydrostatic
    method, projection included, so a method obtaining the pressure at
    no cost at all would still land at $R\approx1.09$.  Panel~(b)
    marks the central value $1.36$ and the bracket $1.24$ to $1.38$
    of~\eqref{eq:R_result}, whose ends are the operation-weighted and
    the traffic-weighted readings of the same table with the
    depth-averaged solve at either end of its range.  The two panels
    have different horizontal scales.}
  \label{fig:pass_budget}
\end{figure}
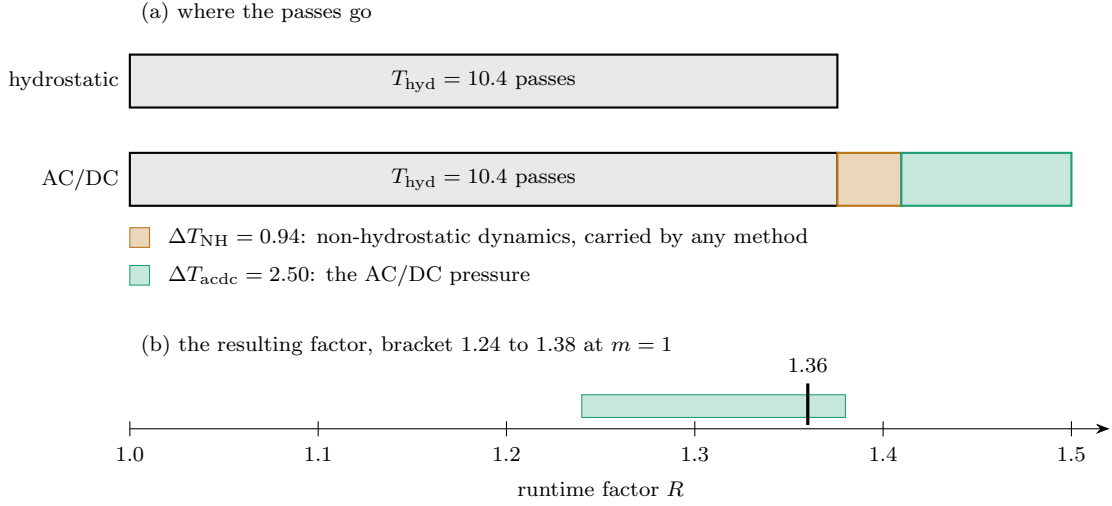

Three further readings of the same numbers.  The pressure machinery is
$0.88$ of the $9.7$ passes of operations of the AC/DC step, under a
tenth, and $\Delta T_{\mathrm{acdc}} = 2.50$ of its $13.8$ passes of
traffic, under a fifth; the second figure bounds what any improvement
of the pressure treatment can recover.  The dynamics increment
$\Delta T_{\mathrm{NH}}$ is not recoverable by any method, so a
hypothetical free pressure would still give $R \approx 1.09$.  Increments are fixed while the baseline grows with every tracer, each
tracer adding one transport pass and one tridiagonal solve: at five
tracers the factor is about $1.15$ in operations and $1.25$ in
traffic, and at the twenty to fifty tracers of coupled biogeochemistry
it falls to $1.03$--$1.06$ in operations and $1.07$--$1.11$ in
traffic.  The projection method, priced in the same unit, spends
$10^{2}$--$10^{3}$ passes and as many rounds of global communication
on the pressure alone, against AC/DC's $0.88$ passes and none.

In a coupled configuration where the ocean is not on the critical path
the factor can vanish from the wall clock, because the atmosphere sets
the pace.  AC/DC therefore runs at about $1.3$ to $1.4$ times the
hydrostatic wall-clock time at any resolution and machine size: a
scale estimate under the stated configuration and the measured
memory-bound regime of \citet{Klocke2025}, not a measurement of AC/DC
and not a theorem.

\section{Theoretical Analysis of AC/DC}\label{sec:ACDC_Bousinesq}\label{sec:properties}
Consider again the Boussinesq equations
\begin{equation}\begin{split}\label{Boussinesq_eq}
  \Dt\bv + \mathbf{N}_{\mathrm{3D}} 
  + \nabla B_{\mathrm{3D}}
    &= \mathcal{D}(\bv) + b\,\bk,\\[2ex]
 \Div\bv &= 0, \\
 \Dt b +\nabla\cdot(b\bv)=\Delta b,
\end{split}\end{equation}
where the rotational part of the nonlinearity is given by $\mathbf{N}_{\mathrm{3D}}:=\bom_a\times \bv$, with absolute vorticity $\bom_a:=\nabla\times\bv+\mathbf{f}$ and $\mathbf{f}$ the Coriolis vector,
the Bernoulli function is  $\nabla B_{\mathrm{3D}}:=\nabla (E_{\mathrm{kin}}+p)$, with $E_{\mathrm{kin}}$
the kinetic energy, $p$ the pressure, and $b$ denotes the buoyancy.

The physical plausibility of AC/DC rests on three quantitative
estimates, established for the Boussinesq equations in this section: the energy conservation
 (Section~\ref{sec:energy}), the volume-conservation error introduced by the artificial-compressibility
relaxation and its consequences for tracer conservation (Section~\ref{sec:hybrid_div}),
and the fidelity of the internal-gravity-wave dispersion relation
(Section~\ref{sec:hybrid_dispersion}). 

\subsection{Total Energy}\label{sec:energy}

In stratified flows the buoyancy is prognostic, and its work exchanges
kinetic and potential energy.
The weighted budget of~\eqref{eq:KL_energy_inviscid} produces the pairing
$\langle({\rhoa}/\rho_0)\,b,\,w\rangle$, and the question is
which reservoir absorbs it.  It has an answer valid for an
arbitrary equation of state, so we give that first and describe then the
familiar forms as special cases.

\paragraph*{Transport identity.}
If $r := {\rhoa}/\rho_0$ satisfies the flux
form~\eqref{eq:ac_flux} then for any field $F$, on a closed domain,
\begin{equation}\label{eq:transport_identity}
  \frac{d}{dt}\int_\Omega r\,F\dd V
    = \int_\Omega r\,\bigl(\Dt F + \bv\cdot\nabla F\bigr)\dd V ,
\end{equation}
because $\int F\,\Dt r = -\int F\,\Div(r\bv) = \int r\,\bv\cdot\nabla F$.
This is the mechanism used below.

\paragraph*{The Potential Energy.}
Let the buoyancy be diagnosed from an equation of state,
$b = b(\Theta,S,z)$, in temperature, salinity and depth, and let
$\Theta$ and $S$ be transported in flux form against the same volume
fluxes that carry $r$, so that by the argument
of~\eqref{eq:transport_identity} applied to each tracer they are
advected, $\Dt\Theta + \bv\cdot\nabla\Theta = 0$ and likewise for $S$.
Define
\begin{equation}\label{eq:pe_def}
  \Phi(\Theta,S,z) := \int_{z_{\mathrm{ref}}}^{z}
     b(\Theta,S,z')\,\dd z' ,
  \qquad
  \boxed{\;\mathrm{PE} := -\int_\Omega
     \frac{{\rhoa}}{\rho_0}\,
     \Phi(\Theta,S,z)\dd V \; }.
\end{equation}
Since $z$ is a fixed coordinate and $\partial\Phi/\partial z|_{\Theta,S} = b$,
\begin{equation}
  \Dt\Phi + \bv\cdot\nabla\Phi
    = \Phi_\Theta\!\cdot\!0 + \Phi_S\!\cdot\!0 + \Phi_z\,w = b\,w ,
\end{equation}
so~\eqref{eq:transport_identity} gives the conversion
\begin{equation}\label{eq:ke_pe_conversion}
  \frac{d\,\mathrm{PE}}{dt}
    = -\Bigl\langle \frac{{\rhoa}}{\rho_0}\,b,\;w\Bigr\rangle,
\end{equation}
and this holds for any equation of state.
The weighting is
built in, so the pairing produced by the momentum equation is absorbed
whole and no $\alpha^{-1}\langle\psi b,w\rangle$ mismatch arises.

Two densities appear here:  the pseudo-density
${\rhoa}$ of~\eqref{eq:rho_ac} responds to pressure
relaxation alone; it is a volume variable, and $\int_\Omega{\rhoa}\dd V$
is conserved by~\eqref{eq:ac_flux}, and  the density entering $b$ is the
thermodynamic $\rho(\Theta,S,z)$, which in a Boussinesq model is used
only in the buoyancy term, and whose domain integral is
not conserved.

\paragraph*{Total energy.}
With~\eqref{eq:pe_def} the conserved density-weighted total energy is
\begin{equation}\label{eq:EKL_def}
  \boxed{\;
  E_{\mathrm{KL}} := k_{\mathrm{dw}}
  + \frac{1}{2\alpha\rho_0}\!\int_\Omega\psi^2\,\dd V
  + \mathrm{PE}
  + \mathrm{PE}_{\mathrm{sfc}},
  \qquad
  \mathrm{PE}_{\mathrm{sfc}} := \frac{g}{2}\!\int_A \eta^2\,\dd A\;},
\end{equation}
the last term being the barotropic surface reservoir required by the
free surface and derived in
Proposition~\ref{prop:free_surface}(iii); it vanishes identically on a
fixed domain, $\eta\equiv0$. 
Starting from the inviscid
budget~\eqref{eq:KL_energy_inviscid}, using
$\Div({\rhoa}\bv/\rho_0) = -\alpha^{-1}\Dt\psi$
from~\eqref{eq:ac_flux} and the conversion~\eqref{eq:ke_pe_conversion},
\begin{equation}\label{eq:EKL_budget}
  \frac{dE_{\mathrm{KL}}}{dt}
  = -\frac{1}{\alpha}\Bigl\langle
      \frac{p_{\mathrm{hyd}}+p_{\mathrm{sfc}}}{\rho_0},\;\Dt\psi\Bigr\rangle
    \;\underbrace{-\;\frac{1}{\alpha}\Bigl\langle
      \frac{p_V}{\rho_0},\;\Dt\psi\Bigr\rangle}_{\text{AC/DC only}}
    - (\text{dissipation}).
\end{equation}
Two terms remain in AC/DC, and one in pure AC.  Both are of the same
kind: the work of a pressure
against the artificial compressibility.  The first is the
hydrostatic and surface pressures, the second is present only when the
column solve is used, is  the column pressure.

\paragraph*{Discrete Energy Budget}\label{sec:disc_energy}

The energy~\eqref{eq:dw_energy} with density weighting is defined as the following sum over faces
\begin{equation}\label{eq:kdw_h}
  K_h := \tfrac12\sum_f r_f\,u_f^2\,|f|\,|f^*| ,
  \qquad r_f := \langle r\rangle_f ,
\end{equation}
to each face we assign the volume $|f|\,|f^*|$ of the diamond
spanned by the primal and the dual face.

\paragraph*{The discrete total energy and its budget.}
With $K_h$ defined, the remaining components of~\eqref{eq:EKL_def} discretise
on the cells, and the discrete counterpart of
the conserved total is
\begin{equation}\begin{split}\label{eq:EKL_h}
  &E_h := K_h
       + \frac{1}{2\alpha\rho_0}\sum_{C\in\mathcal{C}} |C|\,\psi_C^2
       + \mathrm{PE}_h
       + \mathrm{PE}_{\mathrm{sfc},h},\\
\text{with }&  \mathrm{PE}_h := -\sum_{C\in\mathcal{C}} |C|\,r_C\,\Phi_C,
 \quad \text{and }\quad
 \mathrm{PE}_{\mathrm{sfc},h} := \tfrac{g}{2}\sum_{C\in\mathcal{C}_{\mathrm{sfc}}} |C|\,\eta_C^2
\end{split}\end{equation}
The elastic norm is the
volume-weighted one, $\|\psi\|_h^2 = \sum_{C\in\mathcal{C}}|C|\psi_C^2$: it is the metric in
which the discrete gradient and divergence are adjoint to one another,
so that the pseudo-pressure and the velocity form a conjugate pair.

Every term of~\eqref{eq:EKL_budget} has its discrete counterpart, and $dE_h/dt$ is balanced as follows
\begin{equation}\label{eq:EKL_budget_h}
  \frac{dE_h}{dt}
  + \frac{1}{\alpha}\Bigl\langle
      \frac{p_{\mathrm{hyd}}+p_{\mathrm{sfc}}}{\rho_0},\,\Dt\psi\Bigr\rangle_h
  + \frac{1}{\alpha}\Bigl\langle \frac{p_V}{\rho_0},\,\Dt\psi\Bigr\rangle_h
  + \mathcal{D}_h
  \;=\; 0 ,
\end{equation}
$\mathcal{D}_h$ the discrete dissipation and the inner product on prisms is defined as
$\langle x,y\rangle_h := \sum_{C\in\mathcal{C}} |C|\,x_C y_C$.  
Equation~\eqref{eq:EKL_budget_h}
states that nothing  remains, in particular no cubic term,
which is Proposition~\ref{prop:dw_removes}(iii) transferred to the discrete
setting, and no contribution from the Lamb term. All terms are diagnosable.

\subsection{Error Analysis}\label{sec:errors}

\subsubsection{Divergence and volume}\label{sec:hybrid_div}
The divergence error, equivalently the volume-conservation error,
 quantifies how well AC/DC approximates
incompressibility.

With the
boundary conditions of Section~\ref{sec:acdc_closure}, Dirichlet at
$z=\eta$, Neumann at the bottom, the eigenfunctions of $-L_z$ on a
column of depth $H$ are $\sin(k_nz)$ with
\begin{equation}\label{eq:kz_spectrum}
  k_n = \frac{(2n+1)\pi}{2H},
  \qquad n = 0,1,2,\dots,
  \qquad k_{\min} = \frac{\pi}{2H},
\end{equation}
on which $L_z$ acts as multiplication by $-k_n^2$.  The
vertical wavenumber is 
bounded away from zero: the vertically constant mode is not in
the column solve's spectrum at all, which is why the expressions below
never encounter the singularity their $k_z\to0$ form would suggest.
Writing $k_z$ for $k_n$ and treating it as continuous at grid scale,
the column solve gives
$p_V/\rho_0 = S'/(k_z^2)$ and the non-hydrostatic pressure is
$p_{\mathrm{NH}}/\rho_0 = S'/(k_x^2+k_z^2)$.  AC/DC
pseudo-pressure converges to the residual
$\psi \to p_{\mathrm{NH}} - p_V$:
\begin{equation}
  |\psi_{\mathrm{hyb}}|
  = |p_{\mathrm{NH}} - p_V|
  = \left|\frac{S'}{k_x^2+k_z^2} - \frac{S'}{k_z^2}\right|
  = \frac{k_x^2}{k_z^2}\,|p_{\mathrm{NH}}|.
  \label{eq:psi_ratio}
\end{equation}
Since the AC divergence error is proportional to $|\psi|/\alpha$, the
AC/DC divergence error is smaller than the AC error without column solve by
the wavenumber-dependent factor $k_x^2/k_z^2$:
\begin{equation}
  \frac{|\Div\bv|_{\mathrm{hyb}}}{U/L}
  \;\sim\; \frac{k_x^2}{k_z^2}\,\delta^2\,\Fr_{\mathrm{baro}}^2 .
  \label{eq:hybrid_div_error}
\end{equation}
Equation~\eqref{eq:psi_ratio} is exact on each mode;
\eqref{eq:hybrid_div_error} is a scaling estimate, obtained from
$|\Div\bv|\sim|\Dt\psi|/\alpha$ with the scales of
Section~\ref{sec:ac}, and is written with $\sim$ for that reason.
At the grid scale ($k_x \sim 1/\Delta x$,
$k_z \sim 1/\Delta z$): the reduction factor is
$(\Delta z/\Delta x)^2$.  On a mesh with
e.g.$\Delta x/\Delta z = 100$, AC/DC divergence error is
$10^{-4}$ times the AC-error, i.e. without the column solve.

A projection method does not produce a
divergence-free velocity either: it produces one whose divergence has
been reduced to whatever tolerance the iteration was stopped at.
\citet{Marshall1997b} report terminating their pressure inversion when
the divergence has been reduced to one part in $10^4$, a level they
found sufficient for numerical stability.  The relevant question for
AC/DC is therefore not whether its divergence error is zero, because no
practical method delivers that, but whether it is small enough.

The column solve captures the dominant component of
$p_{\mathrm{NH}}$, because the vertical eigenvalues
$\sim 1/\Delta z^2$ are the largest, leaving only a small
residual for the artificial compressibility step.  A smaller $\psi$ means a smaller AC error.
The two benefits, namely smaller error and relaxed CFL, arise from the
same fact: the vertical modes are the stiffest, and
handling them exactly removes both the stiffest CFL constraint and
the largest contribution to the approximation error.  On
isotropic meshes ($\Delta z = \Delta x$), the factor is unity and
AC/DC offers no advantage.

\begin{remark}[Regime of validity: stratification, not
  slenderness]\label{rem:aspect_ratio}
  We have above evaluated the factor $k_x^2/k_z^2$ at 
  grid scale, $k_x/k_z\sim\Delta z/\Delta x$, which is the right
  reading for grid-scale noise.  For a coherent structure the relevant
  ratio is set by the pressure field instead.  \citet{Marshall1997b} show that whether non-hydrostatic
  dynamics matters depends on  
  \begin{equation}\label{eq:marshall_n}
    \mathrm{n} = \frac{\gamma^2}{R_i},
    \qquad \gamma = \frac{h}{L},
    \qquad R_i = \frac{N^2h^2}{U^2},
    \qquad\text{so that}\qquad
    \mathrm{n} = \frac{U^2}{L^2N^2},
  \end{equation}
  in which the height $h$ cancels identically.  The regime is fixed by
  the horizontal scale measured against the stratification.
  The condition
  fails, as \citet{Marshall1997b} emphasise, in weakly stratified water
  on small horizontal scales, the case of  open-ocean deep
  convection \citep{JonesMarshall1993}.

  For AC/DC the pertinent question is which vertical scale the non-hydrostatic pressure represents, since the column solve is accurate when $k_z^2\gg k_x^2$.  In
  a convecting region the non-hydrostatic pressure is coupled over the
  full depth $D$ of the weakly stratified layer, so $k_z\sim\pi/D$ is
  set by the layer, while $k_x\sim\pi/L$ is set
  by the plume width.  The relevant ratio is therefore $D/L$:
  \begin{equation}\label{eq:DL_ratio}
    \frac{k_x^2}{k_z^2} \;\sim\; \left(\frac{D}{L}\right)^{\!2}
    \qquad\text{(convecting layer of depth $D$, plumes of width $L$).}
  \end{equation}
  \begin{center}\small
  \begin{tabular}{lrrr}
    \toprule
    & $D$ (m) & $L$ (m) & $(D/L)^2$ \\
    \midrule
    Deep convection, wide chimney   & 2000 & 10000 & $0.04$ \\
    Deep convection, plume          & 2000 & 1000  & $4$ \\
    Deep convection, narrow plume   & 2000 & 200   & $100$ \\
    Shallow mixed layer, plume      & 100  & 100   & $1$ \\
    Shallow mixed layer, wide       & 100  & 1000  & $0.01$ \\
    \bottomrule
  \end{tabular}
  \end{center}
  \noindent
  The column solve retains its advantage while the convecting layer is
  shallow compared with the width of its plumes, and loses it, when narrow 
  plumes span a deep weakly stratified layer.  The
  seasonal and shallow mixed layers that cover most of the ocean fall in
  the benign column; deep convection with narrow plumes does not.  Where
  the factor exceeds unity the remedy is~\eqref{eq:alpha_choice}: the
  divergence error falls as $1/\alpha$, so raising $\alpha$ by the same
  factor restores the accuracy at the cost of a proportionally shorter
  sub-step, which remains far below the cost of an elliptic solve.
\end{remark}
\subsubsection{Tracer Transport}\label{sec:tracer_transport}
The residual divergence is not equally harmless for all tracers.  A
nonzero $\Div\bv$ acts as a spurious local source in the
conservative tracer equation. For temperature
this is potentially masked by surface restoring and large effective
diffusivity. 

For \emph{salinity} the situation is different: ocean salinity has no
interior sources, the basin-integrated salt content is expected to
be conserved to high precision over climate-length integrations,
and salinity feeds back on density and hence on the dynamics.  An
uncontrolled spurious source that integrated into a slow drift of
total salt would be a serious defect.

The pseudo-density form of Section~\ref{sec:ac_system} shows how
this is avoided. Equation~\eqref{eq:ac_flux} is an
exact conservation law; building the tracer transport on it gives the following.
The statement is made for a general prismatic mesh whose cell volumes
may move, so that it covers $z$- and $z^*$-coordinates alike: $|c|(t)$
is the volume of cell $c$, $|\Omega|(t)=\sum_c|c|$, interior faces $f$
carry the flux $F_f$ \emph{relative to the moving faces} with
orientations $\sigma_{c,f}=-\sigma_{c',f}$, and $F_f=0$ on boundary
faces.  One point of naming must be fixed here, because the two
readings give different schemes.  The flux $F_f$ that
appears in~\eqref{eq:disc_cont_tracer} is the \emph{pseudo-mass} flux
\begin{equation}\label{eq:Ff_def}
  F_f = r_f\,u_f\,|f| ,\quad \text{ with }\quad r={\rhoa}/\rho_0
\end{equation}
with $u_f$ the face-normal velocity relative to the moving face and
$r_f$ the face value of $r={\rhoa}/\rho_0$: it is the flux
whose divergence advances $|c|r_c$, which is the discrete form
of~\eqref{eq:ac_flux}.  It is not the volumetric flux $u_f|f|$; a
tracer transported against the volumetric flux would satisfy the
relaxation form~\eqref{eq:ac_cont} instead, and
Proposition~\ref{prop:tracer_consistency} and
Proposition~\ref{prop:variance} would be lost.  

\begin{proposition}[Consistent tracer transport under AC/DC]
  \label{prop:tracer_consistency}
 Define $M:=\sum_c|c|r_c$, $\Sigma:=\sum_c|c|r_cC_c$,
  $\bar C_w := \Sigma/M$,
  $\bar C_{\mathrm{vol}} := |\Omega|^{-1}\sum_c|c|C_c$,
  $C' := C-\bar C_{\mathrm{vol}}$ and $m := M-|\Omega|$.

  Suppose the discrete tracer transport satisfies:
  \begin{enumerate}[nosep]
    \item[(H1)] Consistency with pseudo-continuity: the tracer and the
      pseudo-density are advanced with the same discrete fluxes $F_f$
      and the same face values, in the flux forms
      \begin{equation}\label{eq:disc_cont_tracer}
        \frac{d}{dt}\bigl(|c|\,r_c\bigr)
          = -\!\!\sum_{f\in\partial c}\!\sigma_{c,f}F_f ,
        \qquad
        \frac{d}{dt}\bigl(|c|\,r_c C_c\bigr)
          = -\!\!\sum_{f\in\partial c}\!\sigma_{c,f}F_f\,C_f ,
      \end{equation}
      the continuum counterpart being
      $\Dt({\rhoa} C)
       + \nabla\!\cdot({\rhoa} C\bv) = 0$;
    \item[(H2)] Preservation of constants: the face reconstruction of $C$ is
      constancy-preserving;
    \item[(H3)] Time integration: the tracer and ${\rhoa}$ use the same time
      discretisation, and any flux limiter is applied so as to
      preserve~(H1);
    \item[(H4)] Positivity: $r_c>0$ for all $c$ and all $t$.
  \end{enumerate}
  Then
  \begin{enumerate}[nosep]
  \item[(i)] $M$ and $\Sigma$, equivalently the pseudo-density-weighted
    content $\int_\Omega{\rhoa}C\dd V$ are conserved, for {any} face reconstruction $C_f$;
  \item[(ii)] constant tracer fields are preserved;
  \item[(iii)] the physical content $\int_\Omega\rho_0C\dd V$ is
    conserved up to a bounded $\OO(\Fr_{\mathrm{baro}}^2)$
    error, not a secular drift;
  \item[(iv)] the representation error obeys
    \begin{equation}\label{eq:tracer_repr_bound}
      \bigl|\bar C_{\mathrm{vol}}(t)-\bar C_w\bigr|
        = \frac{\bigl|\langle\psi,C'\rangle_h\bigr|}{\alpha M}
        \le \frac{\bigl(2|\Omega|+|m|\bigr)\,\|C'(t)\|_\infty}{M} ,
    \end{equation}
    so that with $r(0)\equiv1$ and a monotone transport
    $|\bar C_{\mathrm{vol}}(t)-\bar C_w| \le 2\|C'(0)\|_\infty$ for all
    $t$, with no growth in time.  If in addition $\|\psi\|_h$ is
    bounded 
    the bound improves to
    $\OO(\alpha^{-1/2})$.
  \end{enumerate}
\end{proposition}

\begin{proof}
  See Appendix~\ref{app:proofs}.
\end{proof}

\begin{remark}\label{rem:bounded_error}
  The proposition shows:  $M$,
  $\Sigma$ and $\bar C_w$ are conserved, i.e.  total salt shows no drift of any size, for any run length, any reconstruction,
  and
  any vertical  coordinate in the $z/z^*$ family.  The representation error is
  bounded uniformly in time by~(iv), by a constant fixed by the
  initial tracer spread and with no growth in $t$; this
rests only on positivity and exact conservation.
  The energy statements of Section~\ref{sec:energy} add the
  size of that constant, $\OO(1)\to\OO(\alpha^{-1/2})$.  
  The
  density-free kinetic energy supports none of this: its budget does
  not close, it is not a sum of non-negative reservoirs, and no bound
  on $\psi$ follows from it at all.
\end{remark}

\subsection{Dispersion relations}\label{sec:hybrid_dispersion}

How  AC/DC reproduces the exact non-hydrostatic
inertia--gravity-wave dispersion is described on
the $f$-plane.  Setting $f = 0$ recovers the non-rotating case, which by
part~(iv) of Proposition~\ref{prop:hybrid_dispersion} is the worst case
for the error bounds.\\

{\bf Linearisation.}
Linearise the AC/DC system in a stratified Boussinesq fluid at rest on
the $f$-plane with the traditional Coriolis force ($\tilde f = 0$),
fields $\propto e^{i(k_x x + k_y y + k_z z - \omega t)}$.  By the
horizontal isotropy of the traditional $f$-plane system one may take
$k_y = 0$ without loss of generality, $k_x$ then denoting the magnitude
of the horizontal wave vector.  The prognostic variables
$(\hat u,\hat v,\hat w,\hat b,\tilde\psi)$ evolve with the column
pressure $\tilde p_V$ determined diagnostically.

The column-solve source is the divergence of all terms except the pressure, 
here the Coriolis and buoyancy terms,
$S = ik_x f\hat v + ik_z\hat b$, so that the column constraint
$-k_z^2\tilde p_V = S$ gives
\begin{equation}
  \tilde p_V = -\,i\,\frac{k_x f\hat v + k_z\hat b}{k_z^{2}} ,
  \qquad\text{and at } f = 0:\quad \tilde p_V = -i\hat b/k_z .
  \label{eq:pV_diag}
\end{equation}
With total pressure $\tilde p_V + \tilde\psi$ the linearised system reads
\begin{alignat}{2}
  -i\omega\,\hat u &= f\hat v \;-\; \frac{k_x}{k_z^{2}}
      \bigl(k_x f\hat v + k_z\hat b\bigr) &&{}-\, ik_x\tilde\psi,
    \label{eq:hyb_u}\\
  -i\omega\,\hat v &= -f\hat u, && \label{eq:hyb_v}\\
  -i\omega\,\hat w &= -\frac{k_x}{k_z}\,f\hat v &&{}-\, ik_z\tilde\psi,
    \label{eq:hyb_w}\\
  -i\omega\,\hat b &= -N^{2}\hat w, && \label{eq:hyb_b}\\
  -i\omega\,\tilde\psi/\tilde\alpha
    &= -(ik_x\hat u &&{}+\, ik_z\hat w).
    \label{eq:hyb_psi}
\end{alignat}

The buoyancy contribution to the vertical
momentum equation is cancelled by the column pressure,
$-ik_z\tilde p_V$ contributing $-\hat b$ against the $+\hat b$ of the
buoyancy force; the column solve absorbs the vertical
hydrostatic--buoyancy balance and leaves only
$\tilde\psi$ behind.  This cancellation is not an artefact of
$f = 0$: it holds with rotation, and the Coriolis
remainder $-(k_x/k_z)f\hat v$ survives.  In the horizontal momentum equation, by
contrast, the column pressure contributes a horizontal buoyancy
pressure gradient that the column solve does not eliminate.

\begin{proposition}[AC/DC dispersion accuracy]
  \label{prop:hybrid_dispersion}
  \label{prop:hybrid_dispersion_rot}
  Let $K^{2} = k_x^{2}+k_z^{2}$ and let
  $\omega_0^{2} = (N^{2}k_x^{2} + f^{2}k_z^{2})/K^{2}$ be the
  non-hydrostatic inertia-gravity-wave frequency.  Then:

  \emph{(i) Dispersion relation.}  The AC/DC system
  \eqref{eq:hyb_u}--\eqref{eq:hyb_psi} has the biquadratic
  \begin{equation}
    \omega^{4}
    - \Bigl[\tilde\alpha K^{2}
        + f^{2}\Bigl(1 - \tfrac{k_x^{2}}{k_z^{2}}\Bigr)\Bigr]\omega^{2}
    + \tilde\alpha\bigl(N^{2}k_x^{2} + f^{2}k_z^{2}\bigr)
    - \tfrac{k_x^{2}}{k_z^{2}}\,N^{2}f^{2}
    = 0,
    \label{eq:hybrid_dispersion_rot}
  \end{equation}
  whereas pure AC, i.e.\ without the column solve, gives
  $\omega^{4} - [\tilde\alpha K^{2} + N^{2} + f^{2}]\omega^{2}
   + \tilde\alpha(N^{2}k_x^{2} + f^{2}k_z^{2}) + N^{2}f^{2} = 0$.
  The fast branch of~\eqref{eq:hybrid_dispersion_rot} is
  $\omega^{2} = \tilde\alpha K^{2} + \OO(1)$, as for pure AC.

  \emph{(ii) Error of the slow branch.}  For $k_z \neq 0$,
  \begin{equation}
    \frac{\omega^{2}_{\mathrm{AC/DC}} - \omega_0^{2}}{\omega_0^{2}}
    = +\,\frac{k_x^{4}\,(N^{2}-f^{2})^{2}}
             {\tilde\alpha\,K^{6}\,\omega_0^{2}}
      + \OO(\tilde\alpha^{-2}),
    \qquad
    \frac{\omega^{2}_{\mathrm{AC}} - \omega_0^{2}}{\omega_0^{2}}
    = -\,\frac{k_x^{2}k_z^{2}\,(N^{2}-f^{2})^{2}}
             {\tilde\alpha\,K^{6}\,\omega_0^{2}}
      + \OO(\tilde\alpha^{-2}).
    \label{eq:rot_disp_errors}
  \end{equation}
  Both methods therefore reproduce the dispersion  at leading
  order in $1/\tilde\alpha$; the ratio of their errors is
  $k_x^{2}/k_z^{2}$ in magnitude, independently of $f$ and $N$, with
  the AC/DC error lying above the exact frequency and the pure-AC
  error below.

  \emph{(iii) Non-rotating limit.}  At $f = 0$,
  \eqref{eq:hybrid_dispersion_rot} reduces to
  \begin{equation}
    \omega^{4} - \tilde\alpha K^{2}\,\omega^{2}
      + \tilde\alpha\,k_x^{2}N^{2} = 0,
    \qquad
    \omega^{2}_{\mathrm{slow}}
      = \frac{k_x^{2}N^{2}}{K^{2}}
        + \frac{k_x^{4}N^{4}}{\tilde\alpha K^{6}}
        + \OO(\tilde\alpha^{-2}),
    \label{eq:hybrid_dispersion}
  \end{equation}
  the leading term being the exact non-hydrostatic relation, and
  \eqref{eq:rot_disp_errors} reduces to
  \begin{equation}
    \frac{|\omega^{2}_{\mathrm{AC/DC}} - \omega^{2}_{\mathrm{NH}}|}
         {\omega^{2}_{\mathrm{NH}}}
    = \frac{k_x^{2}N^{2}}{\tilde\alpha K^{4}}
    = \frac{k_x^{2}}{k_z^{2}}\cdot\frac{k_z^{2}N^{2}}{\tilde\alpha K^{4}},
    \label{eq:hybrid_disp_error}
  \end{equation}
  the second factor being the pure-AC error.  The two biquadratics
  differ only by the absence of the $N^{2}$ term in the
  $\omega^{2}$-coefficient: the column solve removes the buoyancy
  coupling from the vertical momentum, eliminating the $N^{2}$
  contribution to the acoustic branch and improving the gravity-wave
  branch.

  \emph{(iv) Effect of rotation.}  At $k_x = 0$,
  \eqref{eq:hybrid_dispersion_rot} factorises as
  $(\omega^{2}-f^{2})(\omega^{2}-\tilde\alpha k_z^{2}) = 0$: pure
  inertial oscillations are reproduced at any $\alpha$.  If
  $f^{2}\le 2N^{2}$, satisfied when the water column is
  stratified, $N\gtrsim f$; it can fail in near-neutral convecting
  layers, where $N<f$ is possible, then rotation only \emph{decreases} both
  errors relative to their non-rotating values, so
  \eqref{eq:hybrid_disp_error} bounds~\eqref{eq:rot_disp_errors}.
  Without that hypothesis the conclusion fails.
\end{proposition}

\begin{proof}
  See Appendix~\ref{app:proofs}.
\end{proof}

\begin{remark}[Where the improvement holds, and where it does not]
  The error ratio $k_x^{2}/k_z^{2}$ of part~(ii) favours AC/DC on the near-hydrostatic
  sector $k_z \gg k_x$, in which the frequency reduces to the
  hydrostatic $N^{2}k_x^{2}/k_z^{2}$, and which is what an anisotropic ocean mesh predominantly has.  On the genuinely
  non-hydrostatic branch $k_x > k_z$ the factor exceeds unity: here
  pure AC is the more accurate of the two, and exact at
  $k_z = 0$.  This is not an artefact of the
  estimate, the grid anisotropy that makes the vertical modes stiff
  is what the column solve exploits, and it is absent when
  $k_x \sim k_z$.

  Two remarks on the impact of this limitation.  First, the absolute error stays
  small even at its worst.  By part~(iv) it suffices to bound the
  non-rotating expression~\eqref{eq:hybrid_disp_error},
  $k_x^{2}N^{2}/(\tilde\alpha K^{4})$.  Writing $s = k_x^{2}$, the
  factor $s/(s+k_z^{2})^{2}$ has derivative
  $(k_z^{2}-s)/(s+k_z^{2})^{3}$, so at fixed $k_z$ it is largest at
  $k_x = k_z$ and equals $1/(4k_z^{2})$; since the admissible
  vertical wavenumbers are bounded below by~\eqref{eq:kz_spectrum}, the
  maximum over the whole admissible band is attained at
  $k_x = k_z = k_{\min} = \pi/2H$ and equals
  \begin{equation}\label{eq:disp_sup}
    \sup\;\frac{\omega^2_{\mathrm{AC/DC}}-\omega_0^2}{\omega_0^2}
      \;=\;\frac{N^{2}}{4\,\tilde\alpha\,k_{\min}^{2}} .
  \end{equation}
  An example, for $N = 10^{-2}\,\mathrm{s^{-1}}$,
  $\tilde\alpha = gH_{\mathrm{full}}$ and $H = 4000$\,m this is
  $4\times10^{-3}$ in $\omega^{2}$, hence about $2\times10^{-3}$ in
  frequency, at a horizontal wavelength of some $16$\,km.  The worst case therefore falls in the low-mode internal-wave band and not at
  grid scale, where the same expression is smaller by four orders
  of magnitude; even there it remains below the spatial truncation
  error measured in Section~\ref{sec:exp_igw}.  Second, the vertically constant mode
  $k_z = 0$ does not arise here at all: it is excluded from the column
  solve's spectrum by the same boundary conditions that make $L_z$ invertible.
\end{remark}

\section{Implementation of AC/DC}\label{sec:disc_time_impl}\label{sec:disc_time_boussinesq}
This section states the time discretisation.
Section~\ref{sec:templates} discretises classical artificial
compressibility and shows why it stops at ocean resolutions: its
velocity equation is a global three-dimensional system on every step.
Section~\ref{sec:acdc_step} states the AC/DC step, whose acoustic
stage~S3 is the one place the global system could reappear and where it
is instead relaxed; Sections~\ref{sec:variants2}
and~\ref{sec:variants3} give the two concrete forms of that stage.
What remains per sub-step, in either form, is nearest-neighbour stencil
work and at most a solve local to a single column, the access pattern
throughput-oriented hardware is built for.

\subsection{Classical Artificial Compressibility}\label{sec:templates}

Sections~\ref{sec:ac_system} and~\ref{sec:acdc_closure} settled the
time-continuous systems.  This subsection discretises classical
artificial compressibility in time on a \emph{single} interval
$\Delta t$ and states it as Template~A, the reference against which the
AC/DC step is measured.

With a scalar viscosity and no directional treatment, the one-scale
system is discretised exactly as in \citet[\S2.3]{Guermond2015}.  Write
$\mathcal{A}$ for the viscous operator and $\mathbf{R}^{n}$ for the
non-pressure, non-viscous tendency, advection, Coriolis, buoyancy, the
hydrostatic Bernoulli gradient, evaluated explicitly, and
\begin{equation}\label{eq:GM_dictionary}
  \varepsilon := \frac{1}{\tilde\alpha} = \frac{\rho_0}{\alpha}
\end{equation}
for the artificial-compressibility parameter, which is the dictionary
between the present notation and that of \citet{Guermond2015}.  The
step is a coupled velocity--pressure pair, \citet[eq.~(2.6)]{Guermond2015},
\begin{subequations}\label{eq:AC_BE}
\begin{align}
  \frac{\bv^{n+1}-\bv^{n}}{\Delta t}
    + \mathcal{A}\bv^{n+1}
    + \nabla\frac{\psi^{n+1}}{\rho_0}
    &= \mathbf{R}^{n} ,
    \label{eq:AC_BE_mom}\\[2pt]
  \frac{\varepsilon}{\Delta t}\,
    \frac{\psi^{n+1}-\psi^{n}}{\rho_0}
    + \Div\bv^{n+1} &= 0 .
    \label{eq:AC_BE_psi}
\end{align}
\end{subequations}
The two are still coupled, and the coupling is removed without
merging them: solve~\eqref{eq:AC_BE_psi} for $\psi^{n+1}$ and
substitute into~\eqref{eq:AC_BE_mom}.
What is left is again a velocity
system and a pressure system, executed in that order, which
is \citet[eq.~(2.7)]{Guermond2015},
\begin{subequations}\label{eq:AC_BE_split}
\begin{align}
  \frac{\bv^{n+1}-\bv^{n}}{\Delta t}
    + \mathcal{A}\bv^{n+1}
    - \frac{\Delta t}{\varepsilon}\,\nabla\Div\bv^{n+1}
    + \nabla\frac{\psi^{n}}{\rho_0}
  &= \mathbf{R}^{n} ,
  \qquad \bv^{n+1}\!\cdot\!\mathbf{n}\big|_\Gamma = 0 ,
  \label{eq:AC_BE_vel}\\[4pt]
  \frac{\psi^{n+1}}{\rho_0}
  &= \frac{\psi^{n}}{\rho_0}
      - \frac{\Delta t}{\varepsilon}\,\Div\bv^{n+1} .
  \label{eq:AC_BE_pres}
\end{align}
\end{subequations}


\begin{remark}[Where the coefficient $\Delta t/\varepsilon$ separates the
  two settings]\label{rem:coefficient}
  In \citet{Guermond2015} the parameter is tied to the step,
  $\varepsilon\sim\Delta t$, so that $\Delta t/\varepsilon=\OO(1)$
  and~\eqref{eq:AC_BE_vel} reads
  $\bv + \Delta t(\mathcal{A}\bv - \nabla\Div\bv) = \mathbf{r}$: the
  grad--divergence term is a mild stabilisation of viscous strength, and
  it is in that regime that their conditioning estimate
  $\OO(\Delta t\,h^{-2})$ is stated.  Here $\alpha$ is fixed
  by~\eqref{eq:alpha_choice} independently of the step, so
  \begin{equation}\label{eq:coefficient_ratio}
    \frac{\Delta t}{\varepsilon} = \tilde\alpha\,\Delta t
      = c_{\mathrm{ac}}^{2}\,\Delta t ,
    \qquad
    \frac{\Delta t}{\varepsilon}\cdot\frac{\Delta t}{\ell^{2}}
      = \Bigl(\frac{c_{\mathrm{ac}}\Delta t}{\ell}\Bigr)^{2}
      = \Bigl(\frac{\Delta t}{\Delta t^{\mathrm{ac}}_{\ell}}\Bigr)^{2}
  \end{equation}
  on a cell of size $\ell$: the dimensionless strength of the
  grad--divergence term is the squared ratio of the baroclinic
  step to the acoustic transit time of that cell. 
\end{remark}

Equation~\eqref{eq:AC_BE_split} has five properties, all established
by \citet{Guermond2015}.
\begin{enumerate}[nosep]
\item \emph{No Poisson problem.}  The pressure is updated by one
  divergence evaluation.  Nothing elliptic and scalar is solved at any
  point in the step.
\item \emph{Symmetry and definiteness.}  Multiplying
  \eqref{eq:AC_BE_vel} by $\Delta t$ gives
  $\bigl(I+\Delta t\mathcal{A}
    -(\Delta t^{2}/\varepsilon)\nabla\Div\bigr)\bv^{n+1}=\mathbf{r}$,
  and since $\mathcal{A}$ is $H^1$-coercive and
  $\langle-\nabla\Div\bv,\bv\rangle = \|\Div\bv\|^{2}\ge0$,
  \eqref{eq:AC_BE_vel} is amenable to conjugate gradients and has no saddle-point structure.
\item \emph{Order.}  First order in $\Delta t$ and first order in
  $\tilde\alpha^{-1}$, by the analysis of \citet{Guermond2015}.
\item \emph{No boundary condition on $\psi$.}  No artificial boundary
  condition on $\psi$ is required or imposed anywhere
  in~\eqref{eq:AC_BE}--\eqref{eq:AC_BE_split}.
\item \emph{Conditioning.}  On gradient modes the eigenvalues are
  $1 + \Delta t\,\nu|\mathbf{k}|^{2}
     + \tilde\alpha\Delta t^{2}|\mathbf{k}|^{2}$
  and on divergence-free modes
  $1 + \Delta t\,\nu|\mathbf{k}|^{2}$, so
  \begin{equation}\label{eq:AC_BE_cond}
    \kappa \;\sim\; 1
      + \frac{\nu\,\Delta t}{\ell_{\min}^{2}}
      + \frac{\Delta t}{\varepsilon}\frac{\Delta t}{\ell_{\min}^{2}}
    \;=\; 1
      + \frac{\nu\,\Delta t}{\ell_{\min}^{2}}
      + \Bigl(\frac{c_{\mathrm{ac}}\Delta t}{\ell_{\min}}\Bigr)^{2},
  \end{equation}
  $\ell_{\min}$ the smallest grid spacing.  The acoustic term dominates
  the viscous one whenever $\tilde\alpha\Delta t\gg\nu$, which holds by
  many orders under~\eqref{eq:alpha_choice}.
\end{enumerate}

\medskip
\noindent
One step of classical AC is Template~A.
\begin{enumerate}[nosep]
\item[\textbf{A1}] \emph{Explicit tendencies.}  Form $\mathbf{R}^{n}$
  from the advective, Coriolis, buoyancy and hydrostatic-Bernoulli
  terms at time level $n$.
\item[\textbf{A2}] \emph{Velocity system}, \eqref{eq:AC_BE_vel}: one
  global, symmetric positive definite, vector-valued solve containing the viscous operator and the grad--divergence term, with the lagged
  pressure gradient on the right-hand side.
\item[\textbf{A3}] \emph{Pressure system}, \eqref{eq:AC_BE_pres}: one
  divergence evaluation, explicit, local.
\item[\textbf{A4}] \emph{Tracers}: transported by the pseudo-mass flux
  built from $\bv^{n+1}$, as required by~(H1)
  and~\eqref{eq:onescale_tracer}.
\end{enumerate}
Step~A2 is the only global operation and the only expensive one; its
cost is governed by~\eqref{eq:AC_BE_cond}, with $\ell_{\min}$ the
vertical spacing on a prismatic ocean mesh.  By
Remark~\ref{rem:coefficient} this is what stops Template~A at ocean
resolutions, and everything that follows is directed at removing~A2.

\subsection{The AC/DC Step}\label{sec:acdc_step}

AC/DC removes~A2 by treating the two directions of the mesh
differently: the stiff vertical direction is solved exactly within each
column, and only the horizontal remainder is left to the acoustic
stage.  To state this, split every operator into a horizontal and a
vertical part.  Write $\gradh$ and $\gradz$ for the horizontal
and vertical discrete gradient, $\divh$ and $\divz$ for the
corresponding divergences, so that
\begin{equation}\label{eq:dir_operators}
  \nabla = \gradh + \gradz ,
  \qquad
  \Div = \divh + \divz ,
  \qquad
  L_H = \divh\gradh ,
  \qquad
  L_z = \divz\gradz ,
\end{equation}
the discrete counterparts of the splitting~\eqref{eq:laplace_split}.
Let
$\mathcal{D} = \mathcal{D}_H+\mathcal{D}_V$ be the corresponding split
of the dissipation operator~\eqref{eq:visc_hodge_aniso}, with
$\mathcal{D}_V := {\rhoa}^{-1}\partial_z({\rhoa}\nu_v\partial_z\bv_H)$
the part that is tridiagonal within a column.  Both parts are
taken explicitly: the step performs no viscous
column pass, and the vertical viscosity enters S1 with everything else
at time level $n$.

The time stepping of AC/DC is the following.  
\begin{enumerate}
\item[\textbf{S1}] \emph{Explicit tendencies.}  Form
  $\mathbf{R}^{n} = -\mathbf{N}_{\mathrm{3D}}(\bv^{n})
     - \nabla B_{\mathrm{hyd}}^{n} + b^{n}\bk
     + (\mathcal{D}_H+\mathcal{D}_V)\bv^{n}$
  and the predictor
  \begin{equation}\label{eq:DC_S1}
    \bv^{*} = \bv^{n} + \Delta t\,\mathbf{R}^{n}
              - \Delta t\,\nabla\frac{\psi^{n}}{\rho_0} .
  \end{equation}

\item[\textbf{S2}] \emph{Column pressure.}  Solve
  \begin{equation}\label{eq:DC_S2}
    L_z\!\Bigl(\frac{p_V^{\,n+1}}{\rho_0}\Bigr) = \frac{\Div\bv^{*}}{\Delta t},
    \qquad
    p_V^{\,n+1}\big|_{z=\eta} = 0 ,
    \quad \partial_z p_V^{\,n+1}\big|_{z=-H} = 0 ,
  \end{equation}
  one tridiagonal solve per column, and correct
  \begin{equation*}
    \bv^{**} := \bv^{*} - \Delta t\,\nabla(p_V^{\,n+1}/\rho_0).
  \end{equation*}
  The divergence that survives this step is exactly the residual
  of~\eqref{eq:residual},
  \begin{equation}\label{eq:DC_residual}
    \Div\bv^{**}
      = \Div\bv^{*} - \Delta t\,(L_H+L_z)\Bigl(\frac{p_V^{n+1}}{\rho_0}\Bigr)
      = -\,\Delta t\,L_H\Bigl(\frac{p_V^{n+1}}{\rho_0}\Bigr)
      = \Delta t\,S^{*} ,
  \end{equation}
  the vertical part having been removed identically.

\item[\textbf{S3}] \emph{Acoustic relaxation of the residual.}  Advance
  the pair $(\bv,\psi)$ from $\bv^{**}$ over the interval $\Delta t$,
  \begin{equation}\label{eq:DC_S3}
    \frac{\bv^{n+1}-\bv^{**}}{\Delta t}
      + \nabla\frac{\psi^{n+1}-\psi^{n}}{\rho_0} = 0 ,
    \qquad
    \frac{\varepsilon}{\Delta t}\frac{\psi^{n+1}-\psi^{n}}{\rho_0}
      + \Div\bv^{n+1} = 0 ,
  \end{equation}
  and integrate it over $n_{\mathrm{sub}}$ sub-steps of length
  $\Delta t^{\mathrm{ac}} = \Delta t/n_{\mathrm{sub}}$.  The stage is
  left in this generic form: any sub-loop is admissible that reproduces
  \eqref{eq:DC_S3} to the accuracy of the sub-loop and requires no
  global operation.  Section~\ref{sec:variants2} and
  Section~\ref{sec:variants3} give the two concrete forms.  They differ
  in whether $\psi$ is split, and therefore in whether $\Delta x$ or
  $\Delta z$ sets $n_{\mathrm{sub}}$.

\item[\textbf{S4}] \emph{Pseudo-continuity and tracers.}  With the
  face fluxes $F_f$ of~\eqref{eq:Ff_def} built from $\bv^{n+1}$,
  \begin{equation}\label{eq:DC_S4}
    \frac{(|c|r_c)^{n+1}-(|c|r_c)^{n}}{\Delta t}
      = -\!\!\sum_{f\in\partial c}\!\sigma_{c,f}F_f ,
  \end{equation}
  \begin{equation}\label{eq:DC_S4b}
    \frac{(|c|r_cC_c)^{n+1}-(|c|r_cC_c)^{n}}{\Delta t}
      = -\!\!\sum_{f\in\partial c}\!\sigma_{c,f}F_f\,C_f
        - \!\!\sum_{f\ \mathrm{lat}}\!\!\sigma_{c,f}J_f^{\mathrm{diff},n}
        - \!\!\sum_{f\ \mathrm{hor}}\!\!\sigma_{c,f}J_f^{\mathrm{diff},n+1} ,
  \end{equation}
  advection and lateral ($\kappa_h$) diffusion explicit, vertical
  ($\kappa_v$) diffusion implicit-, one tridiagonal solve per column
  per tracer, sharing the sparsity of~S2.  Constants are preserved by
  both: advection by~(H2), diffusion because $[C]_f=0$ annihilates
  $J_f^{\mathrm{diff}}$.
\end{enumerate}

\paragraph*{Why S3 is left generic.}
Two requirements fix the stage.  It must not be
solved as a system: substituting the pressure equation
of~\eqref{eq:DC_S3} into the velocity equation reproduces
\eqref{eq:AC_BE_vel}, the global system of Template~A, conditioned by
the smallest spacing in the mesh.  It must also not cost more than the host already pays for its barotropic mode, which is a statement about
$n_{\mathrm{sub}}$.

An explicit sub-loop is limited by the smallest spacing in the problem,
and on a prismatic ocean mesh that is $\Delta z$.  Stage~S2 removed the
vertical direction from $p_V$; it did not remove it from $\psi$, so the
sub-loop meets the vertical limit unless the vertical structure of
$\psi$ is removed as well.  Removing it is Option~S3-a and leaving it
is Option~S3-b.  Both satisfy the two requirements above and both give
a working algorithm; they behave differently, and the sections that
follow state each in full.

\subsubsection{Option S3-a: Pseudo-Pressure Splitting}\label{sec:variants2}
We decompose $\psi$ as
\begin{equation}\label{eq:psi_split_concept}
  \psi = \bar\psi + \psi' ,\quad\text{ with }\bar\psi(x,y,t):=\frac{1}{H}\int_{-H}^0  \psi (x,y,z,t)\, dz
\end{equation}
with $\bar\psi$ the depth-averaged component and $\psi'$ the remainder.
The remainder holds the vertical structure
that~\eqref{eq:psi_converged} exhibits, and with it the stiffness the
column solve does not remove; it is kept out of the fast dynamics
entirely and updated once per step by a solve local to each column.
The depth average is two-dimensional and contains the fast response, and
is advanced explicitly with the velocity on the sub-interval
$\Delta t^{\mathrm{ac}} = \Delta t/n_{\mathrm{sub}}$.  The two are
therefore on different time scales, $\bar\psi$ cheap and fast, $\psi'$
stiff and slow, which is what makes the per-sub-step work
two-dimensional and the three-dimensional work once-per-step, and is
what the cost accounting of Section~\ref{sec:pass_acdc} rests on.
Steps~S1, S2 and~S4 of Section~\ref{sec:acdc_step} are unchanged; only
the acoustic stage~S3 is taken in the split form, by the two stages
below.  Nothing implicit and global remains: the implicit part is
confined to a single column and the rest is explicit.

\paragraph*{Boundary conditions of the pressure splitting.}
Equation~\eqref{eq:psi_converged} is an interior identity, and
the boundary conditions of the aggregated pressure are as follows:  the projection~\eqref{eq:pressure_laplace} of Section~\ref{sec:splitting} carries
$\mathbf{n}\cdot\nabla(p_{\mathrm{NH}}/\rho_0)=0$ on the fixed
impermeable boundaries  and
$p_{\mathrm{NH}}=0$ at the free surface.  Under the aggregation
$p_{\mathrm{NH}} = p_V+\psi$ the surface condition splits as
$p_V=0$ (imposed by the column solve) together with $\psi=0$, which
is~\eqref{eq:psi_bc} of Section~\ref{sec:free_surface}.  On the lateral
and bottom boundaries $\psi$ has no condition of its own: it is a
prognostic cell field, and the wall condition enters weakly through
$\bv\cdot\mathbf{n}=0$ in the flux form of~\eqref{eq:hybrid_psi}, whose
boundary faces carry zero flux.  The column condition
$\partial_z p_V=0$ then makes the aggregate satisfy the projection's
bottom condition cellwise.

\begin{enumerate}
\item[\textbf{S3-a1}] \emph{Remainder, column-implicit, once per step.}
  With $(\cdot)'$ the column-mean-free part
  of~\eqref{eq:psi_split_concept},
  \begin{equation}\label{eq:C_psiprime}
    \bigl(I - \tilde\alpha\Delta t^{2}L_z\bigr)
      \frac{\psi'^{\,n+1}}{\rho_0}
    = \frac{\psi'^{\,n}}{\rho_0}
      - \tilde\alpha\Delta t\,\bigl(\Div\bv^{**}\bigr)'
      + \tilde\alpha\Delta t^{2}L_H\,\frac{\psi'^{\,n}}{\rho_0} .
  \end{equation}
  $L_z$ is tridiagonal within each column and couples no two columns,
  so~\eqref{eq:C_psiprime} is one small solve per column with no
  horizontal communication, sharing the sparsity of the column pressure
  solve~S2.  The horizontal
  part $L_H$ is evaluated at the old level and carried on the
  right-hand side; it is the only explicit term in the pressure
  update, and it is the one that sets the step restriction below.

\item[\textbf{S3-a2}] \emph{Depth-averaged pair, explicit sub-steps.}
  Holding $\psi'^{\,n+1}$ fixed as an offset, the column pressure
  $p_V$ having already been applied to $\bv^{**}$ in~S2,  advance
  $(\bv,\bar\psi)$ over $n_{\mathrm{sub}}$ sub-steps of length
  $\Delta t^{\mathrm{ac}}$ by St\"ormer--Verlet,
  \begin{subequations}\label{eq:C_leapfrog}
  \begin{align}
    \bar\psi^{\,k+\frac12} &= \bar\psi^{\,k}
      - \tfrac{1}{2}\Delta t^{\mathrm{ac}}\,\tilde\alpha\rho_0\,
        \overline{\Div\bv^{\,k}} ,
      \label{eq:C_leapfrog_a}\\[2pt]
    \bv^{\,k+1} &= \bv^{\,k}
      - \Delta t^{\mathrm{ac}}\,\nabla
        \frac{\bar\psi^{\,k+\frac12} + \psi'^{\,n+1}}{\rho_0} ,
      \label{eq:C_leapfrog_b}\\[2pt]
    \bar\psi^{\,k+1} &= \bar\psi^{\,k+\frac12}
      - \tfrac{1}{2}\Delta t^{\mathrm{ac}}\,\tilde\alpha\rho_0\,
        \overline{\Div\bv^{\,k+1}} ,
      \label{eq:C_leapfrog_c}
  \end{align}
  \end{subequations}
  with $\bv^{\,0}=\bv^{**}$, $\bar\psi^{\,0}=\bar\psi^{\,n}$ and
  $\bv^{n+1}=\bv^{\,n_{\mathrm{sub}}}$,
  $\bar\psi^{\,n+1}=\bar\psi^{\,n_{\mathrm{sub}}}$.  The overbar is the
  depth average of~\eqref{eq:psi_split_concept}, so the pressure
  recursion~\eqref{eq:C_leapfrog_a},~\eqref{eq:C_leapfrog_c} acts on a
  two-dimensional field; only~\eqref{eq:C_leapfrog_b} is
  three-dimensional, and it is a gradient of a field that is constant
  in the vertical plus a fixed offset.

\end{enumerate}

Three consequences follow, and they are what the option is for.

\emph{Sub-step count.}  The vertical channel is implicit
in~\eqref{eq:C_psiprime}, so the only condition on the sub-loop is the
horizontal one,
$c_{\mathrm{ac}}\Delta t^{\mathrm{ac}}/\Delta x = \OO(1)$, hence
$n_{\mathrm{sub}} = \OO(c_{\mathrm{ac}}\Delta t/\Delta x)$
by~\eqref{eq:coefficient_ratio}: the same count, and the same mode, as
the barotropic sub-stepping an ocean model already performs.  Each
sub-step applies one gradient and one depth-averaged divergence, both
nearest-neighbour stencils, and~\eqref{eq:C_psiprime} is block-diagonal
by column, so the sub-loop stays local to a cell and its column.

\emph{The sub-loop flux.}  By
\eqref{eq:C_leapfrog_b} the velocity at sub-step $k$ is affine in the
accumulated $\bar\psi$, so the time-integrated flux of the whole
sub-loop follows from one additional two-dimensional accumulator and a
single lift at the exit.  Step~S4 therefore uses the same $F_f$ it
does in Section~\ref{sec:acdc_step}, and hypotheses~(H1)--(H3) carry
over unchanged.  The weight $r_f$ in~\eqref{eq:Ff_def} is evaluated
once, at the state entering the step, and held fixed
through~\eqref{eq:C_leapfrog}: the time level is not fixed
by~\eqref{eq:Ff_def}, and a constant weight is what keeps the
recursion's operator constant, as the symplectic structure requires.

\emph{Applying the horizontal term.}  This is the one
implementation choice that decides whether the option is affordable, so
it belongs with the scheme.  The explicit term $L_H\psi'$
of~\eqref{eq:C_psiprime} is a gradient, a divergence and the removal of
a column mean.  Applied in sequence, each factor is local and the
composition touches $\OO(1)$ entries per row.  Multiplied out into a
matrix, the column-mean removal couples every cell of a column to every
other and the stored operator grows as (columns)$\times n_z^{2}$.  Both
give the same operator and~\eqref{eq:C_psiprime} is unchanged; the
choice decides only whether the option's memory is a field or a matrix,
and at the sizes it exists for that decides feasibility.

\subsubsection{Option S3-b: An Explicit Acoustic Step}\label{sec:variants3}
Option~S3-a kept the pseudo-pressure implicit and split it so that the
implicit part stayed inside a column.  This option gives up both:
$\psi$ is left whole, and the pair $(\bv,\psi)$ of~\eqref{eq:DC_S3} is
advanced explicitly, so no matrix is assembled and none is inverted
anywhere in the step.  Again only~S3 of Section~\ref{sec:acdc_step}
changes; S1, S2 and~S4 are as before.

\begin{enumerate}
\item[\textbf{S3-b}] \emph{Acoustic relaxation, explicit.}  Over
  $n_{\mathrm{sub}}$ sub-steps of length
  $\Delta t^{\mathrm{ac}} = \Delta t/n_{\mathrm{sub}}$, advance the
  full three-dimensional pair by St\"ormer--Verlet,
  \begin{subequations}\label{eq:D_leapfrog}
  \begin{align}
    \psi^{\,k+\frac12} &= \psi^{\,k}
      - \tfrac{1}{2}\Delta t^{\mathrm{ac}}\,\tilde\alpha\rho_0\,
        \Div\bv^{\,k} ,
      \label{eq:D_leapfrog_a}\\[2pt]
    \bv^{\,k+1} &= \bv^{\,k}
      - \Delta t^{\mathrm{ac}}\,\nabla
        \frac{\psi^{\,k+\frac12}}{\rho_0} ,
      \label{eq:D_leapfrog_b}\\[2pt]
    \psi^{\,k+1} &= \psi^{\,k+\frac12}
      - \tfrac{1}{2}\Delta t^{\mathrm{ac}}\,\tilde\alpha\rho_0\,
        \Div\bv^{\,k+1} ,
      \label{eq:D_leapfrog_c}
  \end{align}
  \end{subequations}
  with $\bv^{\,0}=\bv^{**}$, $\psi^{\,0}=\psi^{\,n}$ and
  $\bv^{n+1}=\bv^{\,n_{\mathrm{sub}}}$,
  $\psi^{\,n+1}=\psi^{\,n_{\mathrm{sub}}}$.  There is no operator to
  invert or factor here: the step applies $\nabla$ and $\Div$.
\end{enumerate}

Every operation in~\eqref{eq:D_leapfrog} is a nearest-neighbour
stencil, so the option needs no reduction and no global communication
of any kind, and the whole cost is in $n_{\mathrm{sub}}$.  Because
$\psi$ is not split, the vertical channel is present
in~\eqref{eq:D_leapfrog}: this is the restriction that~S2 removes from
the \emph{pressure} and the splitting removes from the
\emph{sub-loop}, and that~S3-b reinstates.

\paragraph*{Classical AC and the two options side by side.}

\begin{center}\footnotesize
\begin{tabular}{@{}l p{3.0cm} p{3.4cm} p{3.0cm}@{}}
  \toprule
  & Template~A (classical AC)
  & AC/DC, Option~S3-a (split $\psi$)
  & AC/DC, Option~S3-b (whole $\psi$) \\
  \midrule
  explicit tendencies & A1 & S1 & S1 \\
  viscosity           & inside the global solve A2
                      & explicit, in S1 & explicit, in S1 \\
  column pressure     & ---
                      & S2, column tridiagonal
                      & S2, column tridiagonal \\
  velocity update     & A2, one global SPD solve
                      & S3-a2, explicit sub-steps
                      & S3-b, explicit sub-steps \\
  pressure update     & A3, explicit
                      & S3-a1 column-implicit ($\psi'$);
                        S3-a2 explicit ($\bar\psi$)
                      & S3-b, explicit \\
  tracers             & A4 & S4 & S4 \\
  \midrule
  global operations   & one, cond.~\eqref{eq:AC_BE_cond}
                      & \textbf{none} & \textbf{none} \\
  assembled operators & the global system
                      & per-column blocks only
                      & \textbf{none} \\
  column solves/step  & ---
                      & $n_{\mathrm{tracer}}+2$
                      & $n_{\mathrm{tracer}}+1$ \\
  sub-steps/step      & ---
                      & $n_{\mathrm{sub}}$, 2D field, horizontal
                        condition
                      & $n_{\mathrm{sub}}$, 3D pair, vertical
                        condition \\
  \bottomrule
\end{tabular}
\end{center}
\noindent
The three columns differ in one place only, the acoustic stage~S3:
Template~A inverts the grad--divergence operator whole and pays the
conditioning~\eqref{eq:AC_BE_cond} for it, S3-a splits the
pseudo-pressure, S3-b leaves it whole.  Everything above the rule is
the same computation in a different order.  What the two lower rows
record is the property the paper is after: neither option performs a
global reduction, and both reduce to nearest-neighbour stencils and
column-local work, the access pattern throughput-oriented
architectures are built for.

\subsubsection{Choosing $\alpha$ and the Number of Sub-steps}\label{sec:parameters}

Both options leave two numbers to the user, $\alpha$ and the number of
acoustic sub-steps, and a reader setting up a run needs to know how they are
fixed.  They are not independent, and only one of them is a choice.

\emph{The mechanism.}  Section~\ref{sec:ac_system}
described $\psi$ as a spring: where the flow converges the fluid is
compressed, $\psi$ rises, and $\nabla\psi$ pushes back.  That push-back is
\emph{local} and needs no parameter tuning at all --- it happens in the cell
where the convergence is, within the step, for any $\alpha$.  What $\alpha$
decides is how far the push-back has \emph{travelled} by the end of the
step.  In a truly incompressible fluid the pressure is global and
instantaneous: a convergence anywhere is felt everywhere at once, and that
simultaneity is precisely what the elliptic solve computes.  The AC system
reproduces it only over the time its sound needs to get there.  In one outer
step the news travels a distance $c_{\mathrm{ac}}\Delta t$, and the ratio of
that reach to the scale over which the pressure must act is the
\emph{crossing number}
\begin{equation}\label{eq:crossing}
  \mathcal{C} := \frac{c_{\mathrm{ac}}\,\Delta t}{L} .
\end{equation}
So $\alpha$ sets two things at once, and it is worth keeping them apart: it
is the stiffness of the spring, and, through
$c_{\mathrm{ac}}=\sqrt{\alpha/\rho_0}$, also its \emph{speed}.  A stiffer
spring pushes back harder \emph{and} spreads the news further.  The crossing
number measures the second, and it is the second that decides whether the
constraint acts locally or across the domain.

\emph{The lower bound.}  If $\mathcal{C}<1$ the
pseudo-sound does not cover $L$ within a step, and a pressure field
organised on that scale cannot form at all --- not inaccurately, but not in
principle.  Where $L$ is the domain, $\mathcal{C}\geq1$ is therefore
necessary.  It is not, however, sufficient: the useful value of
$\mathcal{C}$ is larger, it is not the same for every flow, and it is not
derivable from the parameters of the discretisation, because $\alpha$ is
built from the mesh and the time step alone and takes no account of the
speed of the flow it is being asked to constrain.  It has to be calibrated
per configuration, by comparing against a projection reference on the same
mesh, and the experiments of Section~\ref{sec:experiments} use different values for exactly that reason.

\emph{Fixing $\alpha$.}  Given a target crossing number, inverting
\eqref{eq:crossing} with $c_{\mathrm{ac}}=\sqrt{\alpha/(\rho_0|c|)}$ gives
the rule used throughout:
\begin{equation}\label{eq:alpha_from_crossing}
  \alpha \;=\; \rho_0\,|c|\left(\frac{\mathcal{C}\,L}{\Delta t}\right)^{2},
\end{equation}
with $|c|$ the largest cell volume, the largest because it has the slowest sound and so the worst case.  Two features of \eqref{eq:alpha_from_crossing}
are easy to miss and both matter.  First, $\alpha$ depends on the
\emph{mesh}: it includes a cell volume, so a constant $\alpha$ cannot be
correct on two meshes and has to be recomputed whenever the mesh changes.
Second, $\alpha \propto \Delta t^{-2}$.  Refining the time step at fixed
$\alpha$ lowers the crossing number and weakens the constraint, which is the
opposite of what refinement is supposed to do; holding $\mathcal{C}$ fixed
is what makes a sequence of runs a refinement of one scheme.

\emph{The sub-step count.}  The acoustic stage of
Option~S3-b is explicit, so its sub-step is bounded by its own
Courant condition, $c_{\mathrm{ac}}\Delta t_{\mathrm{sub}} \lesssim
\Delta x_{\min}$. 
Combining the bound with \eqref{eq:alpha_from_crossing} gives the count
directly,
\begin{equation}\label{eq:nsub}
  n_{\mathrm{sub}} \;\simeq\; \frac{\mathcal{C}\,L}{\Delta x_{\min}} ,
\end{equation}
and the time step has cancelled.  Refining $\Delta t$ at fixed
$\mathcal{C}$ does not change the number of sub-steps; refining the
\emph{mesh} does, in proportion to $1/\Delta x_{\min}$.  A graded mesh pays
at both ends, since $\alpha$ is set by the largest cell
in~\eqref{eq:alpha_from_crossing} and $n_{\mathrm{sub}}$ by the smallest
in~\eqref{eq:nsub}; that product is the price of the grading, and it is why
the telescoped disc of Section~\ref{sec:exp_bubble} takes so many acoustic
sub-steps per outer step.

\emph{In practice}, then: choose $\mathcal{C}$ for the flow, subject to
$\mathcal{C}\geq1$ where the constraint must act across the domain; obtain
$\alpha$ from \eqref{eq:alpha_from_crossing}; let the sub-step count follow
from the Courant condition.

\subsection{Well-Prepared Initial Conditions}\label{sec:initialisation}

Artificial compressibility turns a constrained system into an unconstrained one with a fast mode, and every such relaxation
inherits a requirement that the constrained system does not have: the
initial data must be compatible with the constraint, not merely with
the prognostic variables.  

Two conditions are involved:  the first is
that the initial velocity be discretely divergence-free,
$\Div\bv_0 = 0$; the second is that the initial pseudo-pressure be the
one that keeps it divergence-free, i.e. $\psi_0$ must be such that
$\Dt(\Div\bv) = 0$ at $t = 0^+$, which by~\eqref{eq:ac_mom} means
\begin{equation}
  \Div\!\left[\nabla\frac{\psi_0}{\rho_0}\right]
  = \Div\!\left[\mathcal{D}(\bv_0) + b_0\,\bk
    - \mathbf{N}_{\mathrm{3D}}(\bv_0) - \nabla B_{\mathrm{hyd},0}\right].
  \label{eq:well_prepared}
\end{equation}

Linearising~\eqref{eq:ac_psi}--\eqref{eq:ac_mom} about
rest gives $\Dt^2\psi = c_{\mathrm{ac}}^2\Delta\psi$, so an
inconsistent $\psi_0$ of magnitude $\Delta\psi$ launches a wave whose
velocity signature is $\Delta\psi/(\rho_0 c_{\mathrm{ac}})$.  For a
buoyancy anomaly of scale $\ell$ and amplitude $b$, the pressure
required by~\eqref{eq:well_prepared} is $\psi_0 \sim \rho_0 b\ell$,
against a physical velocity $\sqrt{b\ell}$, so the spurious velocity
relative to the physical one is
$\sqrt{b\ell}/c_{\mathrm{ac}} = \Fr_{\mathrm{baro}}$.  Initialisation
error therefore enters at $\OO(\Fr_{\mathrm{baro}})$, whereas every
error controlled elsewhere in this paper, the divergence error, the
dispersion error, the energy residual, enters at
$\OO(\Fr_{\mathrm{baro}}^2)$.  An unprepared start is an error one order lower in the small parameter,
and at the values of Section~\ref{sec:ac_system} that is two orders of
magnitude.  The two
failure modes scale differently in the AC parameter: an inconsistent
$\psi_0$ produces an error decaying as $\alpha^{-1/2}$, whereas a
non-zero initial divergence produces a velocity error independent of
$\alpha$.

 Condition~\eqref{eq:well_prepared} can be reached
by running the AC sub-stepping itself in pseudo-time with damping
added, until $\Div\bv$ falls below the tolerance of interest.  This is
Chorin's original steady-state use of artificial compressibility
\citep{Chorin1967} applied once at initialisation, and it inherits all of the properties argued for in
Section~\ref{sec:pressure_calc}.  Where an
elliptic solver does happen to be available, a single projection at
$t=0$ is equally valid and cheaper, and it does not weaken the cost
argument of Section~\ref{sec:pressure_calc}.

\section{Telescoping Meshes: Where the Advantages Converge}\label{sec:telescoping}
Variable-resolution meshes (Figure~\ref{fig:telescoping}) are a defining capability of unstructured-mesh ocean models such as ICON-O and MPAS-Ocean: a single global mesh with for example $\Delta x\sim 50$\,km in the open ocean narrowing continuously to $\Delta x\sim 500$\,m or finer in some target region.  The outer region is firmly hydrostatic, the inner region may need non-hydrostatic dynamics, and the two must coexist in one model.  (Throughout, $\delta$ denotes the aspect ratio of the resolved flow, distinct from the grid ratio $H/\Delta x$.) 

The projection method faces a difficulty here that it does not face on a uniform mesh: 
the physics is local and the solve is not.  Solved globally, the iteration count, the per-iteration cost and the communication are all set by the global problem, while the non-hydrostatic physics is confined to a localized patch.  Restricted to that patch, the solve requires a condition at the hydrostatic--non-hydrostatic interface: $p_{\mathrm{NH}}=0$ is only approximately right, since the transition is gradual, and a sharp Dirichlet condition reflects internal gravity waves back into the refined region.  Sponge layers mitigate this and require tuning; coupling a non-hydrostatic model to a hydrostatic one across the interface requires bespoke treatment of different prognostic variables, time scales and conservation properties on either side. 

Artificial compressibility removes the interface entirely.  One equation set holds everywhere: $\psi$ is prognostic on every cell and the AC equation~\eqref{eq:ac_psi} is solved globally, so there is no interface, no boundary condition at one, and no switching.  In the coarse region $p_{\mathrm{NH}}\sim\rho_0\delta^2U^2\approx0$, so $\psi$ stays near zero, the prognostic $w$ relaxes to its diagnostic value, and the system reduces to the hydrostatic primitive equations without being told to.  The physics self-selects. Pure AC performs this  globally: its sub-stepping advances $\psi$ and evaluates $\Div\bv$ and $\nabla\psi$ on the full three-dimensional mesh at every sub-step, with the sub-step count set by the finest cell anywhere in the domain: the coarse region is sub-stepped far more often than its own dynamics requires, and the overhead does not fall as the refined patch shrinks. AC/DC places the work where the physics is.  The column solve is local and independent per column, $\OO(n_z)$ regardless of horizontal resolution, so it is cheap where $p_V\approx0$ and captures the dominant non-hydrostatic pressure.  

\begin{figure}[H]
  \centering
  \begin{tikzpicture}[
    font=\footnotesize,
    >={Stealth[length=1.8mm]},
    ocean/.style={draw,thick},
    ptitle/.style={align=center,text width=90mm,font=\footnotesize\bfseries},
    pcap/.style={align=center,text width=86mm,font=\scriptsize},
  ]
  \def\PL{28mm}\def\PR{46mm}   
  \begin{scope}[yshift=0mm]
    \fill[black!30] (0,0) rectangle (80mm,-15mm);
    \foreach \x in {0,8,16,24,56,64,72}
      \draw[black!45] ({\x mm},0) -- ({\x mm},-15mm);
    \foreach \i in {0,...,8}
      \draw[black!45] ({28mm+\i*2mm},0) -- ({28mm+\i*2mm},-15mm);
    \draw[black!70,thick] (\PL,0) -- (\PL,-15mm);
    \draw[black!70,thick] (\PR,0) -- (\PR,-15mm);
    \draw[ocean] (0,0) rectangle (80mm,-15mm);
    \fill[black] (37mm,-7.5mm) circle (0.9mm);
    \draw[->,black!80] (37mm,-7.5mm) -- (77mm,-7.5mm);
    \draw[->,black!80] (37mm,-7.5mm) -- (3mm,-7.5mm);
    \node[anchor=south,font=\scriptsize] at (37mm,0.6mm) {refined patch};
    \node[anchor=north,font=\scriptsize] at (40mm,-16.2mm)
      {every iteration sweeps the whole mesh; conditioning set by $\Delta x_{\min}$};
    \node[ptitle,anchor=north] at (40mm,-21mm)
      {(a) projection: cost set by $\Delta x_{\min}$, paid everywhere};
  \end{scope}
  \begin{scope}[yshift=-34mm]
    \foreach \x in {0,8,16,24,56,64,72}
      \draw[black!45] ({\x mm},0) -- ({\x mm},-15mm);
    \foreach \i in {0,...,8}
      \draw[black!45] ({28mm+\i*2mm},0) -- ({28mm+\i*2mm},-15mm);
    \begin{scope}
      \clip (\PL,0) rectangle (\PR,-15mm);
      \fill[black!12] (37mm,-7.5mm) circle (13mm);
      \foreach \r/\op in {9/20,5/30}{\fill[black!\op] (37mm,-7.5mm) circle (\r mm);}
      \foreach \r in {5,9}{\draw[black!55] (37mm,-7.5mm) circle (\r mm);}
      \draw[densely dashed,black!65] (37mm,-7.5mm) circle (13mm);
    \end{scope}
    \draw[black!70,thick] (\PL,0) -- (\PL,-15mm);
    \draw[black!70,thick] (\PR,0) -- (\PR,-15mm);
    \draw[ocean] (0,0) rectangle (80mm,-15mm);
    \fill[black] (37mm,-7.5mm) circle (0.9mm);
    \node[anchor=south,font=\scriptsize] at (37mm,0.6mm) {refined patch};
    \node[anchor=north,font=\scriptsize] at (40mm,-16.2mm)
      {coarse region untouched: $p_{\mathrm{NH}}\approx0$, $\psi$ relaxed to
       $\OO(\delta^4\Fr_{\mathrm{baro}}^2)$};
    \node[ptitle,anchor=north] at (40mm,-21mm)
      {(b) AC/DC: cost confined to the patch, one equation set throughout};
  \end{scope}
  \end{tikzpicture}
  \caption{\it \small A telescoping mesh under the two algorithms:
    $50$\,km in the open ocean, refined to $500$\,m or finer in a
    target region.  For the projection~(a) the
    elliptic problem stays global.  Its
    conditioning follows the \emph{finest} spacing, each iteration sweeps the entire mesh, including the
    coarse region where $p_{\mathrm{NH}}$ is negligible.\\  
    AC/DC~(b) has one equation everywhere, and $\psi$ relaxes to
    $\OO(\delta^4\Fr_{\mathrm{baro}}^2)$ in the hydrostatic region.  The coupling of
    Figure~\ref{fig:acdc_idea}(c) has only the target region to cover,  the
    completion time $L/c_{\mathrm{ac}}$ is set by the target region's size.}
  \label{fig:telescoping}
\end{figure}
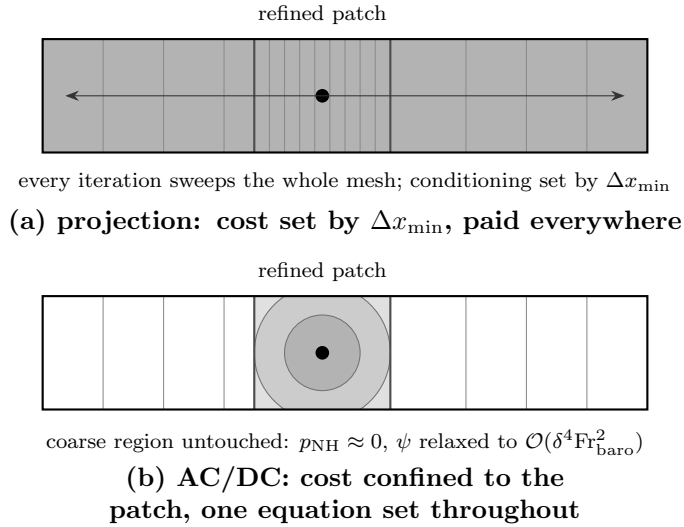
The error adapts with the cost.  For resolved flows  the prefactor $(k_x/k_z)^2$ of~\eqref{eq:hybrid_div_error} is taken at the aspect ratio of the structure, which turns that estimate into 
$\delta^4\Fr_{\mathrm{baro}}^2$.  With $\delta\approx7\times10^{-3}$ for the flow a $50$\,km mesh resolves, a horizontal scale of some $600$\,km against a depth of $4$\,km, the divergence error is $\sim10^{-13}$ in the coarse region, effectively machine zero, while pure AC gives $\delta^2\Fr_{\mathrm{baro}}^2\sim10^{-9}$ there.  In the refined patch, where $\delta$ rises to the submesoscale value $0.2$, the corresponding figures are $10^{-11}$--$10^{-9}$, still far below any other error source there.  Neither the cost nor the error is switched or tuned at the transition; both follow the mesh.

\begin{table}[H]\small  \centering  \caption{\small\it Comparison on a telescoping mesh (50\,km to 500\,m, global    ocean).  ``NH patch'' denotes the submesoscale-resolving refined    region.  Iteration and    reduction counts are the heuristic bands of    Section~\ref{sec:pressure_calc}; the divergence errors are the
    scaling estimates of Section~\ref{sec:hybrid_div}.  The entries    that read ``none'' are exact: they follow from the algorithm.}  \label{tab:telescoping}  \renewcommand{\arraystretch}{0.95}  \begin{tabular}{lccc}    \toprule    & \textbf{Projection}    & \textbf{AC (3D)}    & \textbf{AC/DC} \\    \midrule    Solver scope      & global (or regional + BC)      & global      & global (but locally cheap) \\    Work in coarse region      & large (Poisson iter.)      & moderate (3D sub-steps)      & small ($p_V\!\approx\!0$) \\    Work in NH patch      & large      & moderate      & moderate \\    Interface conditions      & needed (if regional)      & none      & none \\    Iterations per step      & $\mathcal{O}(10^{2})$      & ---      & --- \\    Sequential reductions      & $\mathcal{O}(10^{2})$      & none      & none \\    Div.\ error, coarse      & 0      & $\sim 10^{-9}$      & $\sim 10^{-13}$ \\    Div.\ error, NH patch      & 0      & $\sim 10^{-6}$      & $10^{-11}$--$10^{-9}$ \\    \bottomrule  
\end{tabular}\end{table}

\section{Numerical Experiments}\label{sec:experiments}

This section presents numerical experiments that test the
physical correctness of the AC/DC method against a projection-method reference.  

The projection method provides a ground-truth solution: it
solves the same discrete equations using an exact Leray projection
(Poisson solve on the Laplace operator $L = L_H + L_z$) at
every step. This implies that the numerical tests are confined to idealized test,
and do not show large-scale high-resolution ocean experiments, for which this reference does not exist. Furthermore we use benchmarkk solution to compare both methods agains.

\paragraph{AC parameter calibration.}
The AC parameter $\alpha = \rho_0 c_\mathrm{ac}^2$ is the only
non-trivial hyperparameter of AC/DC, and its calibration is
experiment-specific.  The rule is
$c_\mathrm{ac} = \sqrt{\alpha/\rho_0} \gg U$ where $U$ is the
slow buoyancy-driven velocity scale of the problem.  For
buoyancy-driven flows with $U \sim 1$ (Härtel units), $\alpha = 10^3$
is sufficient; for surface-cooling-driven convection at $\mathrm{Ra} =
10^6$ where $U \sim \sqrt{\mathrm{Ra\,Pr}} = 10^3$, the requirement
escalates to $\alpha = 10^8$.  The validated $\alpha = 1$ default of the regression tests is
calibrated for IGW problems with weak buoyancy forcing, where
Proposition~\ref{prop:hybrid_dispersion} already gives
$\OO(10^{-5})$ error.

\begin{center}
\small
\begin{tabular}{lcll}
\toprule
Experiment & $\alpha$ & Velocity scale $U$ & $n_\mathrm{sub}$, coarse mesh \\
\midrule
E1 lock-exchange    & $10^3$    & $\sim \sqrt{g'H/2} = 0.7$           & 2 \\
E2 IGW dispersion   & $1$ (any) & $\sim N\,\mathrm{amp}$ (weak)        & ---  \\
E3 deep convection  & $10^8$    & $\sim \sqrt{\mathrm{Ra\,Pr}} = 10^3$ & 1 (quads), 2 (tri.) \\
E4 geostrophic adj. & $10^3$    & $\sim 0.5$                           & 2 (quads), 3 (tri.) \\
\bottomrule
\end{tabular}
\end{center}
\noindent\emph{AC parameter $\alpha$ used in each experiment.}
Rule of thumb: $\alpha \gg U^2 \rho_0$ where $U$ is the slow
buoyancy- or forcing-driven velocity scale.

\noindent
$n_\mathrm{sub}$ is the number of horizontal-acoustic sub-steps per
outer time step, set automatically as
$n_\mathrm{sub} = \lceil \Delta t\, c_\mathrm{ac} /
\Delta x_\mathrm{min} \rceil$.  At all listed values,
$n_\mathrm{sub}$ is a small integer, so the AC/DC
overhead is bounded.

\emph{Vertical coordinate.}  All experiments use $z$-coordinates
(fixed geopotential levels) for simplicity and transparency: the
internal-wave, lock-exchange, convection, and energy-conservation
experiments have flat bottoms and no free-surface effects that would
require $z^*$.  On prismatic meshes with $z$-coordinates, the
vertical direction is orthogonal to the horizontal, the vertical
Laplacian $L_z = \partial_{zz}$ is a clean tridiagonal operator per
column, and the Hodge star decomposes as
$\hodge_1 = \mathrm{diag}(\hodge_1^H, \hodge_1^V)$ without
cross-terms.  The free surface and the extension to $z^*$ are treated
in Section~\ref{sec:free_surface}, where the conservation statements
are shown to hold verbatim on moving meshes.

\subsection{Experiment 1: Lock-exchange gravity current}\label{sec:exp_lock}

\emph{Purpose.}  Demonstrate that the projection method and AC/DC
both reproduce the classical Boussinesq lock-exchange front speed,
and that the front structure (including the gravity-current head)
agrees with the \citet{Hartel2000} DNS benchmark.

\emph{Setup.}  Härtel non-dimensional formulation, with the dense
fluid initially confined to $x > L_x/2$ in a rectangular vertical
plane $L_x = 18$, $H = 2$ (in Härtel units, after rescaling by the
half-depth), walls on all boundaries.  Reduced gravity $g' = 1$,
Grashof number $\mathrm{Gr} = 1.25 \times 10^4$, Schmidt number $\mathrm{Sc} = 1$.
The coarse mesh is $240 \times 3 \times 80$ quad cells, the fine one
$480 \times 3 \times 160$.  The figures below are from the fine mesh
with $\Delta t = 1\times10^{-3}$, integrated $8000$ steps to $t = 8$;
the step-refinement refinement sequence of the closing paragraph is run on the
coarse mesh to $t = 3$.  The density is
transported by the validated divergence-consistent upwind scheme
(used unchanged in all experiments).  

\emph{Reference solution.}  Benjamin's energy-conserving front
speed for a current occupying half of a channel of depth $H$ is
$U_f = \tfrac{1}{2}\sqrt{g'H}$ \citep{Benjamin1968}; in units of the
buoyancy velocity $\sqrt{g'H/2}$ used here this is
$U_f = 1/\sqrt{2} \approx 0.707$, an inviscid upper bound that a
current against no-slip walls cannot attain.  \citet{Hartel2000}
normalise by the same buoyancy velocity (their equation~(4), with the
channel half-height), so their Froude numbers can be placed on
this axis without conversion.  Their value rises with Grashof number
along separate no-slip and free-slip curves, both approaching
Benjamin's bound from below; read off the no-slip curve of their
figure~4 at a Grashof number comparable to this run, it is
$\mathrm{Fr} \approx 0.59$--$0.63$.  

\paragraph*{A.\ Projection against AC/DC.}
\emph{Results: projection versus AC/DC.}  The two algorithms were run on
identical meshes, initial conditions, transport scheme and time
integrator, differing only in the velocity-update block, at two
resolutions of the same trimmed domain.  The front speed is fitted from
the dense-front position $x_f(t)$, and the interval it is fitted over
has to be stated, because two different quantities are being measured.
Over $t\in[1.0,2.4]$ the front is still passing through its peak speed,
so the fit there measures how the two methods compare on the same
transient and is the right window for that purpose.  But it is not a steady front speed and cannot be set beside Benjamin's or a DNS value.
Over $t\in[2.4,7.5]$ the front travels at constant speed, flat to
within $2\%$, and a straight line fits the trajectory three times
better ($0.20\%$ of the distance travelled against $0.54\%$); that is
the window for comparison with the references, and the one
Figure~\ref{fig:E1_front} uses.  On the transient window:

\begin{figure}[ht]
\begin{center}
\noindent\includegraphics*[width=1.\textwidth]{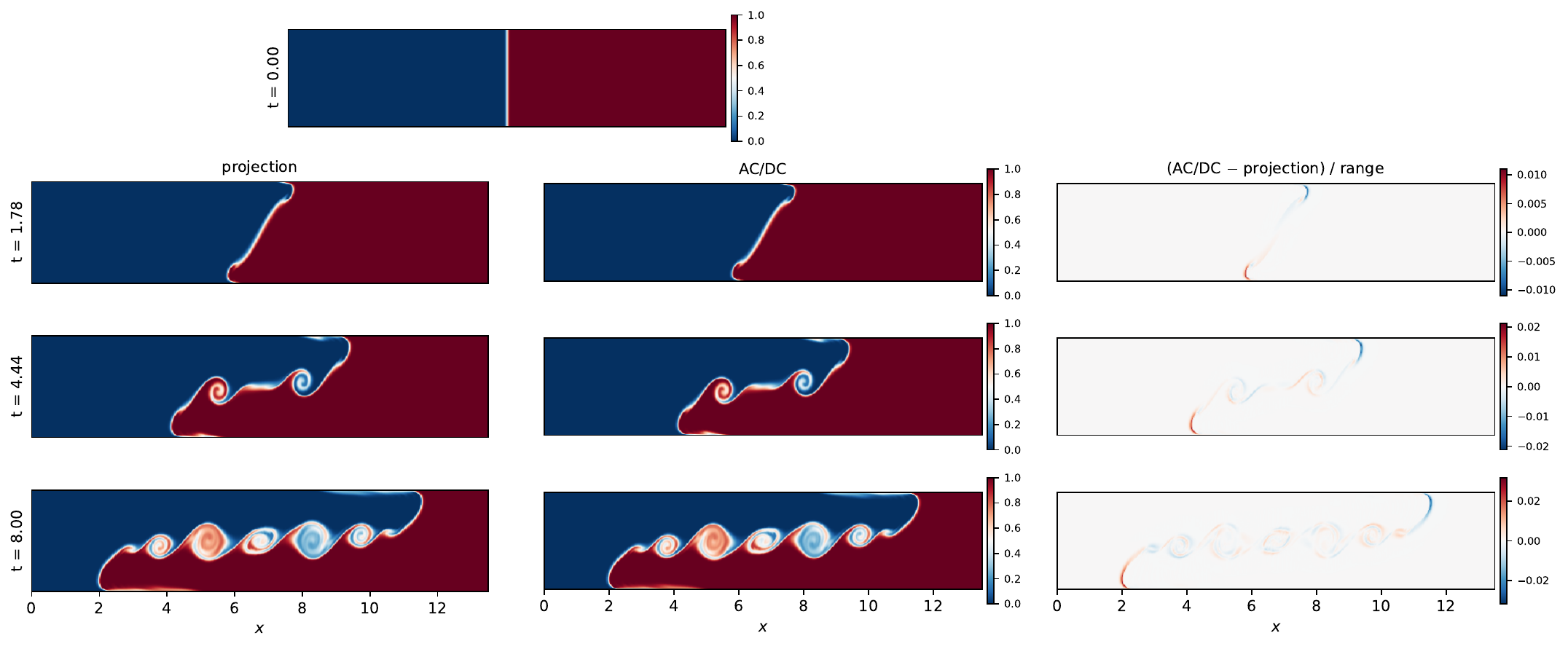}
\noindent\includegraphics*[width=1.\textwidth]{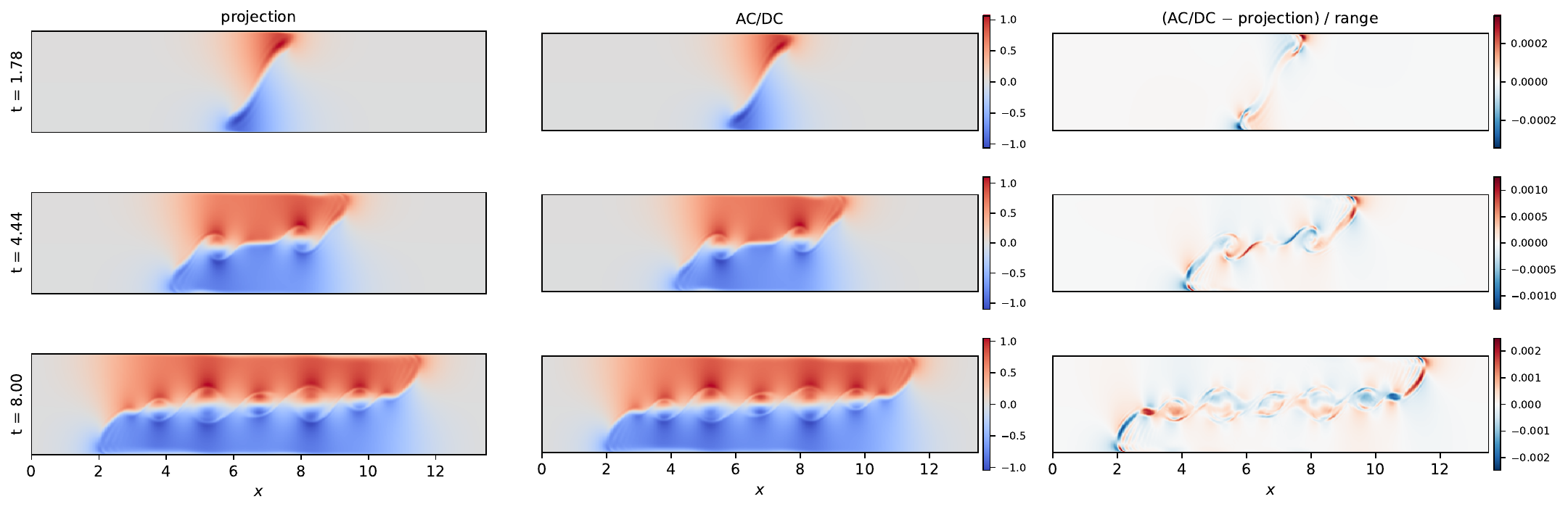}
\caption{\it\small Experiment~\ref{sec:exp_lock}: the density (upper
  block) and the horizontal velocity (lower block) on the two methods.
  \emph{Top strip:} the initial density, which both runs share.
  \emph{Rows:} $t = 1.78$, $4.44$ and $8.00$.  \emph{Columns:}
  projection, AC/DC, and their difference divided by the range of the
  projection field at that instant.  Within a row the two method columns
  share a single colour scale, so the eye is comparing the same numbers;
  the difference column has its own symmetric scale, which therefore
  changes from row to row and has to be read off its own colour bar.
  The difference is normalised by the field's range. The channel is many times longer than it is deep,
  and the panels are drawn with the vertical scale exaggerated by the
  same factor throughout.  The two solutions are indistinguishable by eye at
  every instant, and what difference there is lies on the rolled-up interface.}
  \label{fig:E1_fields}
\end{center}
\end{figure}

\begin{center}
\begin{tabular}{lccc}
\toprule
Mesh & $\Delta t$ & $\mathrm{Fr}_{\mathrm{proj}}$ & $\mathrm{Fr}_{\mathrm{AC/DC}}$ \\
\midrule
$240\times3\times80$  & $2\times10^{-3}$ & 0.697139 & 0.697290 \\
$480\times3\times160$ & $1\times10^{-3}$ & 0.682638 & 0.682760 \\
\bottomrule
\end{tabular}
\end{center}

\begin{figure}[ht]
\begin{center}
\noindent\includegraphics*[width=0.55\textwidth]{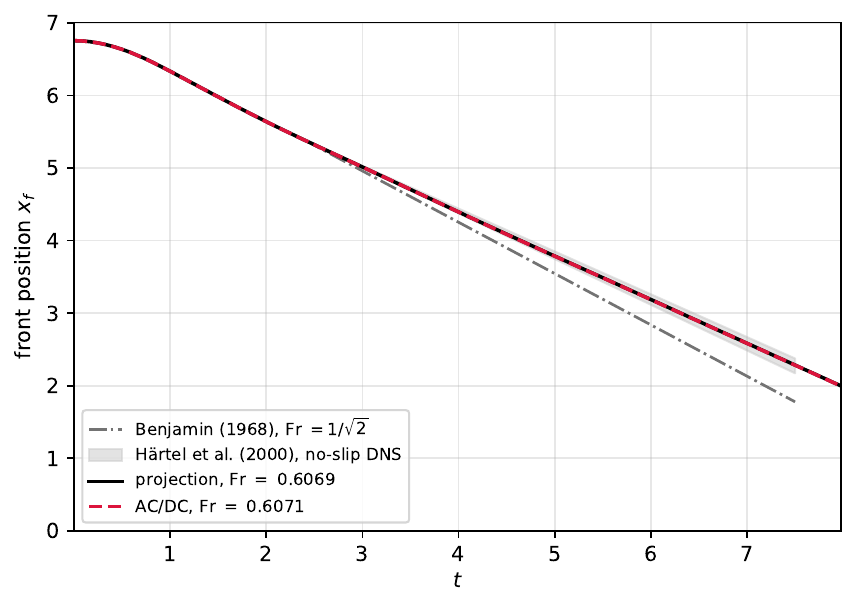}
\caption{\it\small Experiment~\ref{sec:exp_lock}: the front position
  $x_f(t)$ on the two methods, against the two external references.  The
  Froude numbers in the legend are fitted over the quasi-steady interval
  $t \in [2.4, 7.5]$, over which the front speed is flat to within
  $2\%$; only there is a steady front speed what is being measured,
  therefore Benjamin's line is drawn over that interval alone and is
  anchored to the computed front where it begins.  The grey band is the no-slip DNS of \citet{Hartel2000} at a
  comparable Grashof number, $\mathrm{Fr} = 0.59$--$0.63$; it is a
  graphical read of their figure~4 and not a tabulated value, and the
  width of the band reflects that.  The two curves lie on top of one
  another to the width of the line.  Both travel more slowly than
  Benjamin's energy-conserving bound, as a viscous current against
  no-slip walls must, so their trajectories lie \emph{above} his line in
  the figure, and both fall inside the DNS band.}
  \label{fig:E1_front}
\end{center}
\end{figure}
\noindent
The two methods agree on the front Froude number to $1.5\times10^{-4}$
and $1.2\times10^{-4}$ respectively, while changing the mesh moves it by
$1.5\times10^{-2}$: the difference between the two methods is two
orders of magnitude below the difference between two meshes.  On the
quasi-steady window the same fine-mesh pair gives
$\mathrm{Fr} = 0.606930$ against $0.607133$, a difference of
$2.0\times10^{-4}$, and both fall inside the DNS interval quoted above
and below Benjamin's inviscid bound, as Figure~\ref{fig:E1_front}
shows.
Pointwise the fields differ far more, at $t=8$ the density fields
differ by up to $3.2\times10^{-2}$ of the unit density jump, on the
rolled-up interface, and they must, since a sharp front's position is
not a convergent pointwise quantity.  That difference grows through the run, reaching $1.1\times10^{-2}$ at
$t=1.78$ and $2.1\times10^{-2}$ at $t=4.44$, and stays on the interface throughout, which is what
Figure~\ref{fig:E1_fields} shows; in the horizontal velocity it reaches
only $2.5\times10^{-3}$ of that field's range at $t=8$. The carried tracer content is conserved to
$1\times10^{-15}$ relative to projection and to zero for
AC/DC, and the pseudo-mass carried by the tracer agrees with
$\rhoa$ from the flux-form balance to $1.2\times10^{-10}$, so
hypothesis~(H1) of Proposition~\ref{prop:tracer_consistency} holds to
round-off in the configuration these budgets are read in.

A three-times refinement of the time step at fixed mesh
($\Delta t = 2, 1, 0.5 \times10^{-3}$) separates the time-discretisation
error from the spatial one.  Both methods converge at the order of the
integrator, observed orders $0.97$ and $0.90$ in $\mathrm{Fr}$, and
$1.02$ on both in the kinetic energy, and, extrapolated to
$\Delta t\to0$ on the same mesh, they give the same answer:
$\mathrm{Fr}=0.695636$ against $0.695590$, a relative difference of
$6.6\times10^{-5}$, and $6.4\times10^{-6}$ in the kinetic energy.  The
gap between projection and AC/DC itself falls as $\OO(\Delta t)$ (successive ratios
$2.13$ and $2.06$ in the kinetic energy), so what separates them at
finite $\Delta t$ is a splitting error and not the compressibility.  The
elastic reservoir $\tfrac{1}{2\alpha\rho_0}\sum_C|C|\psi_C^2$
of~\eqref{eq:EKL_h} confirms this: across the refinement sequence $\alpha$
quadruples at each halving of $\Delta t$ and the peak elastic energy
falls by factors of $4.0003$ and $4.0001$, i.e.\ exactly as
$\alpha^{-1}$.  In the production run of Figure~\ref{fig:E1_budgets}
the reservoir peaks at $2.02\times10^{-4}$, which is $0.061\%$ of the
peak kinetic energy, and it is identically zero for projection, which has
no such energy reservoir.  The
artificial compressibility is therefore a bounded, accounted reservoir.

\begin{figure}[ht]
\begin{center}
\noindent\includegraphics*[width=0.95\textwidth]{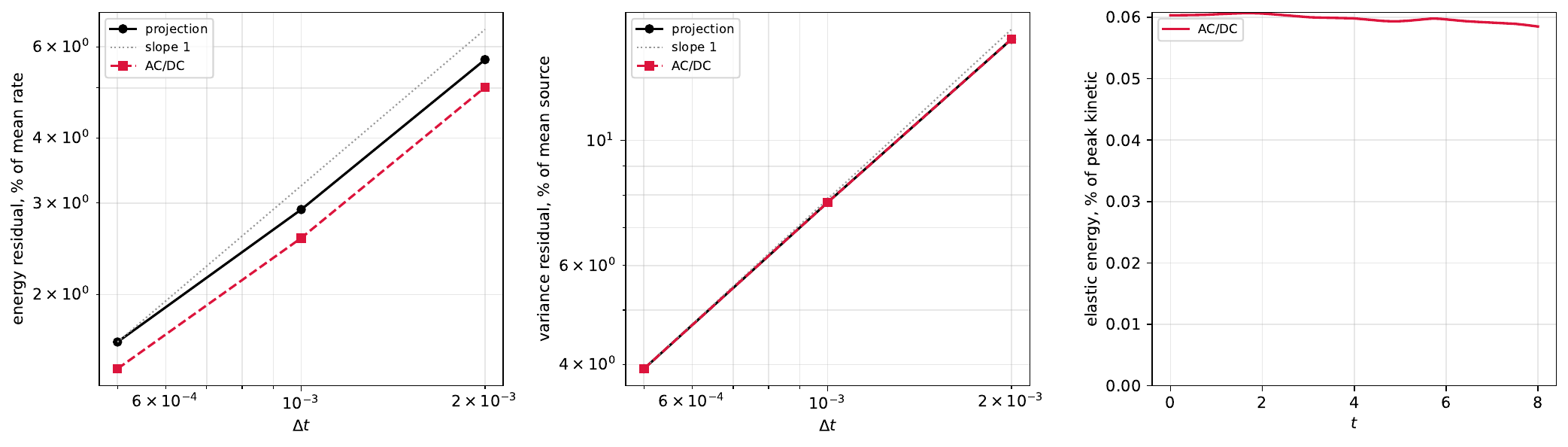}
\caption{\it\small Experiment~\ref{sec:exp_lock}: the two budget
  closures and the elastic reservoir.  \emph{Left and centre:} the
  residual of the discrete total-energy budget~\eqref{eq:EKL_budget_h},
  as a percentage of the mean rate, and of the tracer-variance
  budget~\eqref{eq:variance_budget}, as a percentage of the mean source,
  both against the time step on logarithmic axes with a guide line of
  slope one.  Both runs
  fall at the order of the integrator, and the two lie close enough
  that neither method can be said to close its budget better.
  \emph{Right:} the elastic reservoir
  $\tfrac{1}{2\alpha\rho_0}\sum_C|C|\psi_C^2$ of~\eqref{eq:EKL_h}, as a
  percentage of the peak kinetic energy, over the run. The reservoir stays
  bounded and small for the whole run.}
  \label{fig:E1_budgets}
\end{center}
\end{figure}
\paragraph*{B.\ Energy, variance and the mixing efficiency.}
\emph{Budgets.}  The discrete total energy budget~\eqref{eq:EKL_budget_h}
and the tracer-variance budget~\eqref{eq:variance_budget} were evaluated
at every step.  Both are semi-discrete identities, so the
fully discrete residual is the time-integration error and the test is its
behaviour under $\Delta t$-refinement, not its absolute size.  The energy
residual $|dE_h/dt-\text{rate}|$, as a fraction of the mean rate, runs
$5.67\%,\,2.92\%,\,1.62\%$ for projection and
$5.01\%,\,2.57\%,\,1.44\%$ on AC/DC as $\Delta t$ falls through
$2,\,1,\,0.5\times10^{-3}$, an observed order of $0.90$ on both runs.
The variance residual, as a fraction of the mean source, runs
$15.1\%,\,7.8\%,\,3.9\%$ over the same refinement sequence at an observed order of
$0.97$, and the two methods are indistinguishable in it: they agree to
three decimal places at every step, so the tracer budget does not
distinguish them at all.  On the
production run itself the energy residual is $2.66\%$ for projection
and $2.41\%$ for AC/DC.
Their observed orders are $\approx0.9$ and $1.03$: both budgets close at
the rate of the time integrator, which is
what Propositions~\ref{prop:variance}(i) requires and what excludes a residual
due to the compressibility error.  The numerical
sink $\mathcal{D}_{\mathrm{num}}$ of~\eqref{eq:chi_disc} was
non-negative at every step of every run, as
Proposition~\ref{prop:variance}(iii) requires of a monotone
reconstruction.  Under $\Delta t$-refinement $\mathcal{D}_{\mathrm{num}}$
converges to a  limit, $84\%$ of its value at the
production step, and $\int_\Omega\chi\dd V$ to $97\%$ of its own: the
numerical variance sink is a property of the spatial reconstruction, as
the proposition says, and does not shrink with the step.

\emph{Mixing.}  Corollary~\ref{cor:osborn_cox} makes both diffusivities
computable.  With the mean gradient $\partial_z\overline{\rho}$
taken as a least-squares slope over the interior of the column, at $t=8$:

\begin{center}
\begin{tabular}{lcccc}
\toprule
Mesh & Method & $K_{\mathrm{expl}}$ & $K_{\mathrm{num}}$ & $K_{\mathrm{num}}/\kappa$ \\
\midrule
$240\times3\times80$  & projection & $6.855\times10^{-3}$ & $4.177\times10^{-3}$ & 12.14 \\
$240\times3\times80$  & AC/DC      & $6.827\times10^{-3}$ & $4.150\times10^{-3}$ & 12.06 \\
$480\times3\times160$ & projection & $8.587\times10^{-3}$ & $1.169\times10^{-3}$ &  3.40 \\
$480\times3\times160$ & AC/DC      & $8.589\times10^{-3}$ & $1.168\times10^{-3}$ &  3.40 \\
\bottomrule
\end{tabular}
\end{center}

\noindent
with $\kappa=3.44\times10^{-4}$ the configured diffusivity.  Two remarks
follow.  First, the spurious mixing is the dominant one: the scheme
supplies an effective diffusivity twelve times the prescribed value on
the coarser mesh, falling to $3.4$ times on the finer, this is the spurious-mixing
problem for which the lock exchange is the standard test
\citep{BurchardRennau2008, Ilicak2012}.  Because
$\mathcal{D}_{\mathrm{num}}$ is an explicit face sum, and not  a
residual, that statement is computed.  Second,
and this is the point for AC/DC: the two methods agree on
$K_{\mathrm{expl}}$ to $0.41\%$ and on $K_{\mathrm{num}}$ to $0.65\%$ on
the coarser mesh, and to $0.03\%$ and $0.09\%$ on the finer, with the
convention-free ratio $K_{\mathrm{num}}/K_{\mathrm{expl}}$ agreeing to
$0.25\%$ and $0.12\%$.  The agreement improves with refinement.  Artificial compressibility therefore does not alter the mixing the model
performs..

\emph{The mixing efficiency.}  That same run also fixes
$\Gamma$ itself, the constant of Section~\ref{sec:efficiency} that
parameterisations assume.  The irreversible mixing is
read from the background potential energy, the energy of the adiabatically
sorted state \citep{Winters1995}: unlike the ordinary potential energy,
which rises and falls reversibly as the front sloshes, $E_b$ rises only
when fluid is mixed across density surfaces, so
$\varepsilon_b = dE_b/dt$ and $\Gamma = \varepsilon_b/\varepsilon$.  This experiment is the simplest case for it: unforced, with mixing driven by
shear at the interface, which is the regime the Osborn relation assumes.
On the finer mesh at $t=8$:

\begin{center}
\begin{tabular}{lccc}
\toprule
Method & $\Gamma$ (viscous $\varepsilon$) & $\Gamma$ (total $\varepsilon$)
       & $R_f$ \\
\midrule
projection & $0.1878$ & $0.1830$ & $0.1547$ \\
AC/DC      & $0.1879$ & $0.1836$ & $0.1551$ \\
\bottomrule
\end{tabular}
\end{center}

\noindent
The two methods agree to $0.05\%$ and $0.3\%$ in the two columns.  The value matches the canonical $\Gamma\approx0.2$, computed here from the sorted
state and the dissipation.
Figure~\ref{fig:E1_mixing} shows the ratio over the run: it settles early
and the two curves lie on one another to the width of the line.

\begin{figure}[ht]
\begin{center}
\noindent\includegraphics*[width=0.62\textwidth]{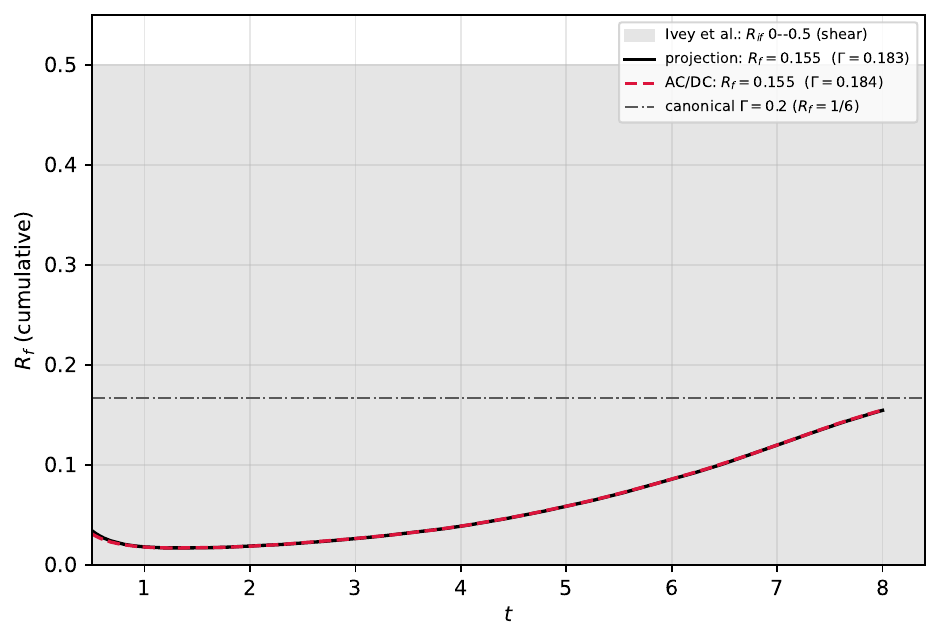}
\caption{\it\small Experiment~\ref{sec:exp_lock}: the flux Richardson
  number $R_f = \varepsilon_b/(\varepsilon_b+\varepsilon)$ on the two
  runs, with the numerical dissipation included in $\varepsilon$.  The ratio is \emph{cumulative} (the mixing integrated from the start of the
run over the dissipation integrated with it).  The instantaneous ratio is not usable here:
  early in the run it divides one near-zero quantity by another and
  swings across the whole axis.  For the same reason the horizontal axis
  begins at $t=0.5$, after the ratio has settled; the integral itself
  still starts at $t=0$ and no value depends on where the axis is cut.
  The dash-dotted line is the canonical $\Gamma = 0.2$, which
  through~\eqref{eq:gamma_Rf} is $R_f = 1/6$, and the shaded band is the
  range \citet{Ivey2021} report for shear-dominated mixing.  The final
  $\Gamma$ of each run is given in the legend, since it is $\Gamma$ that
  the parameterisation literature quotes and $R_f$ that stays a fraction
  in either regime.  Compare Figure~\ref{fig:E3_mixing}, the same
  diagnostic on a convecting flow.}
  \label{fig:E1_mixing}
\end{center}
\end{figure}

Which mixing regime this is can also be measured.
\citet{Ivey2021} separate the mechanisms by the gradient Richardson
number $\mathrm{Ri} = N^2/S^2$, shear-driven mixing dominating below
$0.25$ and convective mixing above $1.0$.  Evaluated over the interfacial cells (those carrying the stratification,
where alone the ratio is meaningful), the median $\mathrm{Ri}$ rises from $0.24$ at $t=3.6$
through $0.30$ at $t=5.3$ to $0.41$ at $t=8$, with the fraction of the
interface below $0.25$ falling from $51\%$ to $28\%$.  The lock exchange
therefore begins at the shear-dominated boundary and matures into the
intermediate regime in which both mechanisms contribute.  The
$R_f = 0.155$ measured above lies inside the range of $0$ to $0.5$ that
\citet{Ivey2021} give for shear-dominated mixing, and near the canonical
$\Gamma \approx 0.2$; the upper limit of that range is an upper bound
attained only in the most efficient shear flows.  The two methods agree on these figures to three decimal
places at every instant.

What is put in the denominator differs between the two columns.  The first uses only the dissipation the
viscous operator performs; the second adds the numerical dissipation,
which the total energy budget returns as its leftover and which amounts to
$2.6\%$ and $2.4\%$ of the viscous part on the two methods.  A mixing
efficiency is a ratio to the energy the turbulence loses, so leaving the
numerical part out inflates it: an efficiency quoted without saying which
$\varepsilon$ is in its denominator is not a determined quantity.  That
the correction is small here is a property of this experiment, where the
prescribed viscosity does most of the work..

\subsection{Experiment 2: Wave dispersion --- internal gravity waves and the free surface}\label{sec:exp_igw}

The purpose of this experiment i to quantify the dispersion error of the projection
scheme and of AC/DC against the analytical
non-hydrostatic dispersion relation.  Direct time-integration of standing
internal-gravity-waves modes is sensitive to the initial-condition projection.  
We therefore use an operator-theoretic measurement: for a fixed buoyancy
mode shape $b_m(x,z)$, compute the discrete IGW frequency via
$\omega_d^2 = N^2 \langle b_m, L_H L^{-1} b_m\rangle$, where $L_H$
is the horizontal restriction of $L$ via the cell-volume mask.
This isolates the spatial discretisation error from the time-stepping
error.

\emph{Experimental Setup.}  Doubly periodic-x, wall-z box, $L_x = L_y = 8$,
$H = 4$.  Uniform stratification $N^2 = 10^{-4}$.  Horizontal
modes $n_x \in \{1, 2, 4, 8\}$ at fixed vertical mode $n_z = 1$. 
The spatial resoluton is $40 \times 40 \times 32$.


\emph{Comparison.}  Two operator-theoretic measurements:

(i) \emph{Projection:} $\omega_d^2 = N^2 \langle b_m, L_H L^{-1}
b_m\rangle$ using the exact discrete Poisson inverse.  This is the
scheme's pure spatial-discretisation error.

(ii) \emph{AC/DC analytical:} the AC/DC dispersion relation from
Proposition~\ref{prop:hybrid_dispersion} is the quartic
$\omega^4/\tilde\alpha - K^2\omega^2 + k_x^2 N^2 = 0$, with $K^2 = k_x^2 + k_z^2$.
We take the slow (IGW) root
$\omega^2_{\mathrm{AC/DC}} = (\tilde\alpha K^2/2)\bigl(1 -
\sqrt{1 - 4 k_x^2 N^2 / (\tilde\alpha K^4)}\bigr)$.  At given
$\tilde\alpha$, this is exact: it has no spatial discretisation error
(it is the continuous AC/DC limit) and isolates the AC closure
error, which is $\OO(1/\tilde\alpha)$ per
Proposition~\ref{prop:hybrid_dispersion}.

\emph{Results.}  On the coarse mesh, $\tilde\alpha = 1$:

\begin{center}
\begin{tabular}{ccccc}
\toprule
$n_x$ & projection (quads) & projection (triangles) & AC/DC analytical & proposition formula \\
& $|\Delta\omega^2|/\omega^2_{\mathrm{NH}}$ & $|\Delta\omega^2|/\omega^2_{\mathrm{NH}}$ & $|\Delta\omega^2|/\omega^2_{\mathrm{NH}}$ & $k_x^2 N^2 / (\tilde\alpha K^4)$ \\
\midrule
1 & $5.5 \times 10^{-2}$ & $2.7 \times 10^{-1}$ & $4.1 \times 10^{-5}$ & $4.1 \times 10^{-5}$ \\
2 & $1.54\times 10^{-2}$ & $2.7 \times 10^{-2}$ & $2.6 \times 10^{-5}$ & $2.6 \times 10^{-5}$ \\
4 & $2.0 \times 10^{-3}$ & $6.5 \times 10^{-3}$ & $9.0 \times 10^{-6}$ & $9.0 \times 10^{-6}$ \\
8 & $9.1 \times 10^{-3}$ & $4.5 \times 10^{-4}$ & $2.5 \times 10^{-6}$ & $2.5 \times 10^{-6}$ \\
\bottomrule
\end{tabular}
\end{center}

\begin{figure}[H]
  \centering
  \includegraphics[width=0.75\textwidth]
    {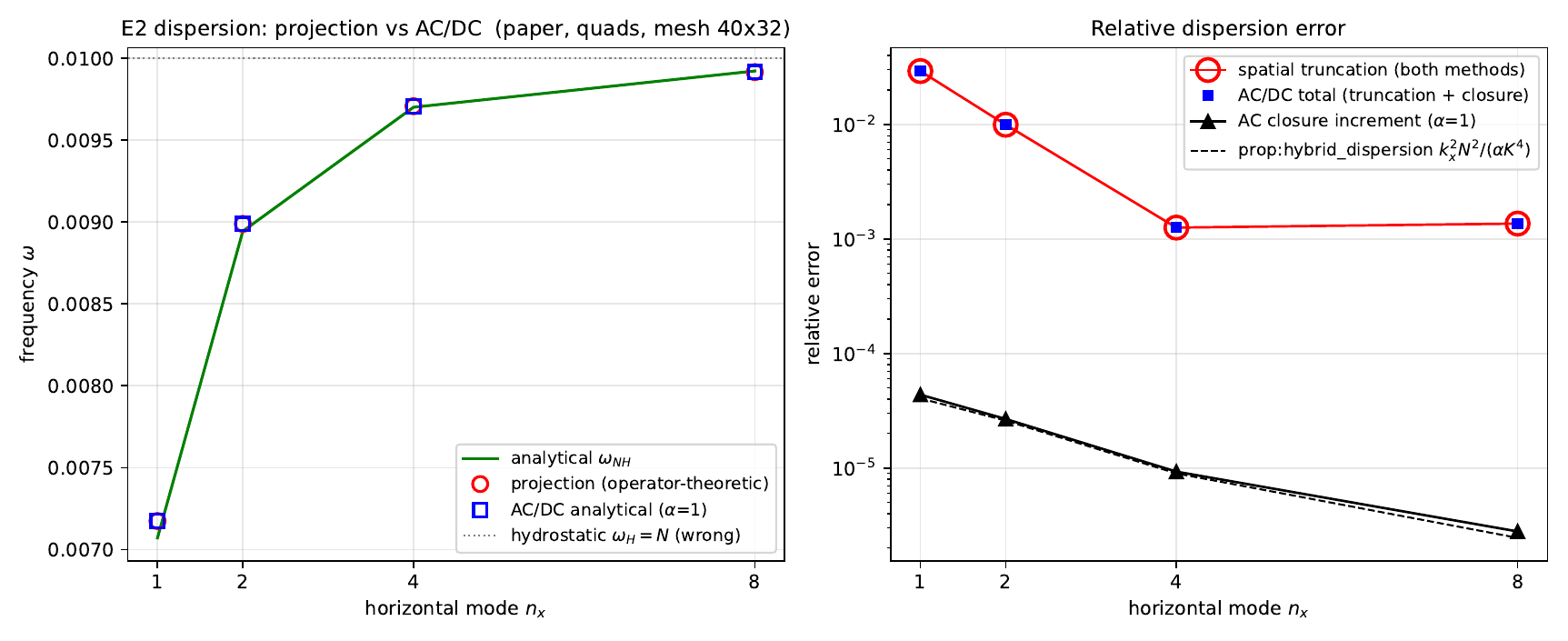}
  \caption{\it\small E2 at resolution ($40\times40\times32$),
    $\tilde\alpha = 1$.  Left: measured frequency against horizontal
    mode number, with the analytical non-hydrostatic curve and the
    hydrostatic $\omega = N$ for reference.  Right: relative error in
    $\omega^2$.  The spatial truncation (open circles) and the AC/DC
    total (filled squares) coincide, because AC/DC's frequency is the
    discrete incompressible frequency plus the closure increment
    (triangles), which lies three orders of magnitude below.  
    }
  \label{fig:e2_dispersion}
\end{figure}

\noindent
The last two columns match to all reported digits across all four
modes: the operator-theoretic AC/DC measurement \emph{exactly}
reproduces the analytical formula
$k_x^2 N^2 / (\tilde\alpha K^4)$ from
Proposition~\ref{prop:hybrid_dispersion}.  This is direct numerical
verification of the proposition; on the log-error plot, the dashed reference line for the formula lies exactly under the AC/DC curve.

The two error sources are independent:

\begin{itemize}
\item Projection error ($5.5 \times 10^{-2}$ at the longest mode,
$\OO(10^{-2}\text{--}10^{-3})$ at $n_x \geq 2$, on the coarse quad
mesh): pure spatial truncation, expected to drop as $h^2$ under
refinement.
\item AC/DC closure error ($\OO(10^{-5})$ at $\tilde\alpha = 1$):
drops as $1/\tilde\alpha$, so at the production value
$\tilde\alpha = gH_{\mathrm{full}} \approx 4\times10^{4}\,
\mathrm{m^2\,s^{-2}}$ of~\eqref{eq:alpha_choice}, four to five orders
above the $\tilde\alpha = 1$ used here, it falls to $\OO(10^{-10})$,
far below the spatial error.
\end{itemize}


\paragraph*{Free-surface standing waves.}\label{sec:exp_seiche}

\emph{Purpose.}  Exercise the free-surface machinery of
Sections~\ref{sec:free_surface} and~\ref{sec:variant_split}, which
none of E1--E4 activates, and test the most basic non-hydrostatic
free-surface phenomenon: the dispersion of surface gravity waves.  The
exact and the hydrostatic answer are both closed-form, so the
comparison against hydrostatic dynamics is supplied analytically,
without a hydrostatic model run.  The long-wave members of the sweep
retain, in addition, the null test for the pseudo-pressure.

\emph{Setup.}  Rectangular basin, walls on the lateral and bottom
boundaries, free surface above.  No stratification ($N^2=0$), no
rotation, no forcing.  One run per mode: the initial state is
$\eta(x,0) = \eta_0\cos(k_n x)$ with $k_n = n\pi/L_x$, released from
rest, with $\eta_0/H \sim 10^{-3}$ to keep the dynamics linear.  The
mode numbers are chosen so that $k_nH$ spans from $\ll 1$ (the seiche
limit) to order one; the vertical resolution is chosen to resolve the
$\cosh\bigl(k_n(z+H)\bigr)$ structure of the deepest mode in the
sweep.  Runs are repeated for $\alpha$ separated by a decade.

\emph{Reference solution.}  The exact eigenfrequencies of standing
surface gravity waves in a walled basin,
\begin{equation}
  \omega_n^2 = g\,k_n\tanh(k_nH),
  \label{eq:surface_dispersion}
\end{equation}
and, on the same axis, the hydrostatic relation
$\omega_n^2 = g\,k_n^2H$, which is~\eqref{eq:surface_dispersion}
linearised in $k_nH$.  The gap between the two curves at
$k_nH \gtrsim 1$ is the non-hydrostatic signal, referenced entirely
against analytics.  In the long-wave members the two curves merge and
the correct pseudo-pressure is zero: any $\psi$ signal above the
level set by $(k_nH)^2$ is spurious there, which keeps the null
character of the test.

\emph{Diagnostics.}  (i) The measured frequency of each mode against
the exact curve~\eqref{eq:surface_dispersion} and against the
hydrostatic curve, across the sweep in $k_nH$.  (ii)
$\|\bar\psi\|$ and $\|\psi'\|$ as time series in the long-wave
members, expected to remain at the $\OO\bigl((k_nH)^2\bigr)$ level
predicted for the $\eta$--$\bar\psi$ coupling in
Section~\ref{sec:variant_split}.  (iii) The sub-step count required
for stability, confirming that the AC step imposes no restriction
beyond the barotropic one the model already has
(Section~\ref{sec:cfl_advantage}).

\emph{Predictions and falsification.}  The measured frequencies
should follow~\eqref{eq:surface_dispersion} within the spatial
truncation error and the compressibility correction controlled by
$\alpha$---the $\alpha$-decade repetition separates the two; this is
the experiment's principal falsifiable content, and an AC/DC
frequency that tracks the hydrostatic curve at $k_nH$ of order one
falsifies the non-hydrostatic content of the free-surface formulation
outright.  In the long-wave members $\|\psi\|$ should remain at the
$\OO\bigl((k_nH)^2\bigr)$ level throughout; a growing $\bar\psi$
signal indicates that surface height and depth-averaged
pseudo-pressure are wrongly coupled, and invalidates
the cost argument of Appendix~\ref{app:cost_detail}, which assumes
the barotropic solve is reused unchanged.

\emph{Results.}  The prediction is met, and
Figure~\ref{fig:e5_dispersion} shows it.  Across six standing modes
spanning $k_nH$ from $0.031$ to $1.005$ --- from the seiche limit to
the intrinsically non-hydrostatic end --- the measured periods follow
the exact relation to within $1.24\%$ on the projection run and
$1.05\%$ on the AC/DC run, the largest error falling at the shortest
wave, where the mesh has sixteen cells per wavelength.  Measured
against the hydrostatic relation the same periods are wrong by up to
$16.15\%$ and $15.93\%$.  The two methods therefore agree with the exact
dispersion and with each other, and depart from the hydrostatic curve
exactly where the theory says they must; the AC/DC run is the closer of
the two at every mode.

\begin{figure}[tbp]
  \centering
  \includegraphics[width=0.62\textwidth]
    {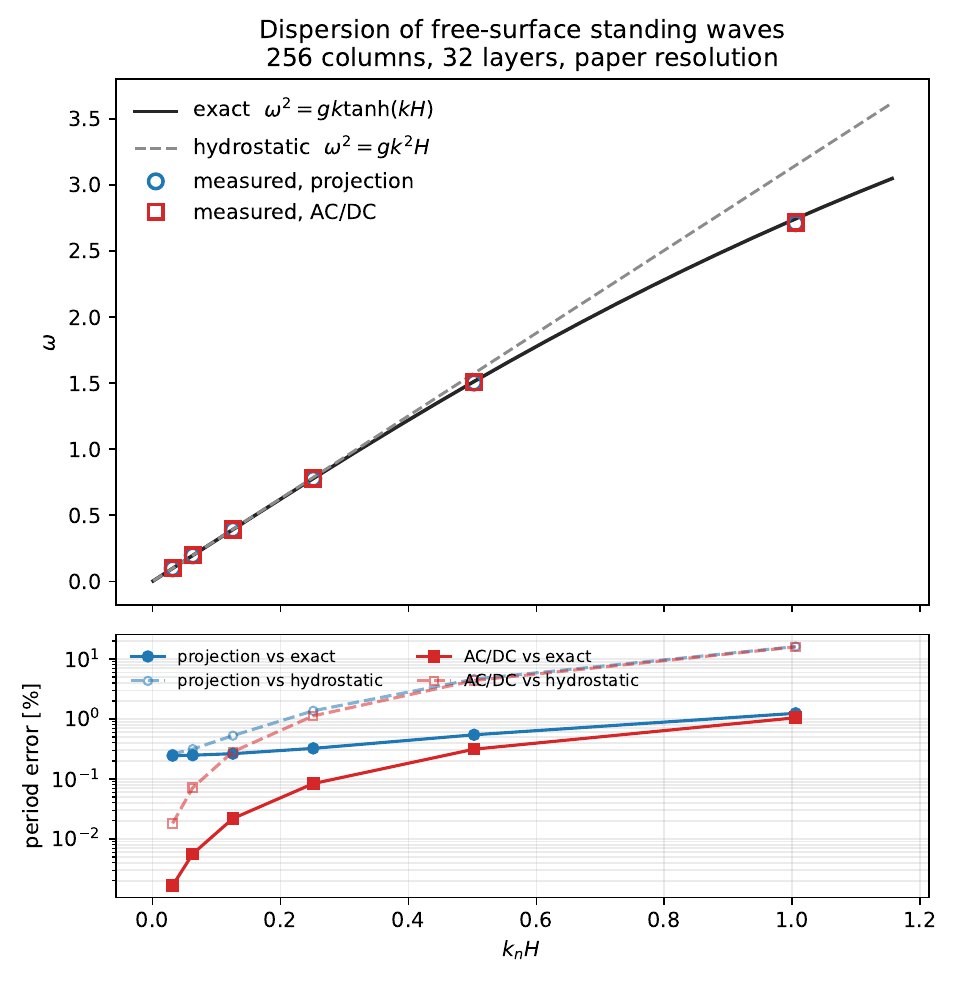}
  \caption{\it\small Experiment~\ref{sec:exp_seiche}: dispersion of
    free-surface standing waves, on a mesh of $256$ columns and $32$
    layers.  \emph{Top:} the measured frequency of each standing mode
    against $k_nH$, with the exact relation
    $\omega^2 = gk\tanh(kH)$ and the hydrostatic $\omega^2 = gk^2H$
    drawn for reference.  Both methods lie on the exact curve and leave
    the hydrostatic one as $k_nH$ approaches unity.  \emph{Bottom:} the
    period error of each run against \emph{both} references, on a
    logarithmic scale; the two solid curves are the error against the
    exact relation and the two pale dashed ones the error against the
    hydrostatic relation, and it is the separation between them that makes the experiment.  Measured against the exact relation the
    error stays at or below $1.24\%$ (projection) and $1.05\%$ (AC/DC)
    across the whole sweep; measured against the hydrostatic relation
    the same numbers grow to $16.15\%$ and $15.93\%$ at
    $k_nH\approx1$.  A method whose frequencies tracked the hydrostatic
    curve there would falsify the non-hydrostatic content of the
    free-surface formulation, and neither does.}
  \label{fig:e5_dispersion}
\end{figure}

\subsection{Experiment 3: Open-ocean deep convection}\label{sec:exp_convection}

\emph{Purpose.}  Demonstrate a genuinely three-dimensional
non-hydrostatic flow: surface-cooling-driven open-ocean deep
convection in the \citet{MarshallSchott1999} paradigm.  Convective
plumes are intrinsically non-hydrostatic (aspect ratio
$H/\ell_\mathrm{plume} \sim 1$); a hydrostatic model produces
spurious chimneys with no entrainment.

\emph{Setup.}  Non-dimensional Boussinesq formulation on the project's
Rayleigh--B\'enard infrastructure.  Domain $L_x = L_y = 2$, $H = 1$
(non-dim, $H$ as the natural length scale), with stable linear
background stratification (temperature increases upward,
$T_\mathrm{base}(z) = z$).  Rayleigh number $\mathrm{Ra} = 10^6$,
Prandtl number $\mathrm{Pr} = 1$.  A circular surface-cooling patch
of radius $r_p = 0.15\,L_x$ is centered at $(L_x/2, L_y/2)$ and
applies a buoyancy-loss source rate of $0.5$ (in $T$-units per
free-fall time) over the topmost cell layer.  Bottom: insulating
(no flux).  Coarse mesh $16 \times 16 \times 10$, fine mesh
$64 \times 64 \times 32$.  Time step $\Delta t = 10^{-5}$; integration to
$t = 0.04$ (about $40$ free-fall times, $1/\sqrt{\mathrm{Ra\,Pr}}$).

\emph{Reference solution.}  The convective velocity scale is
$w_* = (B_0 H)^{1/3}$; in non-dim,
$\sqrt{\mathrm{Ra\,Pr}} \approx 1000$ is the free-fall scale.
Plume reaches the bottom on timescale $t_{\mathrm{conv}} \sim H/w_*$.

The AC parameter requires careful calibration for this experiment.
The buoyancy-driven velocity scale is $\sqrt{\mathrm{Ra\,Pr}} = 10^3$,
so the AC acoustic speed must satisfy $c_\mathrm{ac} =
\sqrt{\alpha/\rho_0} \gg 10^3$, i.e.\ $\alpha \gg 10^6$.  An
$\alpha$-sensitivity test confirms this: at $\alpha \in \{10^5, 10^6, 10^7\}$,
AC/DC develops a runaway acoustic instability and NaNs by
$t \approx 0.012$--$0.015$.  At $\alpha = 10^8$, AC/DC runs to
completion and tracks projection closely.

\paragraph*{A.\ Projection against AC/DC.}
\emph{Results.}  At $t = 0.04$, $n_\mathrm{sub} = 1$
(quads), $n_\mathrm{sub} = 2$ (triangles):

\begin{center}
\begin{tabular}{lcccc}
\toprule
Mesh & Method & plume depth & $\max|w|$ & $\min T'$ \\
\midrule
quads     & projection & 0.8 & 1963 & $-1.00$ \\
quads     & AC/DC ($\alpha = 10^8$) & 0.8 & 2105 & $-1.00$ \\[2pt]
\bottomrule
\end{tabular}
\end{center}

\noindent

Appendix~\ref{app:alpha_law} collects the measurements that fix $\alpha$ for
this experiment: the $1/\alpha$ law, and the diagnostics bearing on the
approximation itself.  This section keeps to the method comparison and the
mixing analysis alone.

\begin{figure}[tbp]
  \centering
  \includegraphics[width=\textwidth]{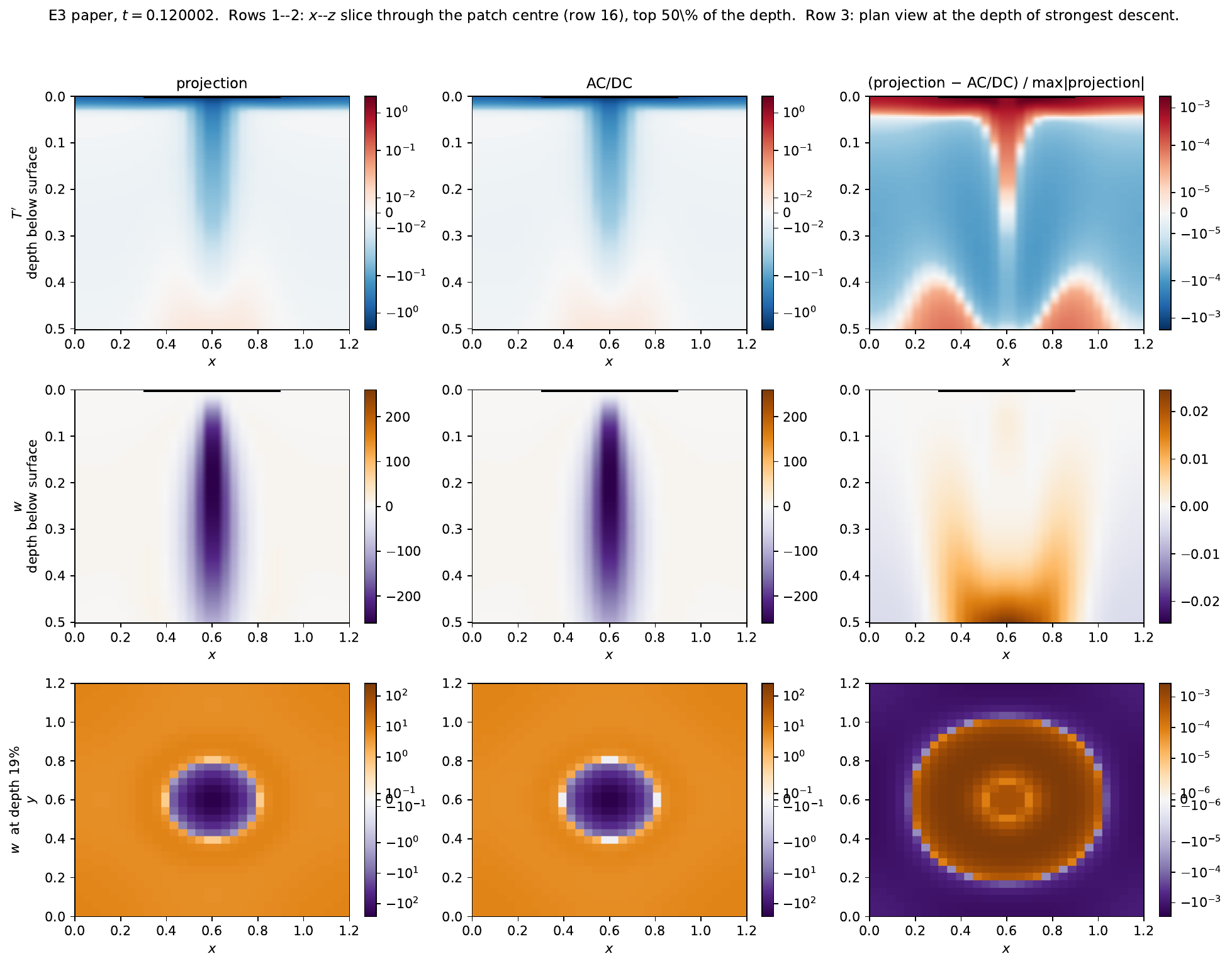}
  \caption{\it\small Experiment~\ref{sec:exp_convection} at $t=0.12$:
    projection (left), AC/DC (centre) and their difference normalised by the
    projection field's own maximum (right), so the third column reads as a
    fraction. Rows one and two are an $x$--$z$ section through the centre of the cooling patch (the buoyancy anomaly $T'$
and the vertical velocity $w$) over the upper half of the depth, with the
    patch marked on the surface.  Row three is a plan view of $w$ at the
    depth of strongest descent, which shows the whole circulation: a compact descending core inside
    a broad annulus of return flow.  $T'$ and the plan view are drawn on
    symmetric-logarithmic scales because they span decades from the surface
    layer to the deep field; $w$ in section is linear.}
  \label{fig:e3_arms}
\end{figure}

\begin{figure}[tbp]
  \centering
  \includegraphics[width=\textwidth]{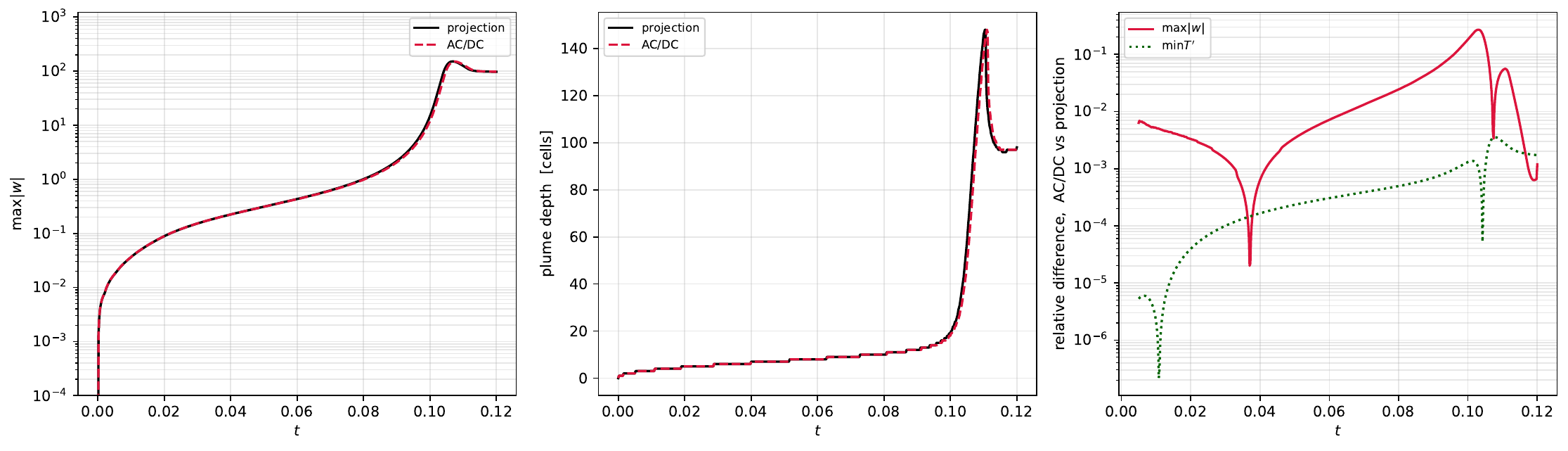}
  \caption{\it\small Experiment~\ref{sec:exp_convection} through the
    transition, on the interval both runs cover.  Left: the maximum vertical
    velocity, logarithmic, growing by an order of magnitude per $0.01$ time
    units at onset.  Centre: the plume depth in cells.  Right: the relative
    difference between the two runs for the two continuous diagnostics.  The
    plume depth is deliberately absent from the third panel: it is quantised
    to whole cells, so its relative difference is exactly zero while the
    runs agree and jumps when they differ by one cell of $267$.  The curves start at $t=0.005$;
    before that both runs have $\max|w| \sim 10^{-13}$.}
  \label{fig:e3_time}
\end{figure}

\emph{Agreement between the methods.}  Before
the onset of convection the runs differ by $0.95\%$ in $\max|w|$ and
$0.024\%$ in $\min T'$; after it saturates, by $0.18\%$ in both.  Between
those, at $t=0.1030$, the relative difference in $\max|w|$ reaches $27\%$.
That excursion is not a divergence of the solutions but a small offset in
when they cross the same trajectory: $\max|w|$ is growing by an order
of magnitude per $0.01$ time units there, so a shift of well under one per
cent of the elapsed time accounts for the whole of it, and $\min T'$, which is not
growing steeply, never differs by more than $0.4\%$ anywhere,
including through the same instant.  The plume depths differ by at most
$33$ cells of $267$ at the steepest moment and by $0.3$ cells once the flow
has saturated.  Figure~\ref{fig:e3_arms} shows the consequence in the
fields: the descending core, the return annulus and the anomaly beneath the
patch are the same structures in both runs, and the difference column is
everywhere a small fraction of what is there.

\begin{figure}[tbp]
  \centering
  \includegraphics[width=\textwidth]{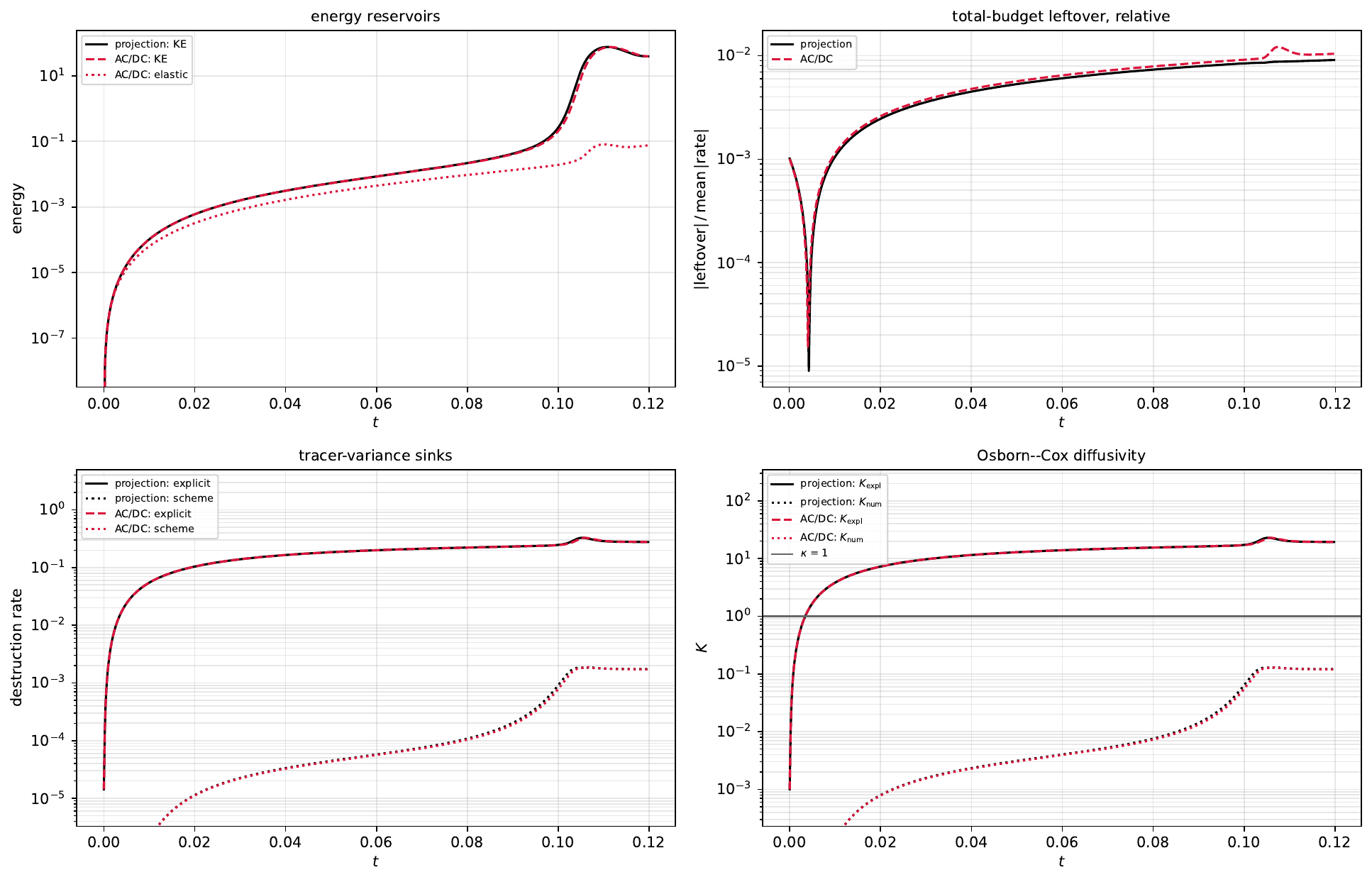}
  \caption{\it\small Energy, dissipation and mixing for the two methods of
    Experiment~\ref{sec:exp_convection}; all four ordinates logarithmic,
    projection in black and AC/DC in red.  \emph{Top left:} the energy reservoirs: kinetic energy for both methods, and the elastic
reservoir that exists only under artificial compressibility.  \emph{Top right:}
    the total energy budget's leftover, that is the rate of change of
    kinetic plus potential plus elastic energy minus the rate the measured
    terms account for, divided by the mean size of those terms; the ordinate
    is therefore the fraction by which the budget fails to close.
    \emph{Bottom left:} the two sinks of tracer variance separated by
    Proposition~\ref{prop:variance} - the prescribed diffusion
    $\tfrac12\int_\Omega\chi\,\dd V$ and the transport scheme's own
    contribution $\mathcal{D}_{\mathrm{num}}$.  \emph{Bottom right:} the
    same two sinks converted to diffusivities by
    Corollary~\ref{cor:osborn_cox}, with the configured $\kappa=1$ marked.}
  \label{fig:e3_budgets}
\end{figure}

\paragraph*{B.\ Energy, variance and the mixing efficiency.}
\emph{Budgets and the elastic reservoir.}  The total
energy budget's leftover stays at $0.61\%$ of the mean rate on the
projection run and $0.65\%$ on the AC/DC run, never exceeding $0.91\%$ and
$1.22\%$ respectively.  Elastic energy peaks at $0.21\%$ of the kinetic energy, which is the sense in
which the added compressibility perturbs the energetics.

\emph{Mixing efficiency in a convective regime.}
Background potential energy from the sorted state
\citep{Winters1995} is available here too, but three terms must be
separated before a ratio means anything, and this experiment shows why.

The first is the surface forcing.  A buoyancy flux changes the sorted
profile directly, without any mixing, so $dE_b/dt$ is not by itself the
irreversible mixing.  Measured on this run, that term is negligible after the first hundredth of a time unit: the cooled water
becomes the densest in the domain and its sorted height falls from $0.997$ to
$4\times10^{-4}$, and adding density to fluid already at the bottom of
the sorted column does not raise $E_b$.  The second is diffusion of the
background stratification, which raises $E_b$ with no turbulence at all:
in the quiescent window before the onset of convection, where the kinetic
energy is $7\times10^{-3}$ and the viscous dissipation is $61$, $E_b$
still rises at $2.1\times10^{4}$.  That rate is measured, not modelled,
and subtracted.  The third is the numerical dissipation, $4.6\%$ of the
viscous part on the projection run and $7.9\%$ on the AC/DC run, added to
the denominator as in Section~\ref{sec:exp_lock}.

What remains is the turbulent mixing, and with it
$R_f = 0.718$ on the projection run.  As in Section~\ref{sec:exp_lock} the ratio is cumulative (the integrated mixing over the integrated
dissipation), but integrated from the onset of convection since before onset there is no turbulence.  The two methods differ by $3.3\%$ on
the viscous-only form; the comparison between the two cannot be carried into
the corrected form, because on an AC/DC run the total-budget leftover
contains the constraint's work as well as the numerical dissipation, and
separating the two requires the pressure-work terms
of~\eqref{eq:EKL_budget_h}, which are not evaluated here.  The corrected
value is therefore quoted for the projection run, and the AC/DC value is
an upper bound.  Through~\eqref{eq:gamma_Rf} that is $\Gamma = 2.55$, an order
of magnitude above the canonical $0.2$, and the reason is the regime: in this experiment the buoyancy
conversion is a source of kinetic energy, $+1.8\times10^{4}$
against a viscous dissipation of $1.8\times10^{4}$, whereas the Osborn
relation is derived for shear-driven turbulence in which buoyancy is a
sink.  Where the turbulence is fed by buoyancy, $\Gamma$ is not bounded
by unity and $R_f$ is the form that remains a fraction; the two are the
same measurement through~\eqref{eq:gamma_Rf}.  The value is also the
one the observational and numerical literature reports for this
mechanism.  \citet{Ivey2021}, separating the mechanisms by Richardson
number, conclude that when $\mathrm{Ri} > 1.0$ ``the mixing is dominated
by convection\ldots and $R_{if}$ is in the range from $0$ to $0.75$'',
against a corresponding range of $0$ to $0.5$ for shear-dominated
mixing.  The $0.718$ computed here, in a flow driven entirely by surface
buoyancy loss, lies within the convective range and close to its upper limit.  Taken with
Section~\ref{sec:exp_lock}, where the same diagnostic on a shear-driven
flow returns $\Gamma = 0.183$, the pair shows the efficiency computed in
two regimes and departing from the assumed constant in one of them.
Figure~\ref{fig:E3_mixing} shows it over the convecting phase.

\begin{figure}[ht]
\begin{center}
\noindent\includegraphics*[width=0.62\textwidth]{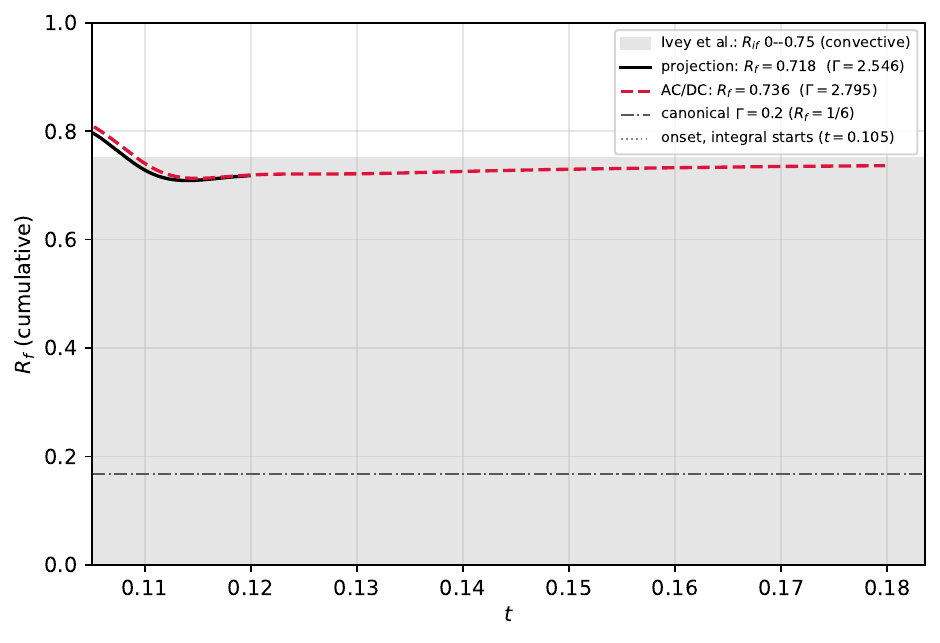}
\caption{\it\small Experiment~\ref{sec:exp_convection}: the flux
  Richardson number $R_f = \varepsilon_b/(\varepsilon_b+\varepsilon)$ on
  the two methods, drawn exactly as in Figure~\ref{fig:E1_mixing} so the two
  regimes can be set beside one another, and with the same three corrections applied: the forcing's direct effect on the
sorted state, the laminar diffusion of the background stratification, and
the numerical dissipation in the denominator.  The dotted vertical line
  marks the onset of convection, where the integral begins: before it
  there is no turbulence for an efficiency to be the efficiency of, and
  integrating from $t=0$ would average the convecting phase with a
  quiescent one.  The shaded band is the range \citet{Ivey2021} report
  for convection-dominated mixing, $R_{if}$ from $0$ to $0.75$, four
  times the canonical value marked by the dash-dotted line.  The measured $R_f$ lies inside that band and near its top, as a vigorously
  convecting flow should, and the corresponding $\Gamma$ in the legend
  is far above unity; $R_f$ is therefore the form drawn here.  The AC/DC run's corrected value is an upper bound: its
  budget leftover includes the constraint's work as well as the numerical
  dissipation, and the two are not separated here.}
  \label{fig:E3_mixing}
\end{center}
\end{figure}

\emph{Mixing.}  Corollary~\ref{cor:osborn_cox} converts each variance sink
into the eddy diffusivity an observer would infer from it, and the two methods
give the same answer.  Once convection is established, $K_{\mathrm{expl}}$
settles at $20.2$ under projection and $20.4$ under AC/DC, twenty times the configured $\kappa$.  That factor is the
Cox number: the convection
sharpens the tracer gradients about twentyfold in mean square over the
background stratification.  The transport scheme's own contribution,
$K_{\mathrm{num}}$, settles at $0.124$ and $0.123$: measured against
$K_{\mathrm{expl}}$, the scheme adds $0.61\%$
and $0.61\%$ of the mixing.  It is positive at
every step of both runs, as Proposition~\ref{prop:variance}(iii) requires of
a monotone reconstruction.  The number a model usually cannot quote is here
not estimated by comparing resolutions or fitting a decay: it is read off a
budget that closes, and $\mathcal{D}_{\mathrm{num}}$ is an explicit face sum.


\subsection{Experiment 4: Rising Bubble on a Telescoping Grid}\label{sec:exp_bubble}

 Experiments
\ref{sec:exp_lock}--\ref{sec:exp_convection} establish that AC/DC and
projection produce the same physics, which is a statement about the
treatment of the constraint and not about meshes; every one of them runs on
a mesh of uniform spacing.  The argument of Section~\ref{sec:telescoping} is
that the advantage of removing the global solve is largest exactly where a uniform mesh cannot go: on a grid whose cells vary by a
large factor,
where an elliptic solve is worst conditioned and a local relaxation is
indifferent.  This experiment runs a rising thermal on such a grid and asks
whether the two methods still agree.

\emph{The mesh.}  A log--polar disc: an equilateral lattice in
$(\ln r,\theta)$ mapped by $z\mapsto e^{z}$, with a hexagonal core and the
hexagon blended to a circle over the outer rings.  Cells coarsen
monotonically outward by a factor of eight from an inner radius $r=0.2$ to
the wall at $r=1.6$, so the base triangles span a factor of $79$ in area,
and the construction produces no obtuse triangle at any grading.  Extruded over $48$ layers the mesh has $161{,}280$ cells.
Figure~\ref{fig:lognormalmesh} shows it: the hexagonal core, the blend to a
circular rim, and the monotone coarsening between them.  The fine region is
deliberately not finer than a uniform mesh would be there: the smallest cell
sets the time step and, for AC/DC, the acoustic sub-step count, so the
saving is entirely in the far field.  Both runs use the same mesh, and
the cell centres of the two archives agree to the last bit.

\begin{figure}[H]
  \centering
  \includegraphics[width=0.5\textwidth]
    {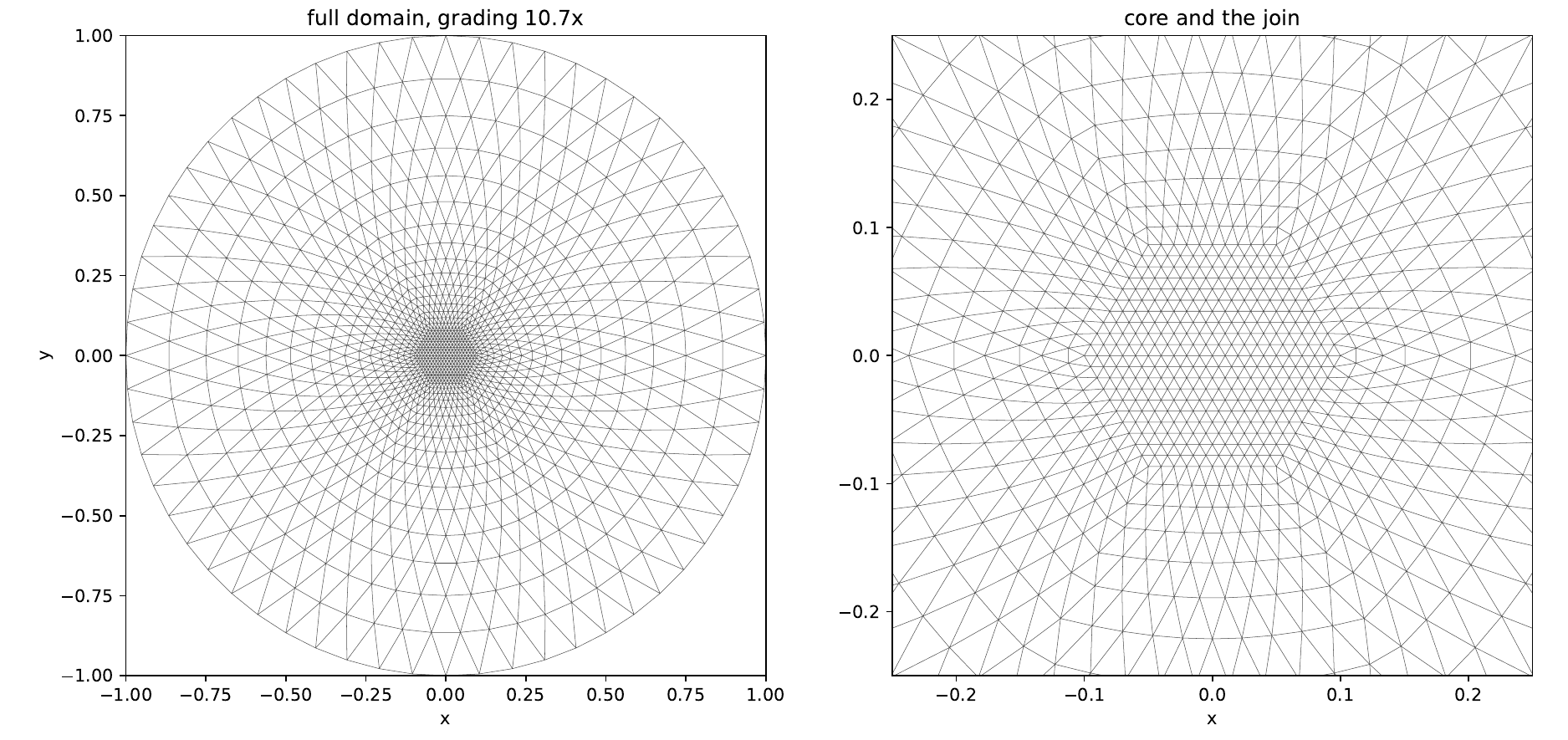}
  \caption{\it\small The telescoping log--polar disc used in
    Experiment~\ref{sec:exp_bubble}, drawn at reduced resolution for
    legibility.  Cells coarsen by a factor of eight from the fine core
    outward, a factor of $79$ in area; the construction is conformal in
    $(\ln r,\theta)$ and produces no obtuse triangle, so the circumcentric
    requirement of Section~\ref{sec:disc_time_impl} is met everywhere.  The
    hexagonal core and the blend to a circular rim are visible in the
    innermost and outermost rings.}
  \label{fig:lognormalmesh}
\end{figure}
\noindent

\emph{The configuration.}  A warm Gaussian thermal of amplitude $2$ and
radius $0.1$ is released at $z=0.15$ into a two-layer stratification,
near-neutral below an inversion at mid-depth and strongly stable above it,
and rises for two advective times.  The stratification is $N^2 = 0.25$
below the inversion and $3.0$ above it, with $g_{\mathrm{eff}} = 1$ and
$\nu = \kappa = 2.5\times10^{-4}$ on both runs.  Both integrate $400$ steps
of $\Delta t = 5\times10^{-3}$ to $t = 2$ with SSPRK3 for the slow terms
and a flux-corrected transport scheme for the tracer.  The AC/DC run runs
Option~S3-b, the explicit St\"ormer--Verlet form of~\eqref{eq:D_leapfrog}:
nothing is assembled and nothing is inverted at any point of the step.  It
uses $\tilde\alpha = 2500$ with $\theta = \tfrac12$ in the acoustic
sub-step, which puts the acoustic Courant number at $0.16$ and requires
$2290$ sub-steps per step, that count being set by the smallest cell of the
grading as Section~\ref{sec:telescoping} describes.

\begin{figure}[tbp]
  \centering
  \includegraphics[width=\textwidth]
    {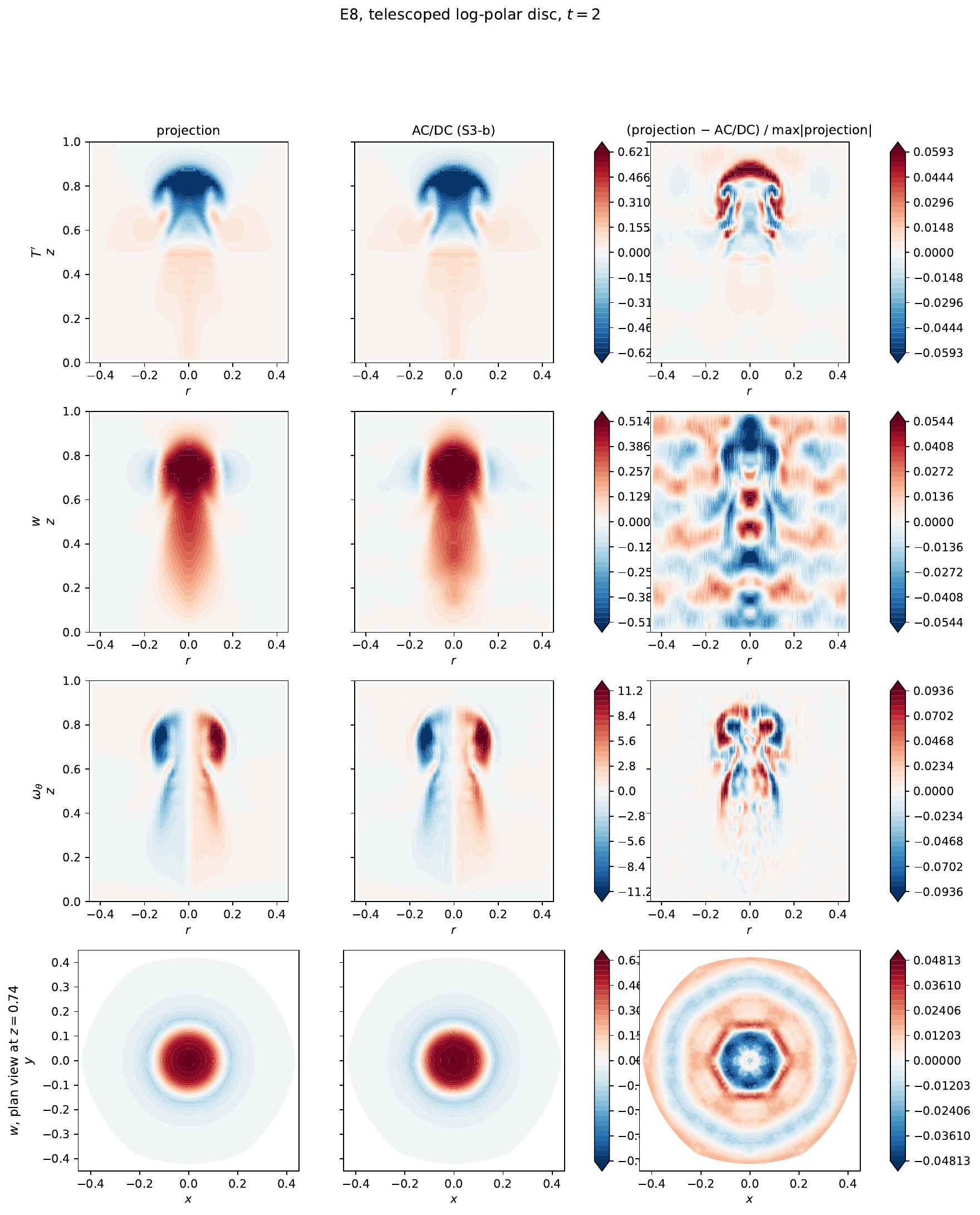}
  \caption{\it\small Experiment~\ref{sec:exp_bubble} at $t=2$: projection
    (left), AC/DC under Option~S3-b (centre), and their difference
    normalised by the projection field's own maximum (right).  Rows: the
    buoyancy anomaly $T'$; the vertical velocity $w$; the azimuthal
    vorticity $\omega_\theta$, whose two opposite-signed lobes are the
    rolled-up vortex ring; and $w$ in plan view at the depth of strongest
    ascent, which shows the whole circulation.  The disc is axisymmetric, so each
    ring is averaged and the meridional rows are mirrored about the axis.  Maximum
    differences: $14.6\%$ in $T'$, $8.5\%$ in $w$, $19.0\%$ in
    $\omega_\theta$ and $5.8\%$ in the plan view.  Note that the difference
    fields are aligned with the mesh (the plan view shows the hexagonal core and
the concentric rings) and not with the flow.}
  \label{fig:e8_arms}
\end{figure}

\paragraph*{A.\ Projection against AC/DC.}
\emph{Comparison of the two methods.}  The
thermal reaches the same height in both runs, $0.9063$, with the two rise
trajectories separated by at most one cell of $48$ over the whole ascent.
Measured over the whole field at $t=2$, and weighted by cell volume because
the cells span a factor of $79$ in area, the buoyancy anomalies differ by
$9.3\%$ in a relative $L^2$ norm and correlate at $0.9962$; their peak
values agree to $1.0\%$ at $t=1.30$ and $0.5\%$ at $t=1.60$, and to $4.6\%$
at the end.  In the same norm the vertical velocity differs by $19.3\%$ and
the azimuthal vorticity by $13.9\%$.

  The \emph{peak} of $\omega_\theta$ differs
by $19.0\%$ while the field differs by $13.9\%$ and its $99$th percentile by
$0.3\%$: that peak is a single cell in the core of the vortex ring, and away
from it the two methods agree closely, so the extremum overstates the difference.  The peak of $w$ does the reverse.  It agrees to $1.3\%$ while the field
differs by $19.3\%$ and its $99$th percentile by $9.4\%$, so there the extremum understates it. 

Figure~\ref{fig:e8_arms} makes the character of the residual visible.  The
difference between the methods follows the symmetry of the \emph{mesh} and
not of the flow: the plan view shows the hexagonal core and the ring
structure of the grading, and the meridional rows show banding at the layer
spacing.  The two methods agree about the thermal itself: head, stem, return flow
and vortex ring alike.


\begin{figure}[tbp]
  \centering
  \includegraphics[width=0.8\textwidth]
    {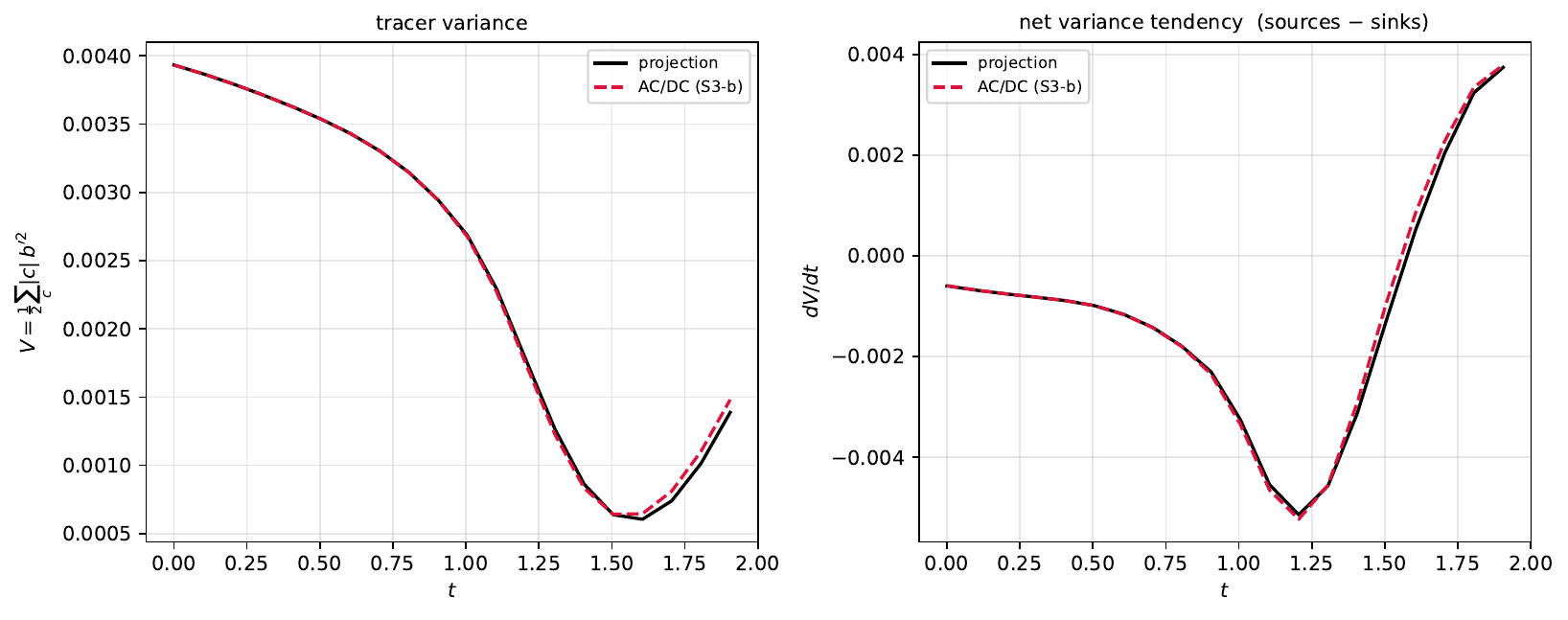}
  \caption{\it\small Tracer variance $V=\frac{1}{2}\sum_c|c|\,b'^2$ (left)
    and its net tendency (right) for the two methods of
    Experiment~\ref{sec:exp_bubble}.  The volume weighting is not optional
    here: the disc's cells span a factor of $79$ in area.  The variance
    falls to a minimum near $t\approx1.55$ and then rises, because the tracer has a source, the conversion $-w N^2$ against the
background stratification, as the thermal overshoots into the stable
    upper layer.  For that reason $dV/dt$ is a net tendency; the
    separation into sinks requires the discrete budget of
    Proposition~\ref{prop:variance} and is reported separately.  Both runs
    change sign at the same time and with the same shape.}
  \label{fig:e8_variance}
\end{figure}

\begin{figure}[tbp]
  \centering
  \includegraphics[width=\textwidth]{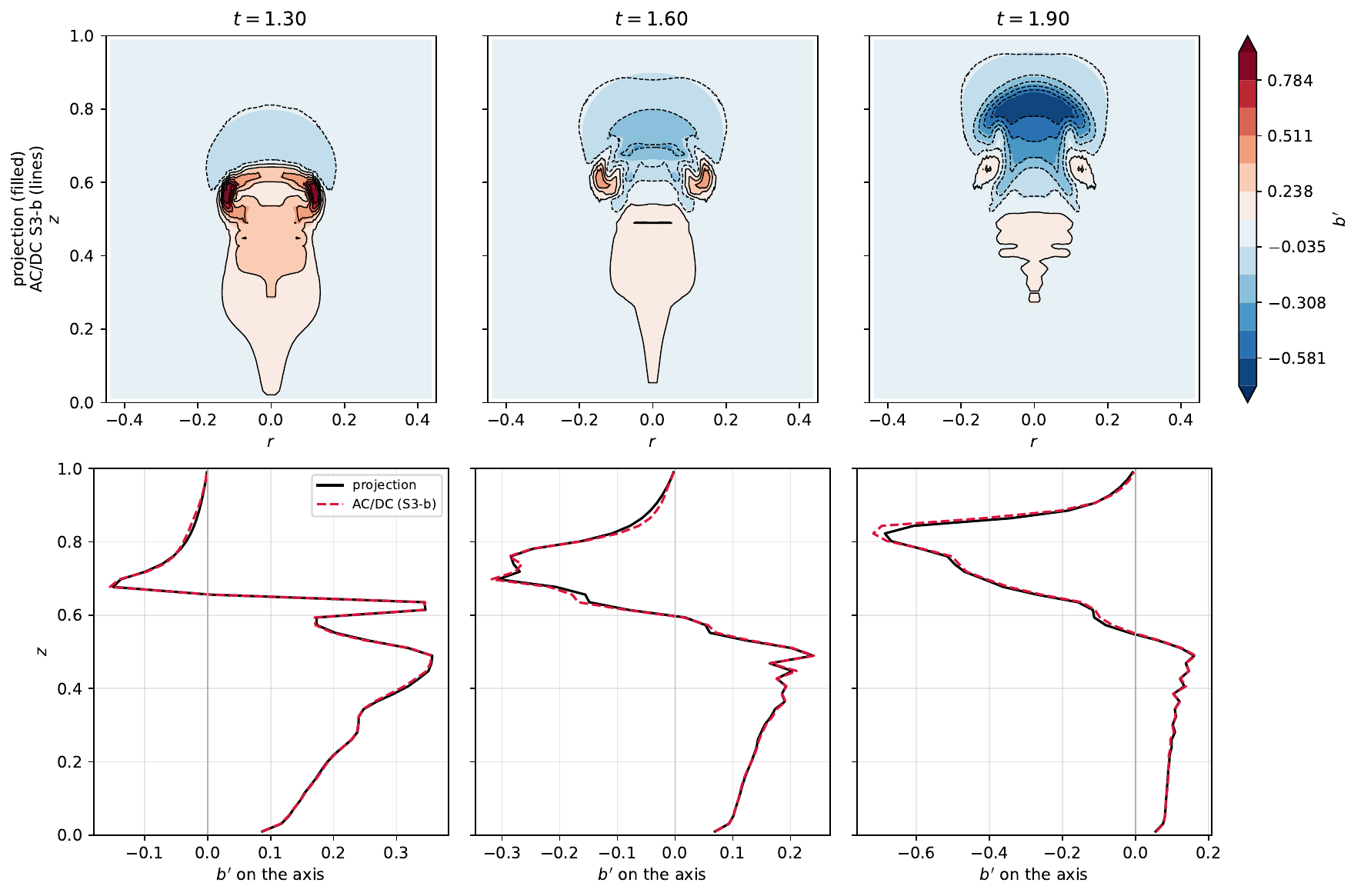}
  \caption{\it\small Experiment~\ref{sec:exp_bubble} in the display
    conventions of the rising-bubble literature.  \emph{Top:} the buoyancy
    anomaly of the projection run as filled bands, with the AC/DC contours
    at the same levels drawn over them, so that a line lying on a
    band edge means the two solutions agree at that level and any offset is
    the disagreement.  \emph{Bottom:} $b'$ along the axis, both runs.  Levels
    are uniformly spaced with the lowest below zero, following their
    Figs.~3 and~5; the values themselves are set by the field, since theirs
    are degrees Celsius for a $0.5\,^\circ$C bubble in a neutral atmosphere
    and $b'$ here is non-dimensional.  The axis profiles differ by
    $1.7\%$, $7.4\%$ and $12.7\%$ of the profile's own range at the three
    instants.}
  \label{fig:e8_contours}
\end{figure}

\paragraph*{B.\ Energy, variance and the mixing efficiency.}
\emph{Mixing.}  Over the run the projection run loses $64.7\%$ of the
initial tracer variance and the AC/DC run $62.4\%$: the artificial
compressibility mixes $3.7\%$ less, not more.  Equivalently, the
variance the AC/DC run still holds at the last sample exceeds the
projection run's by a factor $1.067$.  Figure~\ref{fig:e8_variance} shows why the comparison must be made
on the variance itself: this tracer has a
source, the stratification conversion, and the net tendency changes sign
near $t\approx1.55$ when the thermal overshoots.

\begin{figure}[tbp]
  \centering
  \includegraphics[width=\textwidth]
    {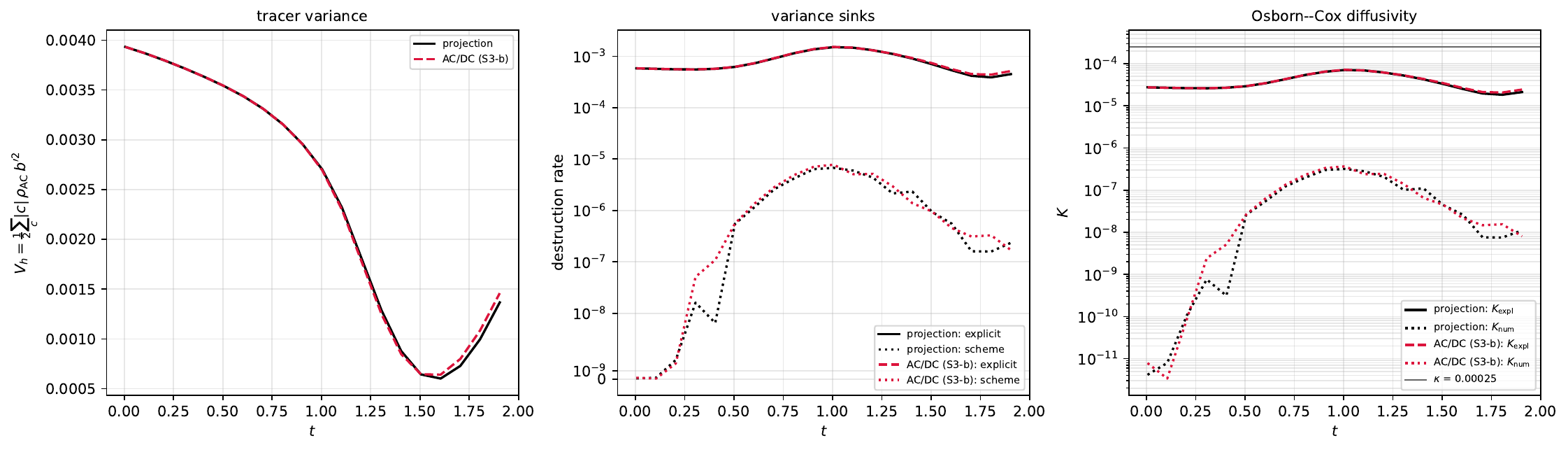}
  \caption{\it\small The tracer-variance budget of
    Experiment~\ref{sec:exp_bubble}; projection in black, AC/DC under
    Option~S3-b in red.  \emph{Left:} the weighted variance $V_h$.
    \emph{Centre:} the two sinks separated by Proposition~\ref{prop:variance}: the prescribed
diffusion
    $\tfrac12\int_\Omega\chi\,\dd V$ (solid) and the transport scheme's own
    contribution $\mathcal{D}_{\mathrm{num}}$ (dotted), the latter smaller by
    two to three orders of magnitude throughout.    \emph{Right:} the same two sinks as
    diffusivities by Corollary~\ref{cor:osborn_cox}, with the configured
    $\kappa$ marked.  The two $K_{\mathrm{expl}}$ curves lie on one another
    for the whole run.}
  \label{fig:e8_budgets}
\end{figure}

The discrete budget of Proposition~\ref{prop:variance} separates that total
into its two parts, and Corollary~\ref{cor:osborn_cox} turns each into the
eddy diffusivity an observer would infer from it.  Both are drawn in
Figure~\ref{fig:e8_budgets}.  Both runs agree:
$K_{\mathrm{expl}}$ ends at $2.1\times10^{-5}$ under projection and
$2.4\times10^{-5}$ under AC/DC, and the two curves are indistinguishable
over the run.  The transport scheme's own contribution is
$K_{\mathrm{num}}/K_{\mathrm{expl}} = 0.052\%$ and $0.033\%$ at the end,
never exceeding $0.52\%$ on either run at any time, and
$\mathcal{D}_{\mathrm{num}} \geq 0$ at every sample of both, as
Proposition~\ref{prop:variance}(iii) requires of a monotone reconstruction.

One number here reads oppositely to
Experiment~\ref{sec:exp_convection}.  
The Cox number $K_{\mathrm{expl}}/\kappa$ is about $0.09$ on
this experiment against about $20$ on the convection one: a compact thermal
in a domain filled by a background stratification has a mean-square
perturbation gradient smaller than the background gradient it is
measured against, whereas open-ocean convection fills its domain and
sharpens the gradients twentyfold.  

\emph{Summary.}  On a grid whose cells span a factor
of $79$ in area, AC/DC without any global operation reproduces the
projection solution: the same rise, the same head and vortex ring, the same
return flow, the same mixing to within a few per cent, with the residual
confined to the scale of individual cells.  Section~\ref{sec:telescoping} claims that the acoustic
stage of Option~S3-b assembles nothing and reduces nothing globally, and is
therefore indifferent to a grading that ill-conditions an elliptic solve.
That claim is tested here.

\section{Integration into General Ocean Circulation Model}\label{sec:ogcm}

\subsection{Hydrostatic--Non-Hydrostatic Decomposition}\label{sec:split}\label{sec:splitting}
 Operational ocean models are built around the primitive equations
with their hydrostatic approximation; a non-hydrostatic capability is
particularly valuable if it can be added to that existing core.
This section expresses the non-hydrostatic equations in the form of the primitive
equations with a non-hydrostatic correction. One key element
is the vector-invariant nonlinearity:
the contraction $(\bom+2\bOm)\times\bv$ is a vector, it decomposes
into a horizontal vector and a vertical scalar, the horizontal
part splits further into a piece that is already present in the
hydrostatic model and a genuinely non-hydrostatic correction.  Once
this decomposition is in hand, the non-hydrostatic algorithm is
written as hydrostatic core plus a correction, and
AC/DC is the device that closes the correction.  The single
approximation introduced anywhere in this section is the AC closure
of the non-hydrostatic pressure (introduced in
Section~\ref{sec:ac}).

The incompressibility constraint brings with it the total pressure as  Lagrange multiplier.  
The hydrostatic pressure is not a second multiplier: it is determined by
the buoyancy field alone.  Writing
$p = p_{\mathrm{hyd}} + p_{\mathrm{NH}}$ is therefore a shift of the
single multiplier by a known field, so that a large part of the answer
is supplied in advance and a smaller residual is left to be determined.
AC/DC splits the residual into
a vertical part obtained exactly by a tridiagonal column solve and a
horizontal part carried by the AC relaxation.  The three stages form one hierarchy.  Vertical integration removes the
buoyancy-driven part of the multiplier, a direct column solve removes
its vertical block, and only the horizontal remainder is left to an
iteration, which is what makes that iteration local.

\paragraph*{Hydrostatic and Non-Hydrostatic Ocean Equations}\label{sec:pe_subsystem}
We denote the 3D velocity field by $\bv=(\bv_H, w)$. Then
the non-hydrostatic system in full three-dimensional
vector-invariant form and the hydrostatic equations in vector-invariant form are given by
\begin{equation}\label{all-ocean_eq}
\begin{split}
  \Dt\bv + \mathbf{N}_{\mathrm{3D}} 
  + \nabla B_{\mathrm{3D}}
    &= \mathcal{D}(\bv) + b\,\bk,\\[2ex]
 \Div\bv &= 0, 
\end{split}
\qquad\qquad
\begin{split}
  \Dt\bv_H + \mathbf{N}_H^{\mathrm{hyd}}
    + \nabla_H B_H &= \mathcal{D}_H(\bv_H),\\  
  \partial_z\!\left(\frac{p_{\mathrm{hyd}}}{\rho_0}\right) &= b,\\
  \nabla_H\cdot\bv_H + \partial_z w &= 0,
\end{split}
\end{equation}
In the non-hydrostatic equations, $\mathbf{N}_{\mathrm{3D}}:=\bom_a\times\bv$ denotes the full three-dimensional rotational term, including the Coriolis force via the absolute vorticity $\bom_a:=\bom+\mathbf{f}$, with relative vorticity  $\bom:=\nabla\times\bv$, and $\mathbf{f}:=(f_H, f_V)$, the 3D Coriolis parameter with horizontal and vertical component; the Bernoulli function $B_{\mathrm{3D}}:=p/\rho_0+\tfrac12|\bv|^2 = p/\rho_0+\tfrac12(|\bv_H|^2 + w^2)$ contains the full pressure and the 3D kinetic energy; the
vertical velocity is prognostic, the buoyancy $b$ is from the equation of state.  
The operator $\mathcal{D}$ denotes a dissipation operator.

In contrast, the hydrostatic primitive equations are obtained by retaining the
nonlinearity in the form
$\mathbf{N}_H^{\mathrm{hyd}}
:= \bom_a^z(\bk\times\bv_H) + w\,\partial_z\bv_H$ and the Bernoulli function
\begin{equation}\label{eq:BH_def}
  B_H := \frac{p_{\mathrm{hyd}}+p_{\mathrm{sfc}}}{\rho_0}
       + \tfrac{1}{2}|\bv_H|^2
\end{equation}
in the horizontal momentum equation
balance, replacing the vertical momentum equation by hydrostatic
balance, and diagnosing $w$ from incompressibility.  The surface
pressure is carried inside $B_H$ because the hydrostatic core computes
it, being the barotropic pressure gradient of the host model, so that
the only pressure absent from the core is $p_{\mathrm{NH}}$.  The sole
difference between $B_H$ and the $B_{\mathrm{hyd}}$
of~\eqref{eq:B_hyd} is then the kinetic-energy contribution,
$\tfrac12|\bv_H|^2$ against $\tfrac12|\bv|^2$.

\paragraph{Pressure.}
We recall the decomposition~\eqref{eq:pressure_split} of the total
dynamic pressure $p$ into the hydrostatic pressure $p_{\mathrm{hyd}}$,
the surface pressure $p_{\mathrm{sfc}}$ and the non-hydrostatic
pressure $p_{\mathrm{NH}}$,
\begin{equation*}
  p = p_{\mathrm{hyd}} + p_{\mathrm{sfc}} + p_{\mathrm{NH}}.
\end{equation*}
This decomposition is standard;  new is the closure
of its third term, and the pressure split of
Section~\ref{sec:acdc_closure} that makes that closure local.
The hydrostatic pressure is determined via the hydrostatic relation
$\partial_z(p_{\mathrm{hyd}}/\rho_0) = b$, integrated
\emph{downward from the free surface},
\begin{equation} \label{eq:pressure_hydro}
p_{\mathrm{hyd}}(x,y,z,t)=\rho_0\int_{\eta}^z b(x,y,z',t)\, dz'
 = -\rho_0\int_{z}^{\eta} b(x,y,z',t)\, dz',
\end{equation}
so that $\partial_z(p_{\mathrm{hyd}}/\rho_0)=b$ holds while
$p_{\mathrm{hyd}}$ vanishes at the free surface.  The surface pressure
is the weight of the displaced column,
\begin{equation}\label{eq:psfc_def}
  p_{\mathrm{sfc}} = \rho_0\,g\,\eta ,
\end{equation}
advanced by the barotropic free-surface solve of the host model
(Section~\ref{sec:disc_time_impl}); it is the pressure appearing as
$\nabla_H(g\eta)$ in the fast momentum
equation~\eqref{eq:substep_mom}.  Note that the total dynamic pressure
does not vanish at $z=\eta$: it equals $p_{\mathrm{sfc}}$ there,
and the consequence for the energy budget is the barotropic reservoir
(cf. Proposition~\ref{prop:free_surface}(iii)).

The challenge is the calculation of the non-hydrostatic pressure $p_{\mathrm{NH}}$. In the pressure-projection method the divergence operator is applied to the velocity equation; by incompressibility the time derivative vanishes and the resulting equation becomes a constraint to be satisfied at each time instant. Taking the divergence of the momentum equation in~\eqref{all-ocean_eq}, using $\Div\Dt\bv=\Dt\Div\bv=0$, and using the hydrostatic relation $\partial_z(p_{\mathrm{hyd}}/\rho_0)=b$ to cancel the vertical part of the hydrostatic pressure gradient against the buoyancy $b\,\bk$, the elliptic equation for $p_{\mathrm{NH}}$ follows
\begin{equation} \label{eq:pressure_laplace}
  \Delta\!\left(\frac{p_{\mathrm{NH}}}{\rho_0}\right)=\Div\Big(-\,\mathbf{N}_{\mathrm{3D}}
  -\nabla_3\!\left(\tfrac12|\bv|^2\right)
  -\nabla_H\!\left(\frac{p_{\mathrm{hyd}}}{\rho_0}\right)
  +\mathcal{D}(\bv) \Big).
\end{equation}
The surface-pressure contribution $-\nabla_H(p_{\mathrm{sfc}}/\rho_0)$ is
carried by the barotropic free-surface solve and is therefore not part
of this projection.  \citet{Marshall1997b}, working under a rigid lid, obtain the
surface pressure first and describe that solution as
provisional: their non-hydrostatic pressure does not vanish at
the upper surface, so solving for it subsequently refines the surface
value, and the two problems are coupled.  Under a free surface the
coupling is carried differently.  The surface elevation is a prognostic
variable advanced by the barotropic sub-stepping, not a constraint to
be inverted, so the provisional-and-refine structure of the rigid-lid
formulation need not appear.  
The purpose of the non-hydrostatic pressure is to adjust $p_{\mathrm{NH}}$ such that the ocean's volume is conserved, i.e., $\Div\bv =0$. 

\paragraph{Vertical velocity.}
In the hydrostatic approximation the vertical velocity changes its nature and becomes diagnostic, it is diagnosed 
from the continuity equation in \eqref{all-ocean_eq} by
vertical integration,
\begin{equation}
  w = w_{\mathrm{diag}}
    = -\int_{-H}^{z}\nabla_H\cdot\bv_H\dd z',
  \label{eq:w_diag}
\end{equation}
so that $w$ in the hydrostatic system is determined by $\bv_H$.

\paragraph{Nonlinear term.} The nonlinear term contains the momentum advection
$(\bv\cdot\nabla_3)\bv=\bom\times\bv + \nabla_3 \frac12|\bv|^2$. The hydrostatic approximation retains only those  terms that belong to the horizontal velocity equation: $(\bv\cdot\nabla_3)\bv|_H=\bom\times\bv|_H 
+ \nabla_H \frac12|\bv_H|^2$; the kinetic energy term contains only the prognostic horizontal velocity, and no longer the now diagnostic vertical velocity. The Coriolis force is treated under the so-called \emph{traditional approximation}, which retains only the vertical component of the Coriolis vector, the component belonging to the horizontal velocity equations.

\paragraph*{The non-hydrostatic system as hydrostatic core plus
  correction.}
The terms a classical hydrostatic ocean model must add to become non-hydrostatic are the following.
 Hydrostatic and non-hydrostatic  differ
in their nonlinearity ($\mathbf{N}_{\mathrm{3D}}$ versus
$\mathbf{N}_H^{\mathrm{hyd}}$), their Bernoulli function
($\tfrac12|\bv|^2$ versus $\tfrac12|\bv_H|^2$), the status
of $w$, and in the pressure term. This implies that the non-hydrostatic system~\eqref{all-ocean_eq} is
  equivalent to the hydrostatic core plus correction terms
  \begin{align}
    \Dt\bv_H + \mathbf{N}_H^{\mathrm{hyd}}
      + \nabla_H B_H
      &= \mathcal{D}_H(\bv_H) + \mathbf{C}_H,
      \label{eq:nh_horizontal}
  \end{align}
  with $\mathbf{N}_H^{\mathrm{hyd}}
  = \bom_a^z(\bk\times\bv_H) + w\,\partial_z\bv_H$ and $B_H$
  of~\eqref{eq:BH_def}, together with
  the prognostic vertical equation
  \begin{align}
    \Dt w &= \mathcal{D}_V(w) + \mathbf{C}_V,
      \label{eq:nh_vertical}
  \end{align}
  driven by the horizontal and vertical non-hydrostatic corrections
  \begin{align}
    \mathbf{C}_H
      &= -\,\tilde f\,(\hat{\mathbf{y}}\times\bv)_H
        - \nabla_H\!\left(\frac{p_{\mathrm{NH}}}{\rho_0}\right),
      \label{eq:CH}\\
    \mathbf{C}_V
      &= -\,\bigl(\bom_a^H\times\bv\bigr)\cdot\bk
        - \partial_z\bigl(\tfrac12 |\bv|^2\bigr)
        - \partial_z\!\left(\frac{p_{\mathrm{NH}}}{\rho_0}\right).
      \label{eq:CV}
  \end{align}
with $\tilde f$ the horizontal Coriolis component,
$\mathbf{f}=(0,\tilde f, f)$, so that
$\tilde f(\hat{\mathbf{y}}\times\bv)_H = \tilde f\,w\,\hat{\mathbf{x}}$.

\subsection{Free Surface}\label{sec:surface}\label{sec:free_surface}
In this subsection we include the movement of the ocean's free surface.

Let the upper boundary be $z=\eta(x,y,t)$, with the kinematic condition,
\begin{equation}\label{eq:kinematic_bc}
  w = \Dt\eta + \bv_H\cdot\nabla_H\eta
  \qquad\text{at } z=\eta ,
\end{equation}
and the dynamic condition that the total pressure equal the
atmospheric pressure, here taken as zero.  Under the
decomposition~\eqref{eq:pressure_split} this condition is carried by
the three parts separately: the hydrostatic pressure vanishes at
$z=\eta$ by the choice of integration limit
in~\eqref{eq:pressure_hydro}, the surface pressure
$p_{\mathrm{sfc}}=\rho_0 g\eta$ of~\eqref{eq:psfc_def} is  the
weight of the displaced column, and
the non-hydrostatic pressure must vanish.  With the aggregation
$p_{\mathrm{NH}} = p_V+\psi$ and the column solve already imposing
$p_V=0$ at the free surface (Section~\ref{sec:acdc_closure}), the
condition enforces
\begin{equation}\label{eq:psi_bc}
  \psi = 0 \qquad\text{at } z=\eta .
\end{equation}
The total dynamic pressure
$p = p_{\mathrm{hyd}}+p_{\mathrm{sfc}}+p_{\mathrm{NH}}$ does not vanish
at the surface $z=\eta$, but equals $\rho_0 g\eta$ there.  
The surface pressure work is the exchange
with the barotropic surface reservoir, and
Proposition~\ref{prop:free_surface}(iii) shows it is absorbed exactly.
This is the boundary condition the residual pseudo-pressure must carry
for the aggregated pressure to satisfy the boundary conditions.  

\paragraph*{The unsplit form under a free surface.}\label{app:unsplit}
Option~S3-a makes the fast loop two-dimensional by splitting $\psi$;
Option~S3-b sub-steps the full three-dimensional pseudo-pressure
directly, which is well defined and serves as the verification
reference, the split scheme reproduces it to round-off.  Coupled to
the free surface of a split-explicit host, whose depth-averaged fast
system is that of Section~\ref{sec:algorithm}, the fast system here
reads
\begin{align}
  \Dt\eta + \nabla_H\cdot(H\bar{\bv}_H) &= 0,
    \label{eq:unsplit_eta}\\
  \frac{1}{\alpha}\Dt\psi + \nabla_H\cdot\bv_H
    + \frac{1}{\alpha}\nabla_H\cdot(\psi\bv_H)
    &= -\Bigl[\partial_z w
       + \tfrac{1}{\alpha}\partial_z(\psi w)\Bigr]_{\mathrm{frozen}},
    \label{eq:unsplit_psi}\\
  \Dt\bv_H + \nabla_H\!\left(g\eta + \frac{\psi}{\rho_0}\right) &= 0,
    \label{eq:unsplit_mom}
\end{align}
with $\psi$ carried on the full three-dimensional mesh.  Here
\eqref{eq:unsplit_psi} is the flux-form AC relation~\eqref{eq:ac_psi}
with $\Div\bv = \nabla_H\cdot\bv_H + \partial_z w$ and
$\Div(\psi\bv) = \nabla_H\cdot(\psi\bv_H) + \partial_z(\psi w)$; with
$w$ frozen the two vertical contributions form a known source, constant
across the sub-step loop, so the only divergence components evolving
under the AC relaxation are the horizontal ones.

The cost is what rules it out of production: sub-stepping $\psi$ on the
full three-dimensional mesh promotes each barotropic update from two
dimensions to three, adding of order $n_z$ times the hydrostatic
barotropic cost.  This is the same three-dimensionality that
Remark~\ref{rem:D_tension} prices from the other side, through the
sub-step count.

\begin{proposition}[Identities on a moving material domain]
  \label{prop:free_surface}
  Let $\Omega(t)$ have upper boundary $z=\eta$
  satisfying~\eqref{eq:kinematic_bc}, all other boundaries fixed with
  $\bv\cdot\mathbf{n}=0$, and let ${\rhoa}$ obey the flux
  form~\eqref{eq:ac_flux}.  Then, with $r={\rhoa}/\rho_0$:
  \begin{enumerate}[nosep]
  \item[(i)] the transport identity~\eqref{eq:transport_identity} holds
    verbatim: $\frac{d}{dt}\int_{\Omega(t)} rF\dd V
      = \int_{\Omega(t)} r(\Dt F + \bv\cdot\nabla F)\dd V$;
  \item[(ii)] the cubic cancellation of
    Proposition~\ref{prop:dw_removes}(iii) holds verbatim: the surface
    term produced by the moving domain cancels the surface term
    produced by the integration by parts, identically and without
    approximation;
  \item[(iii)] the pressure pairing acquires the single surface
    contribution
    $-\int_{z=\eta} r\,\Pi\,\bv\cdot\mathbf{n}\dd \sigma$.  Its
    non-hydrostatic part vanishes under~\eqref{eq:psi_bc} together with
    the column condition $p_V|_{z=\eta}=0$, and its hydrostatic part
    vanishes by~\eqref{eq:pressure_hydro}; what remains is the surface
    pressure work, which is exactly the time derivative of the
    barotropic reservoir,
    \begin{equation}\label{eq:surface_reservoir}
      -\int_{z=\eta} r\,\Pi\,\bv\cdot\mathbf{n}\dd\sigma
        = -\frac{d}{dt}\,\mathrm{PE}_{\mathrm{sfc}},
      \qquad
      \mathrm{PE}_{\mathrm{sfc}} = \frac{g}{2}\int_A\eta^2\dd A ,
    \end{equation}
    so that it is absorbed identically by the definition
    of $E_{\mathrm{KL}}$ in~\eqref{eq:EKL_def};
  \item[(iv)] consequently~\eqref{eq:KL_energy_inviscid},
    \eqref{eq:KL_energy_exact}, \eqref{eq:ke_pe_conversion},
    \eqref{eq:EKL_budget} and
    Proposition~\ref{prop:tracer_consistency} hold unchanged under a
    free surface, with $\mathrm{PE}_{\mathrm{sfc}}$ included in
    $E_{\mathrm{KL}}$ and vanishing when $\eta\equiv0$.
  \end{enumerate}
\end{proposition}

\begin{proof}
  See Appendix~\ref{app:proofs}.
\end{proof}

\paragraph*{The $z^*$ coordinate.}
The natural realisation for a global non-hydrostatic model is
$z^*$-free-surface-following $z$-levels, in which layers move
vertically with $\eta$.  Isopycnal
coordinates are unsuitable for non-hydrostatic dynamics, and
terrain-following coordinates are confined to regional configurations;
hybrid coordinates are beyond the scope of this paper.  Three statements
cover the $z^*$ extension:
\begin{enumerate}
\item 
Proposition~\ref{prop:tracer_consistency} covers the case of
time-dependent cell volumes $|c|(t)$ inside the time derivative and
fluxes $F_f$ taken relative to the moving faces, and its proof uses
only the incidence pattern of faces and cells.  Both the exact
conservation of $\int{\rhoa}C\dd V$ and the
uniform-in-time bound on the representation error therefore hold
verbatim on $z^*$ meshes, as does the variance identity of
Proposition~\ref{prop:variance}.

\item  On $z^*$ the column axes remain vertical, so
$L_z$ remains tridiagonal and the direct solve of
Section~\ref{sec:acdc_closure} is unaffected: the property that makes
AC/DC effective is a property of the prismatic column, not of the levels
being flat.

\item Layer interfaces do tilt on $z^*$, by
$\OO(|\nabla_H\eta|)$; this tilt is smaller by the factor $\eta/H$
than the bathymetric slope the columns already have, and contributes
nothing at leading order to $\|S^*\|/\|S\|$.
\end{enumerate}
\subsection{Errors in Context of other Errors}\label{sec:errors_context}

Having quantified the AC/DC errors above, we place the
volume-conservation (divergence) error within the hierarchy of
volume-budget errors ocean models already have.  
Two distinct quantities are involved: a model may fail to
satisfy its own incompressibility constraint, so that
$\Div\bv \neq 0$; or it may
enforce $\Div\bv = 0$ exactly but on the wrong velocity
field, so that the model velocity $\bv_{\mathrm{model}}$ differs
from the true velocity $\bv_{\mathrm{true}}$.  The first is a
constraint-violation error; the second is a velocity-field error.
We measure both, relative to the characteristic strain rate $U/L$.

\begin{enumerate}
\item \emph{Hydrostatic approximation (velocity-field error).}  The
ocean primitive equations enforce $\Div\bv = 0$ exactly. The
error is not a residual divergence but a wrong velocity field: the
PE velocity differs from the true non-hydrostatic velocity by
$\OO(\delta^2 U)$, and the divergence of that velocity error,
$\Div(\bv_{\mathrm{PE}} - \bv_{\mathrm{NH}})$, is $\OO(\delta^2)$
relative to $U/L$.  

\item \emph{Boussinesq approximation (constraint-violation error).}  The
Boussinesq model enforces $\Div\bv = 0$, whereas the true ocean
satisfies $\Div\bv = -\frac{1}{\rho}\frac{D\rho}{Dt}$.  The Boussinesq velocity field therefore has a residual divergence.

\item \emph{Precipitation minus evaporation (real volume source).}  The
surface freshwater flux $P-E$ is a real volume source in the true
ocean, not a model error; its effective divergence
$(P\!-\!E)/(\rho_w H)$ is $\sim 6\times 10^{-7}$ relative to
$U/L$.  It is included as a physical reference scale.

\item \emph{AC pseudo-compressibility (constraint-violation error).}
AC/DC relaxes incompressibility, so the velocity field retains a residual divergence.  Without the column solve
(pure 3D~AC) this is
$|\Div\bv|/(U/L) \sim \delta^2\,\Fr_{\mathrm{baro}}^2 \sim 10^{-6}$; the
AC/DC column solve reduces it by the factor $(k_x/k_z)^2$, so at grid scale
it reaches $\sim 10^{-10}$ (Table~\ref{tab:error_comparison}).  
\end{enumerate}
The hierarchy is
\begin{equation}
  \underbrace{\delta^2}_{\substack{\text{hydrostatic}\\
    \text{(velocity-field)}}}
  \;\gg\;
  \underbrace{gH_{\mathrm{feat}}/c_s^2}_{\substack{\text{Boussinesq}\\
    \text{(constraint)}}}
  \;\gg\;
  \underbrace{\delta^2\Fr_{\mathrm{baro}}^2}_{\substack{\text{AC (3D)}\\
    \text{(constraint)}}}
  \;\sim\;
  \underbrace{(P\!-\!E)/(\rho_w H U/L)}_{\substack{\text{freshwater}\\
    \text{(real source)}}}.
\end{equation}
The comparison to draw is with the Boussinesq approximation, the
error of the same kind as AC's: the AC constraint violation
is two orders of magnitude below the Boussinesq one that every
Boussinesq ocean model already accepts, and is comparable to the
real P-E forcing.  Set against the velocity-field error of the
hydrostatic approximation -the error AC/DC removes- it is more
than four orders of magnitude smaller.

Table~\ref{tab:error_comparison} collects the errors of the two
formulations, pure AC and AC/DC, alongside their CFL and cost
properties.

\begin{table}[H]
  \centering
  \caption{\it \small Comparison: AC (3D) vs.\ AC/DC.  All errors
    for $\alpha = \rho_0 gH_{\mathrm{full}}$ and the
    submesoscale worst case $\delta = 0.2$,
    $\Fr_{\mathrm{baro}} = 5\times10^{-3}$,
    $\Delta x/\Delta z = 100$.  Divergence errors are relative to
    $U/L$.  The prefactor $(k_x/k_z)^2$
    of~\eqref{eq:hybrid_div_error} admits two readings:
    evaluated at the grid ratio it measures grid-scale noise, evaluated
    at the aspect ratio of a flow structure it measures the error
    on the resolved flow (Remark~\ref{rem:aspect_ratio}).  
    A last block states structural properties.
    The total energy $E_{\mathrm{KL}}$ of~\eqref{eq:EKL_def} obeys the
    balance law~\eqref{eq:EKL_budget} with no numerical remainder: its
    right-hand side is the work of the hydrostatic, surface and for
    AC/DC column pressures against the artificial compressibility,
    together with the dissipation, each an identified and computable
    exchange.  The closure is exact for the curl-divergence viscous
    operator~\eqref{eq:visc_hodge}; the plain Laplacian leaves the
    $\OO(\Fr_{\mathrm{baro}}^2)$ remainder.
    }
  \label{tab:error_comparison}
  \renewcommand{\arraystretch}{0.9}
  \begin{tabular}{lcc}
    \toprule
    & \textbf{AC (3D)} & \textbf{AC/DC} \\
    \midrule
Divergence, grid scale
      & $\delta^2 \Fr_{\mathrm{baro}}^2 \sim 10^{-6}$
      & $(\Delta z/\Delta x)^2 \delta^2 \Fr_{\mathrm{baro}}^2
        \sim 10^{-10}$ \\
    \quad resolved flow feature, $\delta = H/L$
      & $\delta^2 \Fr_{\mathrm{baro}}^2$
      & $\delta^4 \Fr_{\mathrm{baro}}^2$ \\     
    Dispersion
      & $k_z^2 N^2/(\tilde\alpha K^4)$
      & $(k_x^2/k_z^2)\cdot k_z^2 N^2/(\tilde\alpha K^4)$;
        $+\;\OO(\Delta t^2)$ splitting \\
    \midrule
    Total energy $E_{\mathrm{KL}}$
      & \multicolumn{2}{c}{balance~\eqref{eq:EKL_budget} closes;
        every term diagnosable} \\
    Terms on the right-hand side
      & $p_{\mathrm{hyd}}, p_{\mathrm{sfc}}$ work; dissip.
      & $+\ p_V$ work \\
    CFL
      & $\Delta z/c_{\mathrm{ac}}$
      & $\Delta x/c_{\mathrm{ac}}$ \\
    3D cost indep.\ of $\alpha$?
      & No & Yes \\
    2D fast loop indep.\ of $\alpha$?
      & --- & semi-impl.\ yes; split-expl.\ $\propto\!\sqrt{\alpha}$ \\
    \bottomrule
  \end{tabular}
\end{table}

\subsection{Time Stepping Structure}\label{sec:algorithm}

Ocean models deal with two time scales: the fast-moving free surface (barotropic mode) and the slow-moving baroclinic mode. The barotropic mode is handled either by sub-stepping or by semi-implicit time stepping.
The two time-scale structure suggests a corresponding decomposition for the pseudo-pressure.

We recall the splitting~\eqref{eq:psi_split_concept} of the
pseudo-pressure into its depth-averaged component $\bar\psi$ and the
column-mean-free remainder $\psi'$ (Section~\ref{sec:variants2}).
The depth-averaged component satisfies an equation with the structure
of the barotropic mode: it is two-dimensional, it couples to the
surface height through the free surface, and it therefore inherits
whatever barotropic machinery the host model already possesses: no
new solver, no new time-step constraint, and no global communication
beyond the one the hydrostatic model already performs.  Its remainder $\psi'$ holds the vertical structure that~\eqref{eq:psi_converged}
exhibits, and with it the stiffness the column solve does not remove;
it is updated once per baroclinic step by the column-implicit
treatment of step~S3-a1.  Boundary conditions for the split are those stated in Section~\ref{sec:variants2}.

\paragraph*{One ocean-model time step, two AC/DC time objects.}\label{par:two_objects}
The hydrostatic ocean model into which AC/DC is inserted -referred to
as \emph{the model} below- supplies its own time stepping, and
AC/DC does not modify it.  That stepping consists of the baroclinic (outer) step
$\Delta t$ and, where the model has a free surface, a barotropic
treatment of its own: either explicit sub-stepping at $\Delta t^{*}$ or
a single implicit solve per $\Delta t$. 

AC/DC adds two fields that must be advanced in time, on two different time scales:
\begin{enumerate}[nosep]
\item[(O1)] $\bar\psi$, the depth-averaged pseudo-pressure, sub-stepped
  on the acoustic sub-step
  $\Delta t^{\mathrm{ac}} := \Delta t/n_{\mathrm{sub}}$, \
  $n_{\mathrm{sub}} := \lceil \Delta t\,c_{\mathrm{ac}}/\Delta x_{\min}\rceil$;
\item[(O2)] $\psi'$, the column-mean-free remainder, updated once per
  outer step $\Delta t$, implicitly in the vertical.
\end{enumerate}
Both updates run, whether or not the model has a
barotropic mode: in a closed rigid-lid domain the model has the single
step $\Delta t$ and no barotropic loop, and AC/DC still sub-cycles~(O1)
and still updates~(O2) once per step.  Where the model does have a barotropic loop, the sub-stepping of~(O1) can be folded into it; that
is an implementation choice, cheap under the calibration of
Remark~\ref{rem:host_variants}, but it is not part of the method's
definition.

Every statement in this section is attached to one of the resulting
update intervals:
\begin{center}
\small
\begin{tabular}{llll}
\toprule
Advanced field & Interval & Supplied by & Integrator \\
\midrule
baroclinic state $(\bv,T,S,w)$ & $\Delta t$ & model & model's \\
free surface $\eta$, barotropic mode
  & $\Delta t^{*}$ or one solve & model (absent if closed) & model's \\
$\bar\psi$ \ (O1)  & $\Delta t^{\mathrm{ac}}$ & AC/DC
  & neutral sub-step \\
$\psi'$ \ (O2)     & $\Delta t$ & AC/DC
  & column-implicit, once per $\Delta t$ \\
\bottomrule
\end{tabular}
\end{center}
\noindent
All three-dimensional work falls on $\Delta t$; the sub-stepped interval
$\Delta t^{\mathrm{ac}}$ carries two-dimensional fields only.

\paragraph*{No new CFL restriction.}\label{sec:cfl_advantage}\label{sec:cfl}

The main motivation for AC/DC is lightening the CFL constraint by exploiting 
the oceans thin-fluid property.
In pure AC, the acoustic CFL is
$\Delta t \lesssim \ell^*_{\min} / c_{\mathrm{ac}}$, where
$\ell^*_{\min}$ is the smallest edge length, typically
$\Delta z$ on prismatic meshes.  In AC/DC 
the acoustic CFL
is set by the horizontal grid spacing:
\begin{equation}
  \Delta t \lesssim \frac{\Delta x_{\min}}{c_{\mathrm{ac}}}.
\end{equation}
Since typically $\Delta x / \Delta z \sim 10$--$1000$ in ocean models, 
AC/DC allows larger time steps.  The consequences are stated interval
by interval.

\emph{Baroclinic step $\Delta t$.}  Unchanged from the hydrostatic
model: it is limited by the advective CFL $\Delta t < \Delta x/U$.  
No three-dimensional computation in AC/DC scales
with $\alpha$.  Increasing $\alpha$ does not force a smaller baroclinic
step.

\emph{Barotropic treatment.}  Unchanged in both its
variants: a split-explicit model keeps its barotropic sub-step
$\Delta t^{*} = \Delta x/\sqrt{gH_{\mathrm{full}}}$, and a
semi-implicit model keeps an implicit free-surface solve that is
unconditionally stable and imposes no barotropic CFL at all.  AC/DC
adds fields to these algorithms; it does not change what limits them.

\emph{AC/DC acoustic sub-step $\Delta t^{\mathrm{ac}}$.}  This is the
one interval AC/DC introduces, and the one place where $\alpha$ still
enters: $n_{\mathrm{sub}}$
grows as $\sqrt{\alpha}$, so raising $\alpha$
raises two-dimensional sub-steps only.

\paragraph*{The time step template.}
Each step below is labelled by the interval it runs on and by who
supplies it: \textsf{[hydrostatic]} for machinery the hydrostatic model
already has, \textsf{[AC/DC]} for machinery the method adds.  The
template is stated for a model that has a free surface; the closed,
rigid-lid case is the same template with Step~2a absent, and is
recorded at the end.

\begin{enumerate}[nosep]
\item \emph{Step~1: Baroclinic tendencies.  Interval $\Delta t$;
\textsf{[hydrostatic]} core with an \textsf{[AC/DC]} correction.}  Compute
the full three-dimensional vorticity $\bom = \nabla\times\bv$, the full
3D Lamb contraction $\bom_a\times\bv$, the hydrostatic Bernoulli
gradient $\nabla B_{\mathrm{hyd}}$, its surface-pressure part being
advanced by the barotropic loop of Step~2a, and the viscous term.

\item \emph{Step~2a: Barotropic update.  Barotropic interval
($\Delta t^{*}$ sub-steps, or one implicit solve); \textsf{[hydrostatic]}.}
The free surface and the depth-averaged velocity are advanced by
whatever algorithm the model already uses, on the two-dimensional
horizontal mesh,
\begin{equation}
  \Dt\eta + \nabla_H\cdot(H\bar{\bv}_H) = 0 ,
    \label{eq:substep_eta}
\end{equation}
with the barotropic momentum content appearing as the surface-height
gradient $g\nabla_H\eta$ in~\eqref{eq:substep_mom} below.  Nothing in
this step is supplied by AC/DC.

\item \emph{Step~2b: Pseudo-pressure loop, object~(O1).  Interval
$\Delta t^{\mathrm{ac}}$; \textsf{[AC/DC]}.}  Sub-step the
depth-averaged pseudo-pressure against the horizontal velocity.  By the
split of Section~\ref{sec:acdc_closure} this system is
two-dimensional:
\begin{align}
  \frac{1}{\alpha}\Dt\bar\psi
    + \nabla_H\cdot\bar\bv_H
    + \frac{1}{H}\bigl[w\bigr]_{\mathrm{bot}}^{\mathrm{top}}
    + \frac{1}{\alpha}\overline{\Div(\psi\bv)} &= 0,
    \label{eq:ac_baro}\\
  \Dt\bv_H + \nabla_H\Bigl(g\eta + \frac{B\bar\psi}{\rho_0}\Bigr)
    &= -\,\nabla_H\Bigl(\frac{p_V + \psi'}{\rho_0}\Bigr)_{\mathrm{frozen}} .
    \label{eq:substep_mom}
\end{align}
Here $B$ denotes the broadcast operator that assigns the column value
$\bar\psi$ to every cell of that column, so that $B\bar\psi$ is
constant in the vertical.  Equation~\eqref{eq:ac_baro} is the depth average of the
flux-form AC relation~\eqref{eq:ac_psi}.  The vertical velocity $w$ is
not sub-stepped but held at its baroclinic value, so the vertical
divergence contributions form a constant source; 
the baroclinic pseudo-pressure $\psi'$ and the column pressure
$p_V$ enter only as the frozen offset on the right-hand side
of~\eqref{eq:substep_mom}, applied \emph{distributed through the loop}
(see Remark~\ref{rem:psi_prime_stability}).  


\item \emph{Step~3: Vertical implicit pass, column solve, and object~(O2).
Interval $\Delta t$; \textsf{[hydrostatic]} pass carrying two \textsf{[AC/DC]}
solves.}
Implicit vertical viscosity/diffusion  as in standard ocean models; in the same pass the
non-hydrostatic vertical pressure $p_V$ is obtained by the direct
column solve $L_z(p_V/\rho_0)=S$, and the baroclinic pseudo-pressure $\psi'$ is updated
once per step by the column-implicit treatment of
step~S3-a1 (Section~\ref{sec:variants2}).  The two added solves share the
tridiagonal structure of the viscosity solves.

\item \emph{Step~4: Tracer transport.  Interval $\Delta t$;
\textsf{[hydrostatic]}.}  Transport $T$, $S$ with velocity
$\bv = (\bv_H, w)$ and pseudo density $\rho_{AC}$.
\end{enumerate}

\noindent Figure~\ref{fig:flowchart} displays the resulting step for a
semi-implicit host.

\begin{remark}[{Distributed updates}]\label{rem:psi_prime_stability}
The $\psi'$ correction has to be
distributed over the fast loop, as part of the offset
$\nabla(p_V+\psi')$, not as a velocity update after it.  The two give
the same per-step map for the velocity but not for the volume budget, 
an end-of-step correction removes divergence while carrying no volume
flux, so the time-integrated divergence of the actual velocity path,
the quantity a consistency-preserving tracer transport
(Proposition~\ref{prop:tracer_consistency}) integrates, drifts
linearly, independent of $\alpha$. With the distributed
application both budgets close.
\end{remark}

\begin{remark}[Merging object~(O1) into the model's barotropic loop]
\label{sec:variant_split}\label{sec:variant_semiimplicit}\label{rem:host_variants}

(i) In a \emph{split-explicit} model (MOM6, ROMS) the
barotropic mode is advanced by explicit sub-stepping; with
$\alpha = \rho_0 gH_{\mathrm{full}}$ the artificial acoustic speed
equals the barotropic gravity-wave speed
(Section~\ref{sec:cfl_advantage}), so $\bar\psi$ uses the same loop as
$\eta$ and $\bar\bv_H$, one additional two-dimensional field per
sub-step, no new loop and no new CFL constraint.  

(ii) In a \emph{semi-implicit} model (ICON-O, MITgcm) the barotropic mode is
advanced by a single implicit free-surface solve, and~\eqref{eq:ac_baro}
has the same two-dimensional elliptic structure on the same horizontal
mesh; since $\eta$ is the hydrostatic surface pressure and $\bar\psi$
the depth-averaged pseudo-non-hydrostatic pressure, their coupling is
weak (formally $\OO(\delta^2)$ in the aspect ratio) and the
$2\times2$ system for $(\eta,\bar\psi)$ is block-diagonal.  
Either $\bar\psi$ is then solved as a separate two-dimensional
elliptic problem, calling the free-surface solver a second time, or the pair is advanced
as a single $2\times2$ block with the weak off-diagonal coupling
carried implicitly.
\end{remark}

\begin{figure}[tbp]
  \centering
  \begin{tikzpicture}[
    font=\footnotesize,
    >={Stealth[length=2mm]},
    box/.style={draw,rounded corners=2pt,align=center,
      inner sep=3pt,minimum height=7mm,text width=54mm},
    core/.style={box,fill=black!4},
    corr/.style={box,fill=black!12,thick},
    io/.style={draw,align=center,inner sep=3pt,minimum height=6mm,
      text width=32mm,fill=black!3},
    lab/.style={font=\scriptsize\itshape},
    arr/.style={->,thick},
  ]
  \node[io] (start) {State at $t_n$:\ $(\bv,T,S,\eta)$};
  \node[core,below=5mm of start] (s1)
    {\textbf{1. Baroclinic tendencies} (3D, $\Delta t$)\\
     hydrostatic core $\mathbf{N}_H^{\mathrm{hyd}}+\nabla_H B_H$,
     tracers, EOS};
  \node[corr,below=4mm of s1] (s1c)
    {\textbf{$+$ non-hydrostatic correction}\\
     $\mathbf{C}_H,\mathbf{C}_V$ from $\bom_a^H\times\bv$;\
     prognostic-$w$};
  \node[core,below=5mm of s1c] (s2)
    {\textbf{2. Implicit free-surface solve} (2D)\\
     surface height $\eta$-no sub-step loop};
  \node[corr,below=4mm of s2] (s2c)
    {\textbf{$+$ depth-averaged $\bar\psi$}\\
     $2\times2$ block $(\eta,\bar\psi)$-one 2D elliptic solve};
  \node[core,below=5mm of s2c] (s3)
    {\textbf{3. Column pass} (per column)\\
     vertical implicit diffusion,
     $(n_{\mathrm{tracer}}{+}2)$ tridiagonal solves};
  \node[corr,below=4mm of s3] (s3c)
    {\textbf{$+$ direct column solve, $\psi'$ update}\\
     $L_z(p_V/\rho_0)=S$-one extra tridiagonal solve};
  \node[core,below=5mm of s3c] (s4)
    {\textbf{4. Tracer transport} (3D)\\
     $T,S$ with $\bv=(\bv_H,w)$};
  \node[io,below=5mm of s4] (end) {State at $t_{n+1}$};
  \draw[arr] (start)--(s1); \draw[arr] (s1)--(s1c);
  \draw[arr] (s1c)--(s2);   \draw[arr] (s2)--(s2c);
  \draw[arr] (s2c)--(s3);   \draw[arr] (s3)--(s3c);
  \draw[arr] (s3c)--(s4);   \draw[arr] (s4)--(end);
  \node[lab,left=8mm of s1.west,text width=22mm,anchor=east]
    (lc) {hydrostatic core, no new machinery};
  \draw[densely dotted] (lc.east)--(s1.west);
  \node[lab,right=10mm of s1c.east,text width=26mm,anchor=west]
    (la) {shaded: added by AC/DC};
  \draw[densely dotted] (s1c.east)--(la.west);
  \draw[arr,rounded corners=3pt]
    (end.east) -- ++(42mm,0) |-
    node[lab,pos=0.25,right,text width=20mm]{next step}
    (start.east);
  \end{tikzpicture}
  \caption{The AC/DC algorithm for a semi-implicit model (ICON-O),
    one baroclinic step in the order of the core template
    (Section~\ref{sec:algorithm}): baroclinic tendencies,
    barotropic/free-surface update, column pass, tracers.  Light
    boxes are the hydrostatic core, unchanged in its machinery; shaded
    boxes are the additions made by AC/DC.  The depth-averaged pseudo-pressure
    $\bar\psi$ is solved together with the surface height $\eta$ as a
    $2\times2$ block in the implicit free-surface solve; the
    $\eta$--$\bar\psi$ coupling is $\OO(\delta^2)$ in the aspect
    ratio, so the block is diagonal to leading order.  This variant
    has no fast sub-step loop and hence no Strang splitting
    (cf.\ Section~\ref{sec:variant_split}).}
  \label{fig:flowchart}
\end{figure}
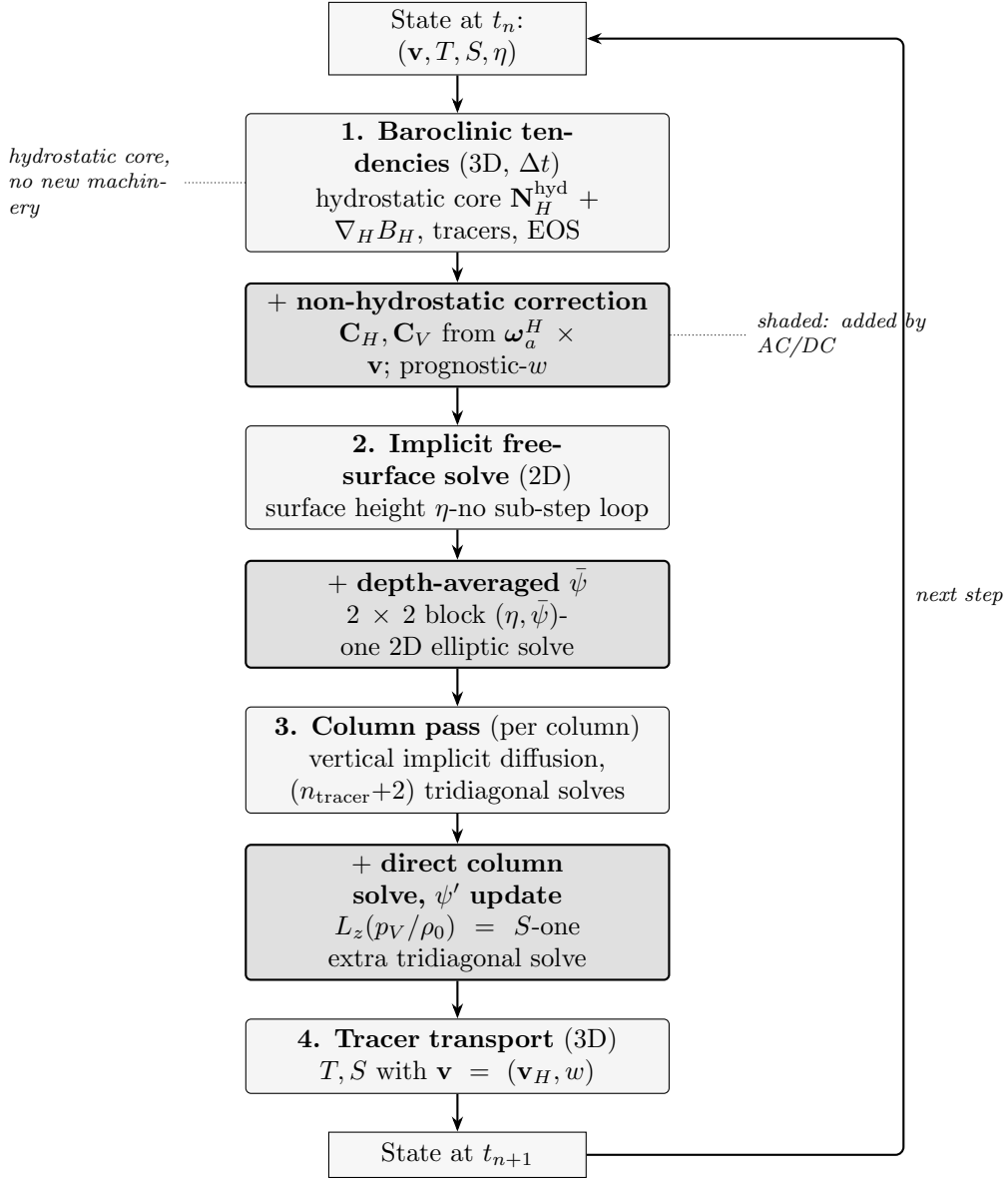

\section{Conclusions}\label{sec:conclusions}
The hydrostatic approximation trades computational speed against computability of small scales.
This paper is an argument that the price need not be paid.

The first part of the argument is that the three-dimensional momentum equation
separates into the hydrostatic primitive equations and a correction
without any approximation being made
(Section~\ref{sec:splitting}); the hydrostatic model that ocean centres
already run is a sub-system of the non-hydrostatic one, and what has to be supplied is the correction. 

Second comes a controlled approximation:
Artificial compressibility replaces an instantaneous global constraint
by relaxation at a finite speed, and every error it introduces involves the single parameter that fixes that speed.  The column solve removes
the vertical part of the
constraint and leaves only the horizontal residual to relax.  
The dispersion of internal gravity waves is then
recovered with an error smaller than pure artificial compressibility
gives by the square of the grid aspect ratio
(Section~\ref{sec:hybrid_dispersion}); and the compressibility error
equilibrates 
(Appendix~\ref{app:boundedness}).  The physical payoff promised in
the introduction, mixing as a computed output not a
prescribed input, survives the approximation: the tracer-variance
budget closes exactly despite the residual divergence, the numerical
part of the variance sink is an explicit face sum, and the Osborn--Cox diffusivity is therefore
computable from the discrete fields
(Proposition~\ref{prop:variance},
Corollary~\ref{cor:osborn_cox}).

Third is the feasibility argument of
Section~\ref{sec:pressure_calc}, whose conclusion is that AC/DC's
overhead is bounded, known before the model is run, and flat under
refinement.  It is a scale analysis in a single unit, and not a timing comparison,
because on the projection side there is nothing to time.

The numerical experiments of Section~\ref{sec:experiments} are there to
show that AC/DC reproduces the processes the hydrostatic approximation
removes, and that where the two should agree they do.  They are not
performance measurements and they are not global.  A
convection-resolving simulation of the world ocean belongs to a
computational campaign and is out of scope here.

What the paper claims is that the computational obstruction keeping
ocean models hydrostatic is removable, at a cost that is a small fixed multiple of the model one
already has, on the machines that are being built.  What follows from
that is a matter for simulation.

\appendix
\section{Appendix}\label{sect:appendix}


\subsection{Proofs}\label{app:proofs}
The proofs of the propositions of Section~\ref{sec:properties}.

\begin{proof}[Proof of Proposition~\ref{prop:tracer_consistency}]
  (i)~Sum either equation of~\eqref{eq:disc_cont_tracer} over $c$.  Each
  interior face is counted once from each adjoining cell with opposite
  orientation, so its two contributions cancel; boundary faces carry
  $F_f=0$.  Hence $\dot M=\dot\Sigma=0$.  The argument uses only the
  incidence pattern, neither the metric nor the reconstruction, so cell
  volumes moving with the coordinate do not affect it.

  (ii)~Setting $C\equiv\mathrm{const}$ under~(H2) reduces the second
  equation of~\eqref{eq:disc_cont_tracer} to the first, so constants are
  stationary; with the fluxes $F_fC_f$ built on the same $F_f$ that
  advance $r$, the reduction is exact for any $\alpha$.

  (iv)~Put $d_c:=r_c-1=\psi_c/\alpha$, so $\sum_c|c|d_c=m$ is constant
  by~(i).  By~(H4), $d_c>-1$, so $d_c^-<1$ pointwise and
  $\sum_c|c|d_c^-\le|\Omega|$; with
  $\sum_c|c|d_c^+-\sum_c|c|d_c^-=m$ this gives
  $\sum_c|c|\,|d_c| = 2\sum_c|c|d_c^- + m \le 2|\Omega|+|m|$, i.e.
  $\|\psi\|_{L^1_h}\le\alpha(2|\Omega|+|m|)$.  From $r=1+\psi/\alpha$,
  $\Sigma-\bar C_{\mathrm{vol}}M = \alpha^{-1}\langle\psi,C'\rangle_h$,
  which is the equality in~\eqref{eq:tracer_repr_bound}; H\"older in
  $L^1$--$L^\infty$ and the preceding bound give the inequality.  A
  monotone transport satisfies a discrete maximum principle, so
  $\|C'(t)\|_\infty\le\|C'(0)\|_\infty$, and $r(0)\equiv1$ gives $m=0$,
  $M=|\Omega|$.  If instead $\|\psi\|_h$ is bounded, Cauchy--Schwarz on
  $\langle\psi,C'\rangle_h$ replaces the $L^1$--$L^\infty$ pairing and
  yields the $\OO(\alpha^{-1/2})$ form.  Claim~(iii) is the same
  statement rewritten: $\int\rho_0C = \int{\rhoa}C
  + \int(\rho_0-{\rhoa})C$, the first term conserved by~(i)
  and the second equal to $-\rho_0\alpha^{-1}\langle\psi,C\rangle_h$,
  bounded above.  The bound is thus derived from positivity and exact
  conservation, not assumed as an a~priori property of the
  relaxation.
\end{proof}

\begin{proof}[Proof of Proposition~\ref{prop:variance}]
  Hypothesis~(H1) gives the two semi-discrete balances
  $\tfrac{d}{dt}\bigl(|c|({\rhoa}C)_c\bigr) = -\rho_0\,(DF)_c$ and,
  on setting $C\equiv1$, $\tfrac{d}{dt}\bigl(|c|\rho_{\mathrm{AC},c}\bigr)
  = -\rho_0(D\Phi)_c$, the latter being~\eqref{eq:disc_cont_tracer};
  cell volumes are carried inside the time derivative so that the
  identity holds on $z^*$ as well as $z$ meshes.
  Differentiating~\eqref{eq:variance_def} and regrouping,
  \begin{equation}\label{eq:variance_split}
    \frac{dV_h}{dt}
      = \sum_c |c|\,C_c\,\frac{d}{dt}({\rhoa}C)_c
        - \tfrac12\sum_c |c|\,C_c^2\,\frac{d}{dt}\rho_{\mathrm{AC},c}
      = -\rho_0\Bigl[\,C^{\mathsf T}DJ
        - \tfrac12 (C^{2})^{\mathsf T}D\Phi\,\Bigr],
  \end{equation}
  where $C^2$ denotes the vector with entries $C_c^2$.  With the stated
  incidence convention, summation by parts on a closed domain gives
  $C^{\mathsf T}Dy = -\sum_f y_f[C]_f$ for any face field $y$, so
  $C^{\mathsf T}DJ = -\sum_f J_f[C]_f$ and, since
  $[C^2]_f = (C_{c^-}+C_{c^+})[C]_f = 2\langle C\rangle_f[C]_f$,
  $(C^2)^{\mathsf T}D\Phi = -2\sum_f\Phi_f\langle C\rangle_f[C]_f$.
  Substituting both into~\eqref{eq:variance_split} and
  inserting~\eqref{eq:tracer_flux},
  \begin{equation}
    \frac{dV_h}{dt}
      = \rho_0\sum_f [C]_f\bigl(J_f - \Phi_f\langle C\rangle_f\bigr)
      = \rho_0\sum_f \Phi_f[C]_f\bigl((RC)_f-\langle C\rangle_f\bigr)
        - \rho_0\sum_f\kappa_f\gamma_f[C]_f^2 ,
  \end{equation}
  which is~\eqref{eq:variance_budget}.  For~(i), $D\Phi$ enters
  \eqref{eq:variance_split} only through the second term, which is
  exactly the amount needed to convert $C^{\mathsf T}DJ$ into the jump
  form above, so it cancels.  Claim~(ii) is immediate.
  For~(iii), the upwind value is
  $(RC)_f = \langle C\rangle_f - \tfrac12\operatorname{sgn}(\Phi_f)[C]_f$,
  whence $\Phi_f[C]_f((RC)_f-\langle C\rangle_f)
  = -\tfrac12|\Phi_f|[C]_f^2$.
\end{proof}

\begin{proof}[Proof of Proposition~\ref{prop:free_surface}]
  For a domain whose boundary moves with normal velocity
  $\mathbf{u}_b\cdot\mathbf{n}$, the transport theorem gives
  $\frac{d}{dt}\int_{\Omega(t)} f\dd V
    = \int_{\Omega(t)}\Dt f\dd V
      + \int_{\partial\Omega} f\,\mathbf{u}_b\cdot\mathbf{n}\dd\sigma$.
  By~\eqref{eq:kinematic_bc} the free surface is material, so
  $\mathbf{u}_b\cdot\mathbf{n} = \bv\cdot\mathbf{n}$ there, and on the
  remaining boundaries both vanish.  Hence throughout
  \begin{equation}\label{eq:reynolds_material}
    \frac{d}{dt}\int_{\Omega(t)} f\dd V
      = \int_{\Omega(t)}\Dt f\dd V
        + \int_{\partial\Omega} f\,\bv\cdot\mathbf{n}\dd\sigma .
  \end{equation}

  (i)~and~(ii) are the same cancellation.  Apply
  \eqref{eq:reynolds_material} to $f = rF$, and then to
  $f = \tfrac12 r|\bv|^2$, using $\Dt r = -\Div(r\bv)$: in each case it
  contributes a surface term $\int_{\partial\Omega} f\,\bv\cdot\mathbf{n}$,
  and integrating $\int F\Div(r\bv)$ -respectively the Bernoulli
  gradient- by parts contributes the same integral with the opposite
  sign.  Neither vanishes on its own; their cancellation, guaranteed by
  materiality, leaves the volume terms of the fixed-domain proof and
  hence the stated identities.

  (iii)~The pressure term is
  $-\int r\bv\cdot\nabla\Pi
    = \int\Pi\Div(r\bv)
      - \int_{\partial\Omega} r\,\Pi\,\bv\cdot\mathbf{n}$, and no
  compensating surface term arises from~\eqref{eq:reynolds_material},
  which produces boundary contributions only for the quantity being
  differentiated.  On the fixed boundaries $\bv\cdot\mathbf{n}=0$.  On
  $z=\eta$ the integrand carries
  $\Pi=(p_{\mathrm{hyd}}+p_{\mathrm{sfc}}+p_V+\psi)/\rho_0$.  Three of
  the four parts vanish there: $p_{\mathrm{hyd}}$
  by~\eqref{eq:pressure_hydro}, $p_V$ by the column solve's Dirichlet
  condition, and $\psi$ by~\eqref{eq:psi_bc}.  Hence
  $\Pi|_{z=\eta} = g\eta$ by~\eqref{eq:psfc_def}, and \eqref{eq:psi_bc}
  also makes the weight $r = 1+\psi/\alpha$ equal to unity there, so the
  weighting drops out of the surface integral.  With the surface
  $z=\eta(x,y,t)$ written as $F = z-\eta = 0$, the outward normal and
  area element combine as
  $\bv\cdot\mathbf{n}\dd\sigma = (w - \bv_H\cdot\nabla_H\eta)\dd A
   = \Dt\eta\dd A$ by the kinematic condition~\eqref{eq:kinematic_bc},
  so
  \begin{equation}\label{eq:surface_work}
    -\int_{z=\eta} r\,\Pi\,\bv\cdot\mathbf{n}\dd\sigma
      = -\int_A g\,\eta\,\Dt\eta\dd A
      = -\frac{d}{dt}\,\frac{g}{2}\int_A\eta^2\dd A ,
  \end{equation}
  which is~\eqref{eq:surface_reservoir}.  All three ingredients are
  needed: \eqref{eq:pressure_hydro} to remove the hydrostatic part,
  \eqref{eq:psi_bc} to remove the pseudo-pressure part and fix the
  weight, and~\eqref{eq:kinematic_bc} to make the surface integral a
  total time derivative.

  (iv)~Given~(i)--(iii) every step of the corresponding fixed-domain
  proof applies verbatim.  For
  Proposition~\ref{prop:tracer_consistency} the discrete statement is
  already written with time-dependent cell volumes and face-relative
  fluxes, which is the discrete form of~\eqref{eq:reynolds_material}.
\end{proof}

\begin{proof}[Proof of Proposition~\ref{prop:hybrid_dispersion}]
  From~\eqref{eq:hyb_v} and~\eqref{eq:hyb_b},
  $\hat v = -if\hat u/\omega$ and $\hat b = -iN^{2}\hat w/\omega$.
  Substituting into~\eqref{eq:hyb_w} gives
  $\hat w = -(k_x/k_z)(f^{2}/\omega^{2})\hat u
   + (k_z/\omega)\tilde\psi$, and~\eqref{eq:hyb_u} becomes
  $\hat u\,[\omega^{4} - f^{2}\omega^{2}(1-k_x^{2}/k_z^{2})
   - (k_x^{2}/k_z^{2})N^{2}f^{2}]
   = k_x\tilde\psi\,\omega\,(\omega^{2}-N^{2})$.
  Substituting both into the divergence
  relation~\eqref{eq:hyb_psi} and clearing denominators
  produces~\eqref{eq:hybrid_dispersion_rot}, the $f^{2}$
  contributions to the constant term cancelling exactly.  For~(ii),
  the slow-root expansion $\omega^{2} = C/B + C^{2}/B^{3}
  + \OO(\tilde\alpha^{-2})$ is applied to both biquadratics -the
  $C^{2}/B^{3}$ term is of the same order as the
  $\OO(1/\tilde\alpha)$ shifts in $B$ and $C$ and must be
  retained- and the simplifications
  $\omega_0^{2}-f^{2} = k_x^{2}(N^{2}-f^{2})/K^{2}$ and
  $\omega_0^{2}-N^{2} = -k_z^{2}(N^{2}-f^{2})/K^{2}$
  give~\eqref{eq:rot_disp_errors}.  For~(iii), setting $f = 0$ reduces
  the $\omega^{2}$-coefficient to $\tilde\alpha K^{2}$ and the
  constant term to $\tilde\alpha N^{2}k_x^{2}$, and with
  $\omega_0^{2} = N^{2}k_x^{2}/K^{2}$ the first
  of~\eqref{eq:rot_disp_errors} becomes $k_x^{2}N^{2}/(\tilde\alpha
  K^{4})$ and the second $k_z^{2}N^{2}/(\tilde\alpha K^{4})$.
  For~(iv), the factorisation at $k_x = 0$ is direct substitution;
  and $(N^{2}-f^{2})^{2}\le N^{4}$ whenever $f^{2}\le 2N^{2}$, while
  $N^{2}k_x^{2} + f^{2}k_z^{2} \ge N^{2}k_x^{2}$, both inequalities
  reducing the error relative to its non-rotating value.
\end{proof}

\subsection{Boundedness of Acoustic Energy}\label{app:boundedness}

The budgets of Section~\ref{sec:acdc_compress} allow no conclusion
on whether the energy stays finite in the presence of the new energy
reservoir that AC/DC adds.  What bounds it has
been present (cf. ~\eqref{eq:visc_hodge}): the curl--divergence
dissipation damps $|\Div\bv|^2$ as well as $|\nabla\times\bv|^2$, and
the divergent part of the velocity is what the pseudo-pressure exchanges
energy with.  The argument is made semi-discretely, both because the
quantity it must control, the bound entering
Proposition~\ref{prop:tracer_consistency}(iv) is discrete, and
because the reduction is then exact on the mesh, resting on
$d\circ d = 0$.

Write $D$, $G$, $C$ for the discrete divergence, gradient and curl,
and $\langle\cdot,\cdot\rangle_h$,
$(\cdot,\cdot)_h$ for the cell and face inner products
of~\eqref{eq:EKL_budget_h} and~\eqref{eq:kdw_h}.  Four classical
properties of a discrete exterior calculus are used and no other:
(E1)~$CG=0$, the discrete curl of a gradient, a statement about
incidence carrying no metric; (E2)~$\langle\phi,Du\rangle_h
= -(G\phi,u)_h$ for face fields vanishing on the boundary, whence also
$DC^{\mathsf T} = -(CG)^{\mathsf T} = 0$; (E3)~$L := -DG = G^{*}G$
symmetric and positive semi-definite, with $\ker L$ the cellwise
constants, eigenvalues
$0 = \lambda_0 < \lambda_1 \le \dots \le \lambda_{\max}$ and an
orthonormal eigenbasis, $L$ being the negative $-(L_H+L_z)$ of the
splitting~\eqref{eq:laplace_split}, the operator the model already assembles; and
(E4)~the dissipation in Hodge form
$\mathcal{D}_h = \nu\,(GD - C^{\mathsf T}\Lambda C)$ with $\Lambda$ a
positive diagonal Hodge factor.

\begin{proposition}[The acoustic reservoir is a damped wave, and stays
  bounded]\label{prop:energy_bound}
  Let $(\bv,\psi)$ solve the semi-discrete AC/DC system on a fixed mesh,
  put $q := \psi/\alpha$ and $\tilde\alpha = \alpha/\rho_0$, and assume
  (E1)--(E4) with $\nu>0$.
  \begin{enumerate}[nosep]
  \item[(i)] Exactly, and with no truncation error introduced by the
    reduction,
    \begin{equation}\label{eq:disc_wave}
      \ddot q + \nu L\,\dot q + \tilde\alpha\,L\,q
        = -D\mathbf{F} + \dot\varrho + \nu L\varrho ,
    \end{equation}
    with $\mathbf{F}$ collecting every momentum tendency other than the
    pseudo-pressure gradient and the dissipation and
    $\varrho := -\alpha^{-1}D(\psi\bv)$.  In the eigenbasis of~(E3) this
    decouples completely into
    $\ddot q_k + \nu\lambda_k\dot q_k + \tilde\alpha\lambda_k q_k = f_k$.
  \item[(ii)] Every mode with $\lambda_k>0$ decays, unconditionally in
    $\alpha$, in $\nu$ and in the mesh, at the rate
    $-\mathrm{Re}\,\mu = \tfrac12\nu\lambda_k$ while
    $\nu^2\lambda_k < 4\tilde\alpha$ and at $\tilde\alpha/\nu$ beyond
    it.  On $\ker L$ the forcing vanishes identically, so the constant
    mode is stationary, the exact pseudo-mass conservation of
    Proposition~\ref{prop:tracer_consistency}(i), recovered from the
    wave equation.
  \item[(iii)] If the modal forcings are bounded,
    $\mathcal{F}_k := \sup_t|f_k| < \infty$, then $\|\psi\|_h$ is
    bounded uniformly in $t$, and with well-prepared initial data
    (Section~\ref{sec:initialisation})
    $\|\psi\|_h = \OO(\alpha^{1/2})$, hence
    $\|\psi\|_h/\alpha = \OO(\alpha^{-1/2})$.
  \item[(iv)] If in addition the forcing spectrum lies well below
    $\omega_k := \sqrt{\tilde\alpha\lambda_k}$, the response is the
    static one, $|q_k| \le \mathcal{F}_k/(\tilde\alpha\lambda_k)$, and
    the bound sharpens to $\|\psi\|_h = \OO(1)$,
    $\|\psi\|_h/\alpha = \OO(\alpha^{-1})$.
  \end{enumerate}
\end{proposition}

\begin{proof}
  Apply $D$ to the momentum equation and put $\Theta := D\bv$.  By~(E4)
  and the adjoint identity of~(E2),
  $D\mathcal{D}_h = \nu DGD - \nu(DC^{\mathsf T})\Lambda C = -\nu LD$,
  the middle term vanishing identically by~(E1), the only step at which
  a discretisation property is used.  With $\dot q = -\Theta + \varrho$
  from the discrete~\eqref{eq:ac_psi}, eliminating $\Theta$
  gives~\eqref{eq:disc_wave}, which involves no operator but $L$, so
  projection on the eigenbasis decouples it.  For~(ii) the coefficients
  of $\mu^2 + \nu\lambda_k\mu + \tilde\alpha\lambda_k$ are positive when
  $\lambda_k>0$, so both roots lie in the open left half-plane and the rates are the root pair; on $\ker L$,
  $\langle 1, D\mathbf{F}\rangle_h = -(G1,\mathbf{F})_h = 0$ by~(E2).
  For~(iii), each mode is a damped oscillator with
  $2\beta_k = \nu\lambda_k$ and $\omega_k^2 = \tilde\alpha\lambda_k$,
  whose impulse response has $L^1$ norm at most
  $(\beta_k\omega_{d,k})^{-1}$, so Duhamel gives
  $\sup_t|q_k| \le \mathcal{T}_k
   + \mathcal{F}_k/(\beta_k\omega_{d,k})$ with $\mathcal{T}_k$ the
  decaying homogeneous part; squaring and summing over $k\ge1$ in the
  orthonormal basis, with $\|\psi\|_h = \alpha\|q\|_h$ and the zero mode
  contributing nothing by~(ii), each summand carries
  $\tilde\alpha^{-1}$ through $\omega_k$.  For~(iv), the transfer
  function $[(\omega_k^2-\Omega^2)^2 + (2\beta_k\Omega)^2]^{-1/2}$ at
  $\Omega \ll \omega_k$ is $\omega_k^{-2}$.
\end{proof}

\subsection{Cost Accounting}\label{app:cost_detail}

This appendix records the provenance of the counted entries of
Table~\ref{tab:cost_accounting} and the detail of the $\bar\psi$ solve;
the counts themselves and their use are in
Section~\ref{sec:pressure_calc}.

The two contraction entries are counted from source, the hydrostatic
one, $2.96\,F_1$, from the ICON-O production code, the non-hydrostatic
one, $3.86\,F_1$, from the reference incompressible-Euler
implementation normalised to the same loop idiom, so their difference
is one extra pass over the edge triples, $0.90\,F_1$, which is the
Lamb-vector correction $\bom_a^H\times\bv$ at the continuous level.
Spanning the counting convention for the contraction increment,
$0.8$--$1.3\,F_1$, widens the overhead from $1.78$ to
$1.68$--$2.18\,F_1$ against a baseline of about $8$, so the factor
stays between $1.21$ and $1.27$ throughout.

\paragraph*{Cost of the barotropic $\bar\psi$ solve.}
With
$\alpha = \rho_0 gH_{\mathrm{full}}$ the depth-averaged AC operator
$1 + (\alpha/\rho_0)(\tau/\Delta x)^2(-\nabla_H^2)$ and the free-surface
operator $1 + gH(\tau/\Delta x)^2(-\nabla_H^2)$ are identical in
stiffness, so the $\bar\psi$ solve is cheaper than the $\eta$ solve for
reasons of tolerance and not of conditioning.  The surface height holds the full barotropic dynamics and needs
$\varepsilon_\eta\sim10^{-13}$; $\bar\psi$ contributes only an
$\OO(1/\tilde\alpha)$ momentum correction, for which
$\varepsilon_\psi\sim10^{-2}$--$10^{-3}$ suffices.  With a CG count
scaling as $\tfrac12\sqrt{\kappa}\,\ln(2/\varepsilon)$,
\begin{equation}\label{eq:iter_ratio}
  \frac{n_{\mathrm{iter}}^{\bar\psi}}{n_{\mathrm{iter}}^{\eta}}
  = \frac{\ln(2/\varepsilon_\psi)}{\ln(2/\varepsilon_\eta)}
  \;\approx\; 0.2\text{--}0.25 .
\end{equation}
No absolute iteration count is claimed by this formula, and none is
needed: the two solves act on the same two-dimensional mesh with the
same operator stiffness, so every factor the formula fails to capture
is common to numerator and denominator and cancels, leaving the
tolerance dependence.  Since the free-surface solve is itself some
$20\%$ of a semi-implicit model's cost, this matters.  The arithmetic
increment lies between
\eqref{eq:iter_ratio} and a doubling of the per-iteration work,
according to whether residuals are monitored separately or a single
stopping criterion is used; the two also differ in communication.  On a
split-explicit host $\bar\psi$ instead is advanced in the existing sub-step loop
as one extra two-dimensional field, at $0.16\to0.24\,F_1$.

Under refinement the projection's cost grows with $H/\Delta x$ through
$\kappa_{\mathrm{prec}} = 1 + H^2/\Delta x^2$, whereas AC/DC's column
solve is $\OO(n_z)$ per column whatever the horizontal spacing and its
sub-step count scales as $\Delta x^{-1}$ exactly as the hydrostatic
model's does, so the \emph{ratio} is flat even though neither cost is
(Sections~\ref{sec:pressure_calc} and~\ref{sec:telescoping}).

\subsection{The Density Weighting of the Kinetic Energy}\label{sec:df_source}

Which quantity is reported as ``the energy'' in the experiments of
Section~\ref{sec:experiments} is settled by the fact that
$K_{\mathrm{df}}$ does not close and that the density weighting repairs
it.  The mechanism is the one that forces the density-weighted
formulation in the barotropic Navier--Stokes equations
\citep{Korn2026b}: AC/DC is structurally barotropic in its
pressure--density closure, and the velocity it produces is not
solenoidal.

The obstruction is specific to the \emph{vector-invariant} form, which puts the kinetic energy inside the Bernoulli gradient and pairs it
against $\bv$ in the energy budget.  Under incompressibility that
pairing integrates by parts to zero and the term disappears silently,
which is what the standard energy argument for the incompressible
vector-invariant equations rests on; under artificial compressibility
it does not vanish, and what is left is cubic in the velocity.  Write
$\Pi := (p_{\mathrm{hyd}}+p_{\mathrm{sfc}}+p_V+\psi)/\rho_0$, so that
the Bernoulli term of~\eqref{eq:acdc_mom} is $\tfrac12|\bv|^2+\Pi$;
only its kinetic-energy part produces cubic terms, so what follows
holds verbatim for pure AC, where $p_V\equiv0$.  Throughout, (A1)
and~(A3) denote the discrete Cartan identity and the Lamb-antisymmetry
axiom of~\citet{Korn2026b}; in the continuum~(A3) holds pointwise,
$\bv\cdot(\bom_a\times\bv)=0$.

\begin{proposition}[The cubic obstruction and its unique
  repair]\label{prop:dw_removes}\label{prop:df_source}
  Let $(\bv,\psi)$ solve~\eqref{eq:ac_psi} and~\eqref{eq:acdc_mom} on a
  domain closed to mass flux, with the Lamb term satisfying~(A3), and
  put $r := {\rhoa}/\rho_0 = 1+\psi/\alpha$.
  \begin{enumerate}[nosep]
  \item[(i)] The density-free kinetic energy obeys
    $dK_{\mathrm{df}}/dt = \mathcal{S}_{\mathrm{df}}
    + \langle\Pi,\Div\bv\rangle + \langle\mathcal{D}(\bv),\bv\rangle
    + \langle b,w\rangle$ with
    $\mathcal{S}_{\mathrm{df}} := \int_\Omega\tfrac12|\bv|^2\Div\bv\dd V$,
    which is what $-\langle\nabla(\tfrac12|\bv|^2),\bv\rangle$ becomes
    on integration by parts.  It vanishes for all velocity fields only
    under $\Div\bv\equiv0$ and is not sign-definite, so
    $K_{\mathrm{df}}$ is no Lyapunov functional for the inviscid
    system.
  \item[(ii)] It cannot be removed by modifying the Lamb term.
    Under~(A3) that term contributes identically zero
    to that budget, so there is nothing to trade against it,
    and any modification cancelling $\mathcal{S}_{\mathrm{df}}$ must
    violate~(A3) and with it the Cartan identity~(A1) on which
    circulation conservation rests.  This is the energy--circulation
    dichotomy of~\citet{Korn2026b} in the present setting.
  \item[(iii)] The repair lies in the energy functional and is exact.
    With $k_{\mathrm{dw}} := \tfrac12\int_\Omega r|\bv|^2\dd V
    = K_{\mathrm{dw}}/\rho_0$ the cubic sector telescopes and
    \emph{nothing cubic remains}:
    $dk_{\mathrm{dw}}/dt = -\langle r\bv,\nabla\Pi\rangle$ plus the
    viscous and buoyancy pairings.
  \item[(iv)] The weight is forced: among corrections
    $K_{\mathrm{df}} + \tfrac12\int_\Omega g|\bv|^2\dd V$ with $g$ a
    function of the state vanishing for $\psi\equiv0$, the cubic terms
    cancel for arbitrary $\bv$ if and only if
    \begin{equation}\label{eq:g_condition}
      \Dt g + \Div(g\bv) = -\Div\bv ,
    \end{equation}
    and $g = \psi/\alpha$ is the unique such choice.
  \end{enumerate}
\end{proposition}

\begin{proof}
  For~(i), pair~\eqref{eq:acdc_mom} with $\bv$; the Lamb term drops
  by~(A3) and the Bernoulli gradient integrates by parts to
  $\mathcal{S}_{\mathrm{df}}+\langle\Pi,\Div\bv\rangle$, in which
  $\tfrac12|\bv|^2\ge0$ while $\Div\bv$ oscillates on the acoustic
  modes of~\eqref{eq:ac_psi}.  For~(ii), (A3) makes the Lamb
  contribution vanish for \emph{every} field and not merely in the
  mean, so a compensating term needs a symmetric part, which
  destroys~(A1).  For~(iii), the flux form~\eqref{eq:ac_flux} reads
  $\Dt r = -\Div(r\bv)$, so
  \begin{equation}\label{eq:dw_split}
    \frac{dk_{\mathrm{dw}}}{dt}
      = \tfrac12\!\int_\Omega (\Dt r)|\bv|^2\dd V
        + \int_\Omega r\,\bv\cdot\Dt\bv\dd V
      = -\tfrac12\!\int_\Omega \Div(r\bv)|\bv|^2\dd V
        + \int_\Omega r\,\bv\cdot\Dt\bv\dd V ,
  \end{equation}
  and the Bernoulli part of the second integral is
  $-\int r\bv\cdot\nabla(\tfrac12|\bv|^2)
   = \int\tfrac12|\bv|^2\Div(r\bv)$, because
  $\Div(r\bv\,\tfrac12|\bv|^2)$ integrates to zero on a closed domain.
  Added to the first term this cancels outright, leaving
  the weighted budget of~(iii) with $\mathcal{S}_{\mathrm{df}}$ nowhere
  and no remainder in its place.  For~(iv), the correction contributes
  $\tfrac12\int(\Dt g)|\bv|^2$ and, by the same integration by parts,
  $\tfrac12\int|\bv|^2\Div(g\bv)$, so with $\mathcal{S}_{\mathrm{df}}$
  the cubic total is
  $\tfrac12\int|\bv|^2[\Dt g+\Div(g\bv)+\Div\bv]$, vanishing for
  arbitrary $\bv$ exactly when~\eqref{eq:g_condition} holds pointwise;
  for $g=r-1$ its left side is $\Dt r+\Div(r\bv)-\Div\bv = -\Div\bv$
  by~\eqref{eq:ac_flux}.  Uniqueness holds because the difference of
  two solutions obeys $\Dt h+\Div(h\bv)=0$, whose solutions are $h=rG$
  with $G$ materially advected, and the class condition forces $G$ to
  vanish on states with $r\equiv1$, hence identically.
\end{proof}

The failure of $K_{\mathrm{df}}$ is therefore structural, and the repair exact at any compressibility.  Had
${\rhoa}$ been advanced by the relaxation
form~\eqref{eq:ac_cont}, a remainder
$\OO(1/\alpha)$ would survive.  What remains in~\eqref{eq:EKL_budget}
is of one kind only, the work of the non-pseudo pressures against the
artificial compressibility, each with an explicit $\alpha^{-1}$ and
vanishing in the incompressible limit; the same weighting turns a
tracer-variance budget into a telescoping face sum in
Proposition~\ref{prop:variance}.

\subsection{The $1/\alpha$ Law, Measured}\label{app:alpha_law}

\emph{Diagnostics bearing on the approximation itself.}  This is the
only experiment in which the stratification collapses, and it is
therefore the only one that probes the regime in which
Remark~\ref{rem:aspect_ratio} says the splitting is worked hardest.
The goal of these diagnostics is to measure what cannot be predicted.
The run reports:
\begin{enumerate}
\item the Richardson number $R_i = N^2h^2/U^2$ and the Marshall
  parameter $\mathrm{n}=\gamma^2/R_i$ of~\eqref{eq:marshall_n},
  evaluated within the convecting patch, confirming that the run is in
  the non-hydrostatic regime;
\item the depth $D$ of the weakly stratified layer and the width $L$ of
  the descending plumes, giving the predicted ratio
  $(D/L)^2$ of~\eqref{eq:DL_ratio};
\item the measured residual fraction $\|S^*\|/\|S\|$
  from~\eqref{eq:residual}, which is the same ratio obtained in
  situ without any assumption about $D$ or $L$.  Agreement between
  this and item~2 is the substantive test of
  Remark~\ref{rem:aspect_ratio};
\item the relative divergence error $|\Div\bv|/(U/L)$ evaluated
  within the plumes and not domain-averaged, since the
  average is dominated by the quiescent surroundings and understates
  it; together with the same quantity from a run without the column
  solve, which should be the smaller of the two wherever
  $\|S^*\|/\|S\|$ exceeds unity.
\end{enumerate}
A pair of runs with $\alpha$ separated by a factor nine tests the predicted
$1/\alpha$ scaling of item~4.  The tracer-variance budget of
Proposition~\ref{prop:variance} is not among these diagnostics: the
surface cooling patch is a tracer source, so the closed-domain
hypothesis of the proposition fails here, and the budget is instead
exhibited in the source-free configuration of
Experiment~1 (Section~\ref{sec:exp_lock}).  These
measurements are what the convection claim of
Section~\ref{sec:intro} rests on.

\emph{The $1/\alpha$ law.}
Every quantity that the artificial compressibility governs follows the
$1/\alpha$ law, the observed ratios lying between $8.94$ and $9.01$
against an exact nine:

\begin{center}
\begin{tabular}{lccc}
\toprule
Quantity & $\alpha_1 = 3.131\times10^{7}$ & $\alpha_2 = 2.818\times10^{8}$ & ratio \\
\midrule
acoustic energy, $t = 0$
  & $6.885\times10^{-5}$ & $7.649\times10^{-6}$ & $9.000$ \\
$\Fr_{\mathrm{baro}}^{2} = \max|\psi|/(\rho_0\alpha)$
  & $2.928\times10^{-3}$ & $3.250\times10^{-4}$ & $9.010$ \\
$\max|\Div\bv|$, time mean
  & $2.582\times10^{-4}$ & $2.877\times10^{-5}$ & $8.975$ \\
$\max|\Div\bv|$, at $t_{\mathrm{end}}$
  & $1.746\times10^{-4}$ & $1.953\times10^{-5}$ & $8.940$ \\
acoustic energy, $t_{\mathrm{end}}$
  & $1.078$              & $1.203\times10^{-1}$ & $8.957$ \\
tracer conservation error
  & $1.148\times10^{-6}$ & $1.281\times10^{-7}$ & $8.964$ \\
\bottomrule
\end{tabular}
\end{center}

\noindent
The acoustic energy is $\|\psi\|^{2}/(2\alpha\rho_0)$ and the
flow is at rest at $t = 0$, so with the same well-prepared $\psi$ of
Section~\ref{sec:initialisation} in both runs the ratio is
$\alpha_2/\alpha_1$ by construction. 

\begin{center}
\begin{tabular}{lccc}
\toprule
Quantity & $\alpha_1$ & $\alpha_2$ & rel.\ difference \\
\midrule
plume depth                        & $0.250$ & $0.250$ & $0$ \\
surface cooling $W_{\mathrm{patch}}$
  & $2.4844\times10^{6}$ & $2.4844\times10^{6}$ & $0$ \\
$T'$ within the patch
  & $-1.77092\times10^{-2}$ & $-1.77080\times10^{-2}$ & $7\times10^{-5}$ \\
kinetic energy, $t_{\mathrm{end}}$
  & $7.5560\times10^{2}$ & $7.5556\times10^{2}$ & $5\times10^{-5}$ \\
$\max|w|$                          & $180.15$ & $180.21$ & $3.4\times10^{-4}$ \\
viscous dissipation $\varepsilon_{\mathrm{visc}}$
  & $2.9448\times10^{5}$ & $2.9457\times10^{5}$ & $3.1\times10^{-4}$ \\
potential energy, $t_{\mathrm{end}}$
  & $5.5005\times10^{4}$ & $5.5052\times10^{4}$ & $8.5\times10^{-4}$ \\
relative mass drift
  & $-6.399\times10^{-4}$ & $-6.409\times10^{-4}$ & $1.6\times10^{-3}$ \\
\bottomrule
\end{tabular}
\end{center}

\noindent
The surface cooling and the diffusive potential-energy source agree to
every printed digit, because  rhey never see $\alpha$.
Everythingwhat follows agrees to better than one part in a
thousand, while the error measures of the first table differ by a
factor nine.  The relative mass drift is the same $-6.4\times10^{-4}$
in both, therefore not caused by the artificial compressibility
at all.  The two tables together say not only that the runs agree,
they say how they differ, and that the difference is the
artificial compressibility.

Three diagnostics differ by more than the flow does, and they are the
price of the smaller $\alpha$:

\begin{center}
\begin{tabular}{lccc}
\toprule
Quantity & $\alpha_1$ & $\alpha_2$ & rel.\ difference \\
\midrule
numerical diffusivity $K_{\mathrm{num}}$
  & $8.959\times10^{-2}$ & $8.935\times10^{-2}$ & $0.27\%$ \\
variance residual, divergence term removed
  & $1.379\times10^{-5}$ & $1.366\times10^{-5}$ & $0.90\%$ \\
energy budget residual
  & $2.736\times10^{3}$  & $2.109\times10^{3}$  & $30\%$ \\
\bottomrule
\end{tabular}
\end{center}

\noindent
All three measure numerical error in the tracer and energy
budgets; all three are worse
at the smaller $\alpha$; and a single mechanism accounts for them,
since the $\alpha_1$ run advects with a velocity field carrying nine
times the divergence while both budgets are derived under
$\Div\bv = 0$.  The energy residual is the clearest case, collecting as
it does any operator pair that fails to be a discrete adjoint.  The absolute
magnitudes remain small: the numerical diffusivity stays within one
part in four hundred of $0.0894$ against a physical diffusivity of
unity, and both runs report the energy residual as $0.1\%$ of the rate
scale.  This is one experiment at one resolution, so the $1/\alpha$ law
is here measured, as well as the  residual.



\begin{center}
\begin{tabular}{cccccc}
\toprule
$c_{\mathrm{ac}}\Delta t/L$ & $\alpha$ & $\Fr_{\mathrm{baro}}^{2}$
  & err $\max|w|$ & err $T'$ lid & pressure \\
\midrule
$1$    & $3.13\times10^{7}$ & $0.29\%$ & $0.115\%$ & $0.005\%$ & $47.9$\,ms \\
$0.7$  & $1.53\times10^{7}$ & $0.59\%$ & $0.186\%$ & $0.009\%$ & $39.9$\,ms \\
$0.5$  & $7.83\times10^{6}$ & $1.15\%$ & $0.323\%$ & $0.016\%$ & $35.2$\,ms \\
$0.35$ & $3.84\times10^{6}$ & $2.30\%$ & $0.606\%$ & $0.030\%$ & $30.7$\,ms \\
$0.25$ & $1.96\times10^{6}$ & $4.35\%$ & $1.056\%$ & $0.054\%$ & $26.7$\,ms \\
\midrule
projection & --- & --- & --- & --- & $8.8$\,ms \\
\bottomrule
\end{tabular}
\end{center}

\noindent
Two features matter.  Every error follows
$1/\alpha$ to within $7\%$ across the whole range: there is no threshold
and no onset, and nothing became unstable at any value, including the
two whose departure exceeds that at which an earlier configuration of
this experiment was recorded as failing.  The errors are ordered by scale, the domain-wide surface signal moves twenty times less
than the plume's peak velocity.  Whatever sets a practical lower bound
on $c_{\mathrm{ac}}\Delta t/L$, it is therefore not that a domain-scale
pressure fails to form, because the domain scale is the part that
survives best.

The two endpoints were then repeated on the paper mesh, to $t=0.013$:

\begin{center}
\begin{tabular}{cccccc}
\toprule
$c_{\mathrm{ac}}\Delta t/L$ & $\alpha$ & $\Fr_{\mathrm{baro}}^{2}$
  & err $\max|w|$ & err $T'$ lid & pressure \\
\midrule
$1$    & $1.57\times10^{7}$ & $0.60\%$ & $0.121\%$ & $0.009\%$ & $348.9$\,ms \\
$0.35$ & $1.92\times10^{6}$ & $4.63\%$ & $0.931\%$ & $0.046\%$ & $238.6$\,ms \\
\midrule
projection & --- & --- & --- & --- & $247.0$\,ms \\
\bottomrule
\end{tabular}
\end{center}

\noindent

\subsection{Porting an Existing Hydrostatic Model}\label{sec:recipe}

The machinery AC/DC needs is machinery hydrostatic models already
have, lux-form transport, a tridiagonal column solver, a barotropic
solve, and the substantial work lies in the non-hydrostatic dynamics
.  The host model must supply
flux-form tracer transport consistent with continuity, a column-wise
tridiagonal solver, a barotropic solver of either kind, and a $z$- or
$z^*$-coordinate prismatic mesh; isopycnal and terrain-following
coordinates fall outside the analysis of
Section~\ref{sec:free_surface}.

\paragraph*{The changes, in dependency order.}
\begin{enumerate}[nosep]
\item[C1.] \emph{Prognostic vertical velocity.}  The largest single
  item, and the only one not specific to AC/DC: $w$ acquires its own
  momentum equation, with vertical advection, the full
  three-dimensional vorticity in the Lamb term and the non-hydrostatic
  vertical pressure gradient.  Hosts making the traditional
  approximation should revisit the Coriolis terms coupling $u$ and $w$.
\item[C2.] \emph{Pseudo-pressure.}  Add $\psi$ as a prognostic cell
  field advanced by~\eqref{eq:ac_psi}, with
  $\alpha = \rho_0 gH_{\mathrm{full}}$ from~\eqref{eq:alpha_choice} and
  $\psi = 0$ at the free surface,~\eqref{eq:psi_bc}.
\item[C3.] \emph{Pseudo-density and tracer transport.}  With
  ${\rhoa} = \rho_0(1+\psi/\alpha)$, the cell mass carried by
  the existing scheme becomes $|c|\,\rho_{\mathrm{AC},c}$ and the
  transporting flux becomes the pseudo-mass flux $F_f = r_fu_f|f|$
  of~\eqref{eq:Ff_def}, with the same face reconstruction and no change
  of limiter or tracer time step.  This is the one point at which a
  wrong reading silently gives a different scheme: transporting against
  the volumetric flux forfeits
  Propositions~\ref{prop:tracer_consistency} and~\ref{prop:variance},
  which is what test~V1 detects.
\item[C4.] \emph{Column solve.}  Apply the existing tridiagonal solver
  to $L_z(p_V/\rho_0) = S$ per column, Dirichlet at the surface and
  Neumann at the bottom, in the vertical pass that already runs.
\item[C5.] \emph{Barotropic coupling.}  Split $\psi = \bar\psi + \psi'$
  and carry $\bar\psi$ in the existing barotropic solver as
  Remark~\ref{rem:host_variants} describes for each host type, updating
  $\psi'$ once per baroclinic step.  Weighting $\bar\psi$ to vanish at
  the surface makes the surface condition exact across the loop.
\item[C6.] \emph{Diagnostics.}  Report the density-weighted kinetic
  energy~\eqref{eq:dw_energy} and the potential
  energy~\eqref{eq:pe_def}.
\end{enumerate}
Unchanged: the equation of state, the hydrostatic pressure, both time
steps, the free-surface solver, the tracer advection scheme and the
vertical mixing.

\paragraph*{Verification refinement sequence.}
Each step tests one proposition and should pass before the next is
attempted; the first three need no dynamics.
\begin{enumerate}[nosep]
\item[V1.] Constant tracer, arbitrary flow: $C$ constant to round-off
  (Proposition~\ref{prop:tracer_consistency}(ii)).  Failure indicates
  C3 fluxes not matching the continuity fluxes, on a sub-stepped modelt,
  most often the outer-step fluxes used in place of the time-integrated
  sub-step fluxes.
\item[V2.] Non-uniform passive tracer:
  $\sum_c|c|\rho_{\mathrm{AC},c}C_c$ constant to round-off for any
  limiter (Proposition~\ref{prop:tracer_consistency}(i)); and
  ${\rhoa} > 0$ throughout, as
  Proposition~\ref{prop:tracer_consistency}(iv) requires.
\item[V3.] Inviscid, unstratified, closed domain: $E$
  of~\eqref{eq:total_energy_dw} conserved to time-truncation while
  $K_{\mathrm{df}}$ plus the same elastic term drifts.  The drift in the
  density-free quantity should \emph{not} vanish under time-step
  refinement, while the drift in $E$ should.
\item[V4.] Linear internal wave: dispersion against
  Proposition~\ref{prop:hybrid_dispersion}; and two runs a decade apart
  in $\alpha$, confirming the $\OO(1/\alpha)$ approach to the
  incompressible solution and no $\alpha$-dependent time-step
  restriction.
\end{enumerate}

Hosts on an arbitrary-Lagrangian--Eulerian vertical grid need one
addition: the remapping step must be applied to ${\rhoa}$
consistently with the tracers, since remapping that conserves
$|c|\,C_c$ but not $|c|\,\rho_{\mathrm{AC},c}C_c$ violates Proposition~\ref{prop:tracer_consistency} at every remap.  


\end{document}